\documentclass[10pt]{amsart}

\usepackage[margin=3cm]{geometry}

\usepackage{latexsym}
\usepackage{amssymb}
\usepackage{graphicx}
\usepackage{xcolor}

\usepackage{hyperref}
\hypersetup{pdfborder = {0 0 0}}
\newtheorem{theorem}{Theorem}
\newtheorem{THM}{Theorem}
\newtheorem{Lemma}[THM]{Lemma}
\newtheorem{lemma}[THM]{Lemma}

\newtheorem{definition}[THM]{Definition}
\newtheorem{proposition}[THM]{Proposition}
\newtheorem{notation}[THM]{Notation}
\newtheorem{assumption}[THM]{Assumption}

\theoremstyle{remark}
\newtheorem{remark}[theorem]{Remark}

\newcommand{\bR}{\mathbb{R}}

\newcommand{\be}{\begin{equation}}
\newcommand{\ee}{\end{equation}}
\newcommand{\bS}{\mathbb{S}}
\newcommand{\bH}{\mathbb{H}}

\newcommand{\lp}{\left(}
\newcommand{\rp}{\right)}
\newcommand{\lb}{\left[}
\newcommand{\rb}{\right]}
\newcommand{\lc}{\left\{}
\newcommand{\rc}{\right\}}
\newcommand{\lab}{\left|}
\newcommand{\rab}{\right|}
\newcommand{\Lap}{\Delta}

\DeclareMathOperator{\dist}{dist}

\newcommand{\sH}{\mathcal{H}}

\newcommand{\E}{\mathbb{E}}
\newcommand{\Prob}{\mathbb{P}}

\newcommand{\Ind}{\boldsymbol{1}}

\newcommand{\eps}{\varepsilon}

\DeclareMathOperator{\spn}{span}
\DeclareMathOperator{\sech}{sech}

\DeclareMathOperator{\Arctanh}{Arctanh}

\newcommand{\TVd}{\mathrm{dist}_{\mathrm{TV}}}
\newcommand{\Law}{\mathcal{L}}

\DeclareMathOperator{\sR}{sR}
\DeclareMathOperator{\Rot}{Rot}

\begin{document}

\title[Non-Markov reflection couplings]{Non-Markovian maximal couplings and a vertical reflection principle on a class of sub-Riemannian manifolds}
\author{Liangbing Luo \and Robert W.\ Neel}
\address{Department of Mathematics, Lehigh University, Bethlehem, Pennsylvania, USA; current address: Department of Mathematics and Statistics, Queen's University, Kingston, Ontario, Canada}
\email{liangbing.luo@queensu.ca}
\address{Department of Mathematics, Lehigh University, Bethlehem, Pennsylvania, USA}
\email{robert.neel@lehigh.edu}
\subjclass[2010]{Primary 58J65; Secondary 53C17, 58J35, 60J45}
\keywords{maximal coupling, sub-Riemannian manifold, SU(2), SL(2,R), non-isotropic Heisenberg group, total variation distance, heat semigroup estimates}

\begin{abstract} We develop an approach to constructing non-Markovian, non-co-adapted couplings for sub-Riemannian Brownian motions in sub-Riemannian manifolds with large symmetry groups by treating the specific cases of the three-dimensional Heisenberg group, higher-dimensional non-isotropic Heisenberg groups, $\operatorname{SL}(2,\bR)$ and its universal cover, and $\operatorname{SU}(2)$. Our primary focus is on the situation when the processes start from two points on the same vertical fiber, since in general Markovian or co-adapted couplings cannot give the sharp rate for the coupling time in this case. Non-Markovian couplings of this type on sub-Riemannian manifolds were first introduced by Banerjee-Gordina-Mariano, for the three-dimensional Heisenberg group, and were more recently extended by B\'en\'efice to $\operatorname{SL}(2,\bR)$ and $\operatorname{SU}(2)$, using a detailed consideration of the Brownian bridge. In contrast, our couplings are based on global isometries of the space, giving couplings that are maximal, as well as making the construction relatively simple and uniform across different manifolds. The coupled processes satisfy a reflection principle with respect to their coupling time, so that the coupling time reduces to the hitting time for one component of the Brownian motion, which is useful in explicitly bounding the tail probability of the coupling time. Further, it's natural to use this coupling as the second stage of a two-stage coupling when considering points on different vertical fibers. We estimate the coupling time in these various situations and give applications to inequalities for the heat semigroup.
\end{abstract}

\maketitle

\section{Introduction}

\subsection{Background and motivation}\label{Sect:Intro1}
Couplings of Brownian motions on sub-Riemannian manifolds have attracted interest in recent years (here we consider sub-Riemannian manifolds with a canonical sub-Laplacian, and we call the associated diffusion Brownian motion, following the convention in, for example, \cite{MagalieAdapted}). More concretely, we are interested in couplings in which the particles meet (in finite time), after which, because they are strong Markov processes, they can be taken to coincide (couplings in which the particles do not meet but instead stay within a controlled distance of each other almost surely, generally called parallel or synchronous couplings in the literature, are also of interest, but we don't pursue them here). Let the coupling time of two Brownian motions $B_t$ and $\tilde{B}_t$ be denoted $\tau$. If $\tau$ is a.s.\ finite, the coupling is said to be successful. The existence of a successful coupling of Brownian motions from any points on a sub-Riemannian (or Riemannian) manifold already has analytic consequences; it implies the weak Liouville property, which says that any bounded harmonic function (harmonic with respect to the same sub-Laplacian associated to the diffusion) must be constant.

Going beyond such a qualitative property, a central motivation for study such couplings is the Aldous inequality relating the coupling time to the total variation distance between the laws of the processes. That is, if $\Law(B_t)$ denotes the distribution (or law) of $B_t$   and $\TVd$ is the total variation distance, we have
\begin{equation}\label{Eqn:Aldous}
\Prob\lp \tau > t\rp \geq  \TVd\lp \Law\lp B_t\rp,\Law( \tilde{B}_t )  \rp.
\end{equation}
From here, there are also connections to various functional inequalities, such as gradient estimates for the heat semigroup, spectral gap estimates, and Poincar\'e and Sobolev inequalities. A coupling is called maximal if it makes the Aldous inequality an equality for all $t>0$. More generally, a coupling is called efficient if the ratio $\frac{\Prob\lp \tau > t\rp}{ \TVd\lp \Law\lp B_t\rp,\Law( \tilde{B}_t)  \rp}$ stays bounded as $t\rightarrow\infty$, which essentially means it achieves the optimal rate of decay for the coupling time (see \cite{BK-Efficiency} for a more precise definition and further details).

Co-adapted (or even Markov) couplings are generally easier to construct than non-co-adapted couplings, and have been developed for a variety of sub-Riemannian structures over $\bR^2$; see \cite{KB-Immersion-2018} and the references therein. Moreover, a version of these ideas has recently been developed for $\operatorname{SU}(2)$ and $\operatorname{SL}(2,\bR)$, which we henceforth abbreviate to $\operatorname{SL}(2)$ (see below for the definition of these spaces and their sub-Riemannian structures), by B\'en\'efice \cite{MagalieAdapted}. However, for sub-Riemannian manifolds, one cannot expect Markovian or co-adapted couplings to be efficient. Indeed, \cite{BGM} (see their Remark 3.2) shows that for the (3-dimensional) Heisenberg group $\bH$, which is generally considered the most elementary example of a sub-Riemannian manifold, the best Markovian coupling for Brownian motions started from points on the same vertical fiber (with respect to the natural submersion structure) has coupling rate $1/\sqrt{t}$ while the sharp rate is $1/t$. Motivated by this, they construct a non-Markovian coupling of two Heisenberg Brownian motions starting on the same vertical fiber that achieves this sharp rate, and is therefore efficient in the sense described above. The idea is as follows. The Heisenberg group can be understood as the flat plane $\bR^2$ along with the associated (signed) area. Indeed, if we let $(x,y,z)$ be the usual (Cartesian) coordinates on $\bH$ (we give details below), then letting $(x_t,y_t)$ be an $\bR^2$-Brownian motion started from $(x_0,y_0)$ and
\[
z_t = z_0 +\frac{1}{2}\int_0^t x_s \,dy_s- \frac{1}{2}\int_0^t y_s\, dx_s
\]
be the associated L\'evy area, started from $z_0$, gives Heisenberg Brownian motion. Projecting onto $(x,y)$ gives a submersion onto $\bR^2$, and the vertical fiber over a point $(x_0,y_0)$ is all points of the form $(x_0,y_0,z)$ for $z\in \bR$. By left-invariance, two points on the same fiber can be taken to be $(0,0,0)$ and $(0,0,2a)$ for some $a>0$. The idea of \cite{BGM} is to allow the $\bR^2$-Brownian motions to separate and then come back together at a future time in such a way that the difference in their L\'evy areas changes, increasing or decreasing by some controlled amount. Such a construction is necessarily non-Markov, and their approach is based on detailed consideration of planar Brownian paths, making significant use of the Karhunen-Lo\`eve expansion and properties of the Brownian bridge. They then iterate this procedure and show that a.s.\ the $z$-coordinates will coincide after one of these steps, at which point the processes have coupled successfully. More generally, if the Brownian motions start from $(x_0,y_0,z_0)$ and $(\tilde{x}_0,\tilde{y}_0,\tilde{z}_0)$ with $(x_0,y_0)\neq (\tilde{x}_0,\tilde{y}_0)$, they consider a two-stage coupling. First, the $\bR^2$-Brownian motions are run under a standard mirror coupling, which will be successful. Once the $(x,y)$ processes are coupled (at some random time and random associated $z$-values), the non-Markovian coupling just described is used, which then results in the full processes being coupled.

Recently, B\'en\'efice \cite{Benefice2023b} (see also the announcement in \cite{MagalieEtAl}) developed a version of this non-Markovian vertical coupling for $\operatorname{SU}(2)$ and $\operatorname{SL}(2)$, which admit a similar submersion structure, although with a compact vertical fiber (diffeomorphic to $\bS^1$). This construction again makes use of a detailed analysis of the corresponding bridge processes, and also makes fundamental use of the compactness of the vertical fibers. Indeed, Remark 1 of \cite{Benefice2023b} explicitly notes that applying this version of the vertical coupling to the Heisenberg group gives a different coupling than that of \cite{BGM}, and it is unclear that it is successful.

Our primary goal in the present paper is to develop an alternative construction to the vertical couplings just described which is more direct (it is given by a single global isometry), which generalizes naturally to a variety of sub-Riemannian manifolds sharing a similar submersion structure, and moreover, which is maximal, rather than merely efficient.  We can quickly describe it, in the case of the Heisenberg group, for simplicity, as follows. Let $(x_t,y_t,z_t)$ be a Heisenberg Brownian motion started from the origin $(0,0,0)$. For some $a\geq 0$, let $\sigma_a$ be the first time $z_t$ hits the level $a$.  For any $(x,y)\neq (0,0)$, let $R_{(x,y)}:\bR^2\rightarrow\bR^2$ be the map that reflects $\bR^2$ over the line containing $(0,0)$ and $(x,y)$. Recalling that $\bR^2$-Brownian motion is preserved under any isometry of $\bR^2$ (up to the initial point possibly moving), we see that
\begin{equation}\label{Eqn:HFlip}
\lp \tilde{x}_t, \tilde{y}_t\rp = \begin{cases}
R_{\lp x_{\sigma_a},y_{\sigma_a}\rp}\lp x_t,y_t\rp & \text{for $t\leq \sigma_a$} \\
\lp x_t,y_t\rp  & \text{for $t> \sigma_a$}
\end{cases}
\end{equation}
is also an $\bR^2$-Brownian motion from the origin. (Observe though that the construction is not Markov.) Further, if $\tilde{z}_t$ is the L\'evy area associated to $\lp \tilde{x}_t, \tilde{y}_t\rp$ and started from $\tilde{z}_0=2a$, we see, using that reflection in $\bR^2$ reverses the sign of the area enclosed by a curve, that
\begin{equation}\label{Eqn:VFlip}
\tilde{z}_t = \begin{cases}
2a-z_t & \text{for $t\leq \sigma_a$} \\
z_t  & \text{for $t> \sigma_a$}
\end{cases} .
\end{equation}
Thus we have two Heisenberg Brownian motions, started from $(0,0,0)$ and $(0,0,2a)$, that couple at time $\sigma_a$. Not only does this reduce the study of the coupling time to the study of a hitting time for one process, but this hitting time satisfies a reflection principle exactly analogous to the classical reflection principle for Brownian motion on the real line, in spite of the fact that the $z_t$ is not a Markov process. Using this, one easily shows that the coupling is maximal. We fill in some details below, but this idea forms the core of the present paper. (As already suggested above, there are several related notions used in the literature regarding the adaptedness of the coupling. The notion of a Markovian-coupling used in \cite{Hsu-Sturm}, closely related to a co-adapted or immersion coupling, does not necessarily imply that the joint process $\lp B_t, \tilde{B}_t\rp$ is Markov. Also, the couplings in \cite{Kendall-H} are taken to be co-adapted, rather than Markov, though primarily to simplify the construction. We haven't emphasized these distinctions because our vertical reflection couplings are not Markovian or co-adapted in any of these senses, so we simply refer to them as ``non-Markovian'' without further elaboration.)

Several other standard sub-Riemannian manifolds admit a similar submersion structure, and this construction generalizes directly to them (although with a slight modification for the circular fiber in the cases of $\operatorname{SU}(2)$ and $\operatorname{SL}(2)$).

Our explicit construction of maximal couplings using global reflections is similar to earlier work of Kuwada. In \cite{Kuwada-Suf} and \cite{Kuwada-Nec}, Kuwada showed that a \emph{reflection structure} (see \cite[(A1) and (A2)]{Kuwada-Suf} or \cite[Definition 1.2]{Kuwada-Nec}) was a sufficient condition for the existence and uniqueness of a maximal Markovian coupling on a Riemannian manifold and a necessary condition for the existence on a Riemannian homogeneous space. Such a maximal Markovian coupling was proved to be a \emph{mirror coupling} in the sense of Kuwada (see \cite[p. 635]{Kuwada-Nec} for precise definition). This mirror coupling is not known to be the same as the reflection coupling in the sense of Kendall-Cranston in general, but with the existence of a reflection structure on the Riemannian manifold, they coincide; see \cite[Theorem 5.1]{Kuwada-Suf}. (See also \cite{KB-Markovian-2017} for the possibility of adding a drift to the diffusion.) Our construction is obviously conceptually similar, in that we rely on the fact that our sub-Riemannian model spaces have the largest possible isometry group given their growth vector, which is related to the fact that they can be realized as bundles over Riemannian homogeneous spaces. The difference is illustrated by comparing the above description of our vertical coupling on the Heisenberg group with the reflection coupling on $\bR^2$, with the standard Euclidean metric, of course. In that case, if the initial distance between the points is $2a$, for some $a>0$, without loss of generality, we can take the points as $(0,0)$ and $(2a,0)$. If $B_t$ is a Brownian motion started from $(0,0)$, then we reflect it over the line $\{x=a\}$ at the first time $B_t$ hits this line. This gives a Brownian motion $\tilde{B}_t$ starting from $(2a,0)$ that couples with $B_t$ at the first hitting time of $\{x=a\}$. Moreover, this coupling is maximal, and it is also Markov (perhaps in spite of this description). The point is that here the reflection is non-random and known from the start, so one knows the evolution of $\tilde{B}_t$ in an adapted way, even before hitting the line. Indeed, this reflection coupling agrees with the Kendall-Cranston reflection coupling, given via an SDE on $\bR^2\times\bR^2$. In contrast, the vertical reflection coupling based on \eqref{Eqn:HFlip}, which we study in this paper, is given by a reflection over a random line in $\bR^2$, which is not known until $\sigma_a$, and this is responsible for the non-Markovian nature of the coupling. Note that it is known that maximal couplings exist in considerable generality (apparently going back to \cite{SAndS-Maximal} in the case of continuous-time processes), and even in the case of Brownian motions on $\bR$, need not be unique (see the simple example given in \cite{Hsu-Sturm} and attributed there to Fitzsimmons). In the work of Kuwada just discussed, a unique maximal coupling was selected by requiring it be Markov. While our vertical couplings are not characterized in this way, we nonetheless emphasize that our vertical reflection coupling shares the same structural properties as the more usual (and Markov) reflection coupling on Riemannian manifolds with global symmetries, namely, it is explicitly constructed via a global isometry, it is maximal, and it reduces the determination of the coupling time to the determination of a hitting time for a single process.

We note that the couplings described above, both the co-adapted and non-Markovian ones, rely on the Brownian motion in question consisting of a well-understood Riemannian Brownian motion augmented with an additional functional (or several). Couplings that more directly generalize the Kendall-Cranston mirror coupling on a Riemannian manifold (see Sections 6.5-6.7 of \cite{EltonBook}), and are compatible with comparison geometry and curvature bounds, have yet to be developed for sub-Riemannian manifolds, despite some groundwork having been laid in \cite{OurRadial} and \cite{Preprint-FEA}. Of course, gradient estimates for heat semigroups have also been studied by other methods, for example, generalized $\Gamma$-calculus, see \cite{Bonnefont-SL,FabriceMichel,BG-Gamma}, or Malliavin calculus, see \cite{AntonMarc}. Moreover, a comparison between $\Gamma$-calculus and parallel couplings for Kolmorogov-type diffusions (which are close cousins to sub-Riemannian diffusions) has been undertaken in \cite{Kolmogorov-Coupling}.

After this paper was submitted, an extension of the approach of \cite{Benefice2023b} to free, step 2 Carnot groups was given by B\'en\'efice \cite{Benefice-New}. In this case, the vertical fiber is multi-dimensional. At the same time, ``drift couplings'' for free, step 2 Carnot groups, in which a Brownian motion is coupled with a Brownian motion with a prescribed drift, with the drift chosen to force the particles to couple by a fixed (finite) time $T$, were introduced  in \cite{Benefice-Drift}. The idea here is that the laws of two Brownian motions at time $T$ can be compared by using the Girsanov theorem to see how the addition of the drift alters the law of the second process. An artful choice of the drift then allows one to study various functional inequalities.

\subsection{Summary of results}\label{Sect:ResSum}

The key to the vertical coupling as described in Equations \eqref{Eqn:HFlip} and \eqref{Eqn:VFlip} is the submersion onto a Riemannian manifold and the large symmetry group of that manifold. This structure is shared by several standard sub-Riemannian manifolds. Here, we explicitly consider the 3-dimensional Heisenberg group, $\operatorname{SL}(2)$ and its universal cover $\widetilde{\operatorname{SL}(2)}$, $\operatorname{SU}(2)$, and non-isotropic Heisenberg groups of any dimension. All of these spaces admit a submersion onto a Riemannian model space with a one-dimensional vertical fiber. We give precise definitions of these spaces (along with their sub-Riemannian structures) and the specific form the vertical coupling takes below. For now, we give a general statement that incorporates all of these cases.

\begin{THM}\label{THM:SummaryCoupling}
Let $M$ be the Heisenberg group (of any dimension, and possibly non-isotropic), $\operatorname{SL}(2)$ or its universal cover, or $\operatorname{SU}(2)$, and let $q$ and $\tilde{q}$ be two points on the same vertical fiber with vertical displacement $2a$ for some $a>0$. Then there exists a hypersurface $S$ in $M$ (not necessarily connected) such that, if $B_t$ is an $M$-Brownian motion from $q$ and $\sigma_a$ is the first hitting time of $S$ for $B_t$, then there exists an isometry of $M$, $I_{B_{\sigma_a}}$, depending on $B_{\sigma_a}$, such that
\[
\tilde{B}_t = \begin{cases}
I_{B_{\sigma_a}}(B_t) & \text{for $t\leq \sigma_a$} \\
B_t & \text{for $t> \sigma_a$}
\end{cases}
\]
is an $M$-Brownian motion from $\tilde{q}$ that couples with $B_t$ at time $\sigma_a$. Moreover, this coupling is successful and maximal, the distribution of $\sigma_a$ depends only on $a$ (for a given choice of $M$), and this coupling satisfies a reflection principle. Namely, there exists an open subset $S^+$ of $M$ (which has $S$ as its boundary) such that
\[
\Prob\lp \sigma_a> t\rp = 1-2 \Prob\lp B_t\in S^+ \rp .
\]
\end{THM}

It will be clear that the natural variant of this coupling could be developed on some other spaces; for example, one could take the quotient of $\bH$ under translation by some $\gamma>0$ in the vertical fiber. However, the natural extent of this approach is unclear, in part because sub-Riemannian structures exhibit considerable variety. In all examples considered here, the horizontal distribution has co-dimension 1, and it's not obvious whether or not this approach to constructing couplings can apply to cases with a more complicated growth vector (such as higher-step structure). Thus, a more comprehensive result (similar to those of Kuwada mentioned above) describing which sub-Riemannian manifolds admit explicit (if non-Markovian) maximal couplings, given by global isometries, would be desirable. Going further in this vein, it is natural to ask whether or not one can construct a maximal coupling for points not on the same vertical fiber, even for the Heisenberg group.

Once we have a vertical coupling of this type, we obviously wish to understand $\sigma_a$ well enough to compute or bound the coupling time. Further, one can then use this vertical coupling as the second stage in a two-stage coupling, as described above. In the case of the three-dimensional Heisenberg group, this follows the lines of \cite{BGM} and gives essentially the same results, just with a simplified proof and better control of the constants, coming from the fact that we use a simpler and maximal vertical coupling. Similarly, in the cases of $\operatorname{SL}(2)$ and $\operatorname{SU}(2)$, the application of the coupling follows the lines of \cite{Benefice2023b}, again with simpler proofs. The cases of $\widetilde{\operatorname{SL}(2)}$ and higher-dimensional, non-isotropic Heisenberg groups, on the other hand, have not been previously treated via couplings.

Concretely, our ability to understand $\sigma_a$ is strongly influenced by the structure of the vertical fiber. In the cases of the Heisenberg groups and $\widetilde{\operatorname{SL}(2)}$, the fiber is a line, which makes the distribution of $\sigma_a$ amenable to study via the reflection principle. In the cases of $\operatorname{SL}(2)$ and $\operatorname{SU}(2)$, the vertical fiber is a circle, which makes determining the distribution of $\sigma_a$ more difficult (see Remark \ref{Rmk:Circle}). For now, we settle for a general exponential decay in these two cases, based on the compactness of the fiber, just as in \cite{Benefice2023b}.

\begin{THM}\label{THM:SummaryVert}
Let everything be as in Theorem \ref{THM:SummaryCoupling}, and let $\Law\lp B_t\rp$ and $\Law\lp \tilde{B}_t\rp$ be the laws of Brownian motions from $q$ and $\tilde{q}$, respectively. Further, let $P_t$ be the heat semigroup on $M$ and $Z$ be the natural vertical vector field tangent to the vertical fiber (see below for the precise description of $Z$ for each space $M$). Then depending on $M$, we have the following
\begin{itemize}
\item If $M$ is one of the Heisenberg groups, then the coupling time satisfies the estimate
\[
\Prob\lp \sigma_a> t\rp =\TVd \lp \Law\lp B_t\rp, \Law\lp \tilde{B}_t\rp \rp \leq \frac{2a}{\alpha_n t}
\]
where $\alpha_n>0$ is an invariant of the structure (see Definition \ref{defn.NonisotropicHeisenbergGroup}), which is equal to 1 for the 3-dimensional Heisenberg group. This implies the vertical gradient estimate
\[
\lab Z P_t f  \rab = \lab \nabla_{V} P_t f \rab \leq \frac{1}{\alpha_n t}\|f\|_{\infty} \quad\text{for any $f\in L^{\infty}$.}
\]
\item If $M=\widetilde{\operatorname{SL}(2)}$, then, for some positive constants $c$ and $T_0$, we have
\[
\Prob\lp \sigma_a> t\rp  =\TVd \lp \Law\lp B_t\rp, \Law\lp \tilde{B}_t\rp \rp \leq c\frac{2a}{\sqrt{t}} \quad \text{for all $t>T_0$,}
\]
which implies the vertical gradient estimate
\[
\lab Z P_t f \rab = \lab \nabla_{V} P_t f \rab \leq \frac{c}{\sqrt{t}}\|f\|_{\infty} \quad \text{for all $t>T_0$.}
\]
\item If $M$ is $\operatorname{SL}(2)$ or $\operatorname{SU}(2)$, then, for some positive constants $C$, $c$, and $T_0$, which do not depend on $a$, such that
\[
\Prob\lp \sigma_a> t\rp =\TVd \lp \Law\lp B_t\rp, \Law\lp \tilde{B}_t\rp \rp \leq Ce^{-ct} \quad \text{for all $t>T_0$.}
\]
\end{itemize}
\end{THM}

For the Heisenbergs groups and $\widetilde{\operatorname{SL}(2)}$, these results are optimal. For the 3-dimensional Heisenberg group, while the coupling of \cite{BGM} already gives the correct asymptotic decay rate of $\frac{c}{\sqrt{t}}$, in our case, the bound above follows from an exact expression in the 3-dimensional Heisenberg group,
\[
\Prob\lp \sigma_a> t\rp = \frac{4}{\pi}\arctan\lp\tanh\lp\frac{\pi}{2}\cdot\frac{a}{t}\rp\rp,
\]
which we derive using the reflection principle; see \eqref{Eqn:HExact} and Theorem \ref{THM:HLimit}. We also note that vertical gradient bounds for the heat semigroup generally give better control over the change in $P_t f$ along vertical fibers than horizontal gradient bounds do. In particular, considering the 3-dimensional Heisenberg group, integrating along the vertical fiber gives
\[
\lab P_tf(0,0,0)-P_tf(0,0,z)\rab \leq \frac{z}{t}\|f\|_{\infty} .
\]
On the other hand, for small $z$, the sub-Riemannian distance between $(0,0,0)$ and $(0,0,z)$ is comparable to $\sqrt{z}$, so integrating the horizontal gradient bound from Theorem \ref{THM:HorSum} below, again for the 3-dimensional Heisenberg group, along a minimal geodesic gives $\lab P_tf(0,0,0)-P_tf(0,0,z)\rab \leq C\frac{\sqrt{z}}{\sqrt{t}}\|f\|_{\infty}$ for some constant $C>0$. And this is weaker for small $z$ or large $t$ than what we just got from the vertical gradient bound. (Indeed, because $Z=XY-YX$, a vertical gradient is in a sense more akin to a second-order gradient.)

For the other two cases, $\operatorname{SL}(2)$ and $\operatorname{SU}(2)$, obviously, one would like to gain sufficient understanding of $B_t$ and $\sigma_a$ to compute the sharp value of the exponent in the above asymptotic bound for $\Prob\lp \sigma_a> t\rp$. We don't consider the spatial dependence in these cases (meaning the dependence on $a$). But we conjecture that the sharp bound in time will also be linear in $a$, so that one obtains a vertical gradient estimate with exponential decay of the vertical gradient, analogous to the other cases. These are also the two cases considered in \cite{Benefice2023b}, where it is shown that the vertical couplings introduced there have a coupling time that decays like
$C_qe^{-ct} \dist_{\sR}(p,p')^{2q}$
for any $q\in (0,1)$, with $C_q$ a constant depending on $q$, $c$ an unknown positive constant, and where $d_{\sR}(p,p')$ is the sub-Riemannian distance between the starting point $p$ and $p'$, taken to be on the same vertical fiber. To get linear dependence in $a$ (in our notation) and thus a vertical gradient bound, one would need this with $q=1$.

For points not on the same vertical fiber, following \cite{BGM}, we use a two-stage coupling. First, the marginal processes given by the submersion onto a Riemannian model space (what we call the horizontal component) are coupled, using standard reflection couplings, with the vertical component ``coming along for the ride.'' Once the horizontal component has successfully coupled (if it does), the two processes are then on the same vertical fiber, with a random vertical displacement, at which point we can run the vertical coupling described in Theorem \ref{THM:SummaryCoupling}. Again, the nature of the results varies depending on $M$. For the 3-dimensional Heisenberg group, the relevant estimates on the horizontal coupling were already done in \cite{BGM}, and we recover their results with our vertical coupling in place of theirs. We also extend these results to any non-isotropic Heisenberg group, with the same sharp order of decay, which we summarize in a moment with the corresponding horizontal gradient bound. In the cases of $\operatorname{SL}(2)$ or $\widetilde{\operatorname{SL}(2)}$, the horizontal component is Brownian motion on the hyperbolic plane, which cannot be successfully coupled. In particular, there does not exist a successful coupling of Brownian motion on $\operatorname{SL}(2)$ or $\widetilde{\operatorname{SL}(2)}$ if the initial points are on different vertical fibers. Thus, we give an asymptotic result determining the limiting behavior of the heat semigroup (we note that the fact that harmonic functions on $\operatorname{SL}(2)$ are constant on fibers was also given in \cite{Benefice2023b}). Finally, for $\operatorname{SU}(2)$, just as for the vertical coupling, we quickly see, essentially by compactness, that there is a global exponential bound on the total variation distance from any two initial points, again as in \cite{Benefice2023b}. We summarize these results as follows; additional related results are found in the main body of the paper.

\begin{THM}\label{THM:HorSum}
For any of the choices of $M$ as in Theorem \ref{THM:SummaryCoupling}, let $\nabla_{\mathcal{H}}$ be the horizontal gradient induced by the sub-Riemannian structure and $P_t$ the heat semigroup associated to $M$-Brownian motion.
\begin{itemize} 
\item Let $M$ be one of the Heisenberg groups. Then there exists a positive constant $C$ such that, for any
 $f\in L^{\infty}$ and for any $t \geq 1$
\[
\Vert \nabla_{\mathcal{H}} P_t f \Vert_{\infty} \leqslant \frac{C}{\sqrt{t}}\Vert f\Vert_{\infty}.
\]
\item Let $M$ be $\operatorname{SL}(2)$ or $\widetilde{\operatorname{SL}(2)}$. Then for any two points $q$ and $\tilde{q}$ not on the same vertical fiber, let $2r$ be the distance (in the hyperbolic plane) between their submersions onto the hyperbolic plane. For $f\in L^{\infty}\lp \widetilde{\operatorname{SL}(2)} \rp$, we have
\[
\limsup_{t\to \infty} \lab P_t f (q) -P_t f\lp\tilde{q}\rp  \rab \leq \frac{4}{\pi}\arctan\lp \tanh \lp\frac{r}{2}\rp \rp \cdot\|f\|_{\infty}
\]
Finally, if $f$ is also harmonic, then it is constant on vertical fibers and is given by the lift of a bounded harmonic function on the hyperbolic plane.
\item Let $M=\operatorname{SU}(2)$. Then there exist constants $K>0$, $k>0$, and $T_0>0$ such that, for any two points $q$ and $\tilde{q}$ of $\operatorname{SU}(2)$, the total variation distance between the laws of Brownian motions $B_t$ and $\tilde{B}_t$ from $q$ and $\tilde{q}$ respectively, satisfies
\[
 \TVd\lp \Law\lp B_t\rp,\Law( \tilde{B}_t )  \rp < K e^{-kt} \quad \text{for all $t>T_0$.}
\]
\end{itemize}
\end{THM}

Here, it is the result for $\operatorname{SU}(2)$ which should be improved. In particular, after the vertical coupling is better understood, one wishes to understand the joint distribution of the horizontal coupling time and vertical displacement upon horizontal coupling, to get the exponential rate $K$ in the above. Moreover, it should be possible to control the dependence of the coupling time on the initial points to obtain the corresponding horizontal gradient estimate. Indeed, \cite{Benefice2023b} gives such an argument for the version of the coupling used there, showing that the horizontal gradient of the heat semigroup decays at rate $e^{-ct}$ for some non-explicit $c>0$. However, we defer this direction until an explicit rate for the vertical coupling can be understood, upon which the exponential rate for the two-stage coupling depends.

\subsection{Acknowledgements}
The second author thanks Erlend Grong for useful discussions about diffusions on Lie groups. The authors also thank the anonymous referees for many suggestions that improved the final version.

\section{The 3-dimensional Heisenberg group}

\subsection{Geometry and submersion structure of $\bH$}

Let $(x,y,z)$ be the standard (Cartesian) coordinates for the 3-dimensional Heisenberg group $\bH\cong \bR^3$. We recall that the sub-Riemannian structure is given by the orthonormal basis
\[
X = \partial_x-\frac{y}{2}\partial_z \quad\text{and}\quad Y=\partial_y+\frac{x}{2}\partial_z ,
\]
which determines both the horizontal distribution $\sH=\spn\{ X,Y \}$ and the inner product on it.
We also introduce the vector field $Z=\partial_z$, so that the Lie bracket structure is given by
\[
[X,Y]=Z \quad\text{and}\quad [X,Z]=[Y,Z]=0 .
\]
As implied by the name, $\bH$ is a Lie group, with the group law
\[
(x,y,z)\cdot(x^{\prime},y^{\prime},z^{\prime})=\lp x+x^{\prime},y+y^{\prime},z+z^{\prime}+\frac{1}{2}\lp xy^{\prime}-yx^{\prime} \rp\rp ,
\]
and multiplication on the left gives isometries of the sub-Riemannian structure (and note that $X$, $Y$ and $Z$ are left-invariant vector fields). We are primarily concerned with the Heisenberg Laplacian $\Delta_{\sH} =X^2+Y^2 $
and the diffusion generated by $\frac{1}{2}\Delta_{\sH}$, which we will call Heisenberg Brownian motion. We also have the horizontal gradient $\nabla_{\mathcal{H}}$, defined so that for any smooth $f$, $ \nabla_{\mathcal{H}}f\in \mathcal{H}$ is the dual to $df|_{\mathcal{H}}$ determined by the inner product on $\mathcal{H}$. In terms of the natural frame, we have $ \nabla_{\mathcal{H}}f=(Xf)X+(Yf)Y$. 

The map $\pi:\bH\rightarrow\bR^2$ given by $\pi(x,y,z)= (x,y)$ is a submersion, as elaborated on in \cite{Montgomery} (especially the first chapter). The description of the sub-Riemannian structure on $\bH$ in terms of the submersion onto $\bR$ and the associated horizontal lift extends to Brownian motion.  In particular, a Heisenberg Brownian motion started from $(x_0,y_0,z_0)$ projects, under $\pi$, to an $\bR^2$-Brownian motion started from $(x_0,y_0)$. Conversely, an $\bR^2$-Brownian motion started from $(x_0,y_0)$ has, for any $z_0\in\bR$, a unique horizontal lift which gives a Heisenberg Brownian motion started from $(x_0,y_0,z_0)$ (see \cite{EltonBook} for a discussion of the horizontal lift of a diffusion to a bundle). Further, as indicated above, the $z$-process of this lift is given by the L\'evy area associated to the $\bR^2$-Brownian motion.

The vertical lines of the form $\{(x_0,y_0,z): z\in \bR\}$ are obviously the fibers of the submersion. If $(x,y,z)$ and $(x,y,z')$ are two points on the same vertical fiber, then we call $|z'-z|$ their vertical displacement, and we observe that it is preserved under left-multiplication, so that it is a well-defined invariant of the sub-Riemannian structure. This also means that for any two such points, we can take Cartesian coordinates such that they become, up to exchanging them, $(0,0,0)$ and $(0,0,|z'-z|)$. Continuing, we call the $\bR^2$-distance between $\pi(x,y,z)$ and $\pi(x',y',z')$ the horizontal displacement between the points. Since rotation about the $z$-axis is also an isometry of the structure, this is also a well-defined invariant of the sub-Riemannian structure. Note that $\bH$ has the largest possible isometry group that respects the vertical fibration. This is compatible with the connection, noted in Section \ref{Sect:Intro1}, between maximal couplings and global symmetries of the space. We also note that $Z$ is not only tangent to the vertical fibers, but is the unique (or unique to up to sign, if we don't consider $\sH$ to be oriented by making $\{X,Y\}$ and oriented frame) vector field tangent to the fiber induced by the sub-Riemannian structure in the sense that $Z=[X,Y]$ (in particular, it is the Reeb vector field, though we don't specifically need that notion here). Since $Z$ gives a natural notion of ``vertical scale,'' we let $\nabla_{V}f = (Zf)f$ be the vertical gradient. Further, for the structures we consider, one often considers an associated Riemannian metric in which $Z$ is taken to complete the orthonormal frame (see, for example, the influential paper \cite{BG-Gamma}). For such a metric, $\|\nabla_{V}f\| = |Zf|$. Motivated by these considerations, we adopt the convention that $\|\nabla_{V}f\| = |Zf|$, although this is simply notation for us.

We finish this section by introducing cylindrical coordinates on $\bH$. Such coordinates consist in replacing $x$ and $y$ by $r$ and $\theta$ in the usual way (and leaving $z$ alone). Then the Heisenberg Laplacian becomes
\[
\Delta_{\sH}= \partial_r^2+\frac{1}{r}\partial_r + \frac{1}{r^2} \partial^2_{\theta} +\frac{1}{4}r^2\partial^2_z+ \partial_{\theta}\partial_z .
\]
It follows that Heisenberg Brownian motion is given, in cylindrical coordinates, by the following system of SDEs,
\[
dr_t = dW_t^{(1)} +\frac{1}{2r_t}\, dt, \quad
d\theta_t = \frac{1}{r_t}\, dW_t^{(2)}, \quad \text{and }
dz_t=  \frac{1}{2}r_t \, dW_t^{(2)}
\]
where $W_t^{(1)}$ and $W_t^{(2)}$ are independent one-dimensional Brownian motions. Further, we see that $z_t$ can be written as
\begin{equation}\label{Eqn:HeisTC}
z_t = z_0+W_{\int_0^t \frac{1}{4}r_s^2 \, ds}
\end{equation}
where $W_t$ is a one-dimensional Brownian motion independent of $r_t$ (and $r_t$ is a 2-dimensional Bessel process). The Heisenberg group is straightforward enough that it's not necessary to introduce cylindrical coordinates, but the direct analogue will be used when we study other spaces below.

\subsection{A maximal vertical coupling}\label{Sec.HeisVCoupling}
We have described the basic construction of the coupling in Section \ref{Sect:Intro1}. While it is clear that the construction given in \eqref{Eqn:HFlip} and \eqref{Eqn:VFlip} produces coupled processes $B_t$ and $\tilde{B}_t$ on $\bH$, with $B_t$ a Heisenberg Brownian motion and the coupling occurring at $\sigma_a$, it may not be completely obvious that $\tilde{B}_t$ is also a Heisenberg Brownian motion. So we begin with a lemma clarifying this point. Moreover, we prove it in the generality needed for the rest of the paper. In particular, we will see that $SL(2)$ has a similar representation as a submersion over the hyperbolic plane $\bH^2_+$ while $SU(2)$ can be viewed as a submersion over the sphere $\bS^2$ (that is, $\bS^2$ and $\bH^2_+$ are the 2-dimensional space forms of curvature $1$ and $-1$, respectively). Also note that $\sigma_a$ is a stopping time with respect to the filtration generated by $(x_t,y_t)$ (since $z_t$ is adapted to this filtration) in all cases considered below, so that the strong Markov property applies.

\begin{Lemma}\label{Lem:BM}
Let $M$ be one of the three 2-dimensional space forms, that is, $\bR^2$, $\bS^2$, or $\bH^2_+$, with polar coordinates $(r,\theta)$. Let $X_t$ be a Brownian motion started from the origin (the point with $r=0$), and, for any $a\neq 0$, let $\sigma_a$ be the first time the associated stochastic area hits the level $a$. Then $\sigma_a$ is finite a.s. Also, for any point $(r,\theta)\in M$ with $r\neq 0$ (and $r< \pi$ if $M=\bS^2$), let $R_{(r,\theta)}$ be the reflection through the unique (complete) geodesic containing $(0,0)$ and $(r,
\theta)$. Then the process 
\[
\tilde{X}_t = \begin{cases}
R_{\lp r\lp X_{\sigma_a} \rp ,\theta\lp X_{\sigma_a}\rp\rp} X_t & \text{for $t\leq \sigma_a$} \\
X_t  & \text{for $t> \sigma_a$}
\end{cases}
\]
is a Brownian motion from the origin.
\end{Lemma}

\begin{proof}
We begin with basic geometric properties shared by all three possible choices of $M$; the point is that the proof relies only on these properties and not on other details of the geometry. Brownian motion started from the origin is rotationally symmetric, as is the associated stochastic area, and thus the distribution of $X_{\sigma_a}$ (for any $a\neq 0$) is also rotationally symmetric. The reflection $R_{(r,\theta)}$ is an orientation-reversing isometry which preserves (the law of) Brownian motion while reversing the sign of the associated stochastic area. In addition, $R_{(r,\theta)}$ depends only on $\theta$ (not on $r$).

On any Riemannian manifold (of dimension more than 1), Brownian motion doesn't hit points. In particular, once $X_t$ leaves the origin, it almost surely never returns, so that $\theta\lp X_{\sigma_a}\rp$ and thus also $\tilde{X}_t$ are well defined. (Note that one could also make an arbitrary choice of a geodesic to reflect over if $ r\lp X_{\sigma_a} \rp=0$, like $\{\theta=0\}\cup\{\theta=\pi\}$, and the geometry of the reflection would work just as well.) 

If $M=\bR^2$, then the stochastic area has the representation \eqref{Eqn:HeisTC}, and because $r_t$ is a 2-dimensional Bessel process, it's immediate that $\int_0^t r^2_s\, ds\rightarrow \infty$ a.s. Thus $z_t$ hits every real value in finite time, and $\sigma_a$ is finite, a.s. As for the stochastic area processes on $\bS^2$ and $\bH^2_+$, they have been studied in \cite{FabriceWangArea} (see also the book-length treatment \cite{FJN-Book}), and it follows from this that $\sigma_a$ is a.s.\ finite in those cases as well. (We will see this explicitly later when we give the analogues of  \eqref{Eqn:HeisTC}, but for now, we record it, so that $\tilde{X}_t$ is well defined without needing a caveat.)

Let $\hat{X}_t$ be the process 
\[
\hat{X}_t=R_{\lp r\lp X_{\sigma_a} \rp ,\theta\lp X_{\sigma_a}\rp\rp} X_t ,
\]
which is well defined because $\sigma_a$ is a.s.\ finite. The main task is to show that $\hat{X}_t$ is a Brownian motion (on $M$, of course). It's clear that $\hat{X}_t$ has continuous paths, so in order to show it is a Brownian motion, it's enough to show that it has the right finite dimensional distributions. 

Consider any finite collection of times $0< t_1<t_2<\cdots <t_n<\infty$ and any sequence of Borel subsets of $M$, denoted $A_1,\ldots, A_n$. Because $\sigma_a$ is almost surely finite, a.e.\ path of $\hat{X}_t$ is given by a reflected path of $X_t$, with the reflection depending only on $\theta\lp X_{\sigma_a}\rp$ (which is almost surely well defined, as mentioned above). Thus, we can decompose the event that $\lp \hat{X}_{t_1},\ldots,\hat{X}_{t_n}\rp \in A_1\times\cdots\times A_n$, which we write as $\cap_i \lc \hat{X}_{t_i}\in A_i\rc$, by conditioning on $\theta\lp X_{\sigma_a}\rp$ as
\[
\Prob \lp \cap_i \lc \hat{X}_{t_i}\in A_i\rc \rp = \frac{1}{2\pi} \int_{\theta=0}^{2\pi} \Prob \lp \left. \cap_i \lc R_{(1,\theta)} X_{t_i}\in A_i\rc \, \right| \, 
\theta\lp X_{\sigma_a}\rp =\theta \rp \, d\theta ,
\]
using that $\theta\lp X_{\sigma_a}\rp$ is uniformly distributed on $[0,2\pi]$ by the above. Let $\Rot_{\theta}$ denote rotation around the origin by the angle $\theta$. By rotational invariance, the distribution of $X_t$ conditioned on the event $\theta\lp X_{\sigma_a}\rp=\theta$ is the same as the distribution of  $X_t$ conditioned on the event $\theta\lp X_{\sigma_a}\rp=0$ and then rotated by the angle $\theta$. It follows that
\[
\Prob \lp \cap_i \lc \hat{X}_{t_i}\in A_i\rc \rp = \frac{1}{2\pi} \int_{\theta=0}^{2\pi} \Prob \lp \left. \cap_i \lc  \Rot_{\theta} \lp R_{(1,0)} X_{t_i}\rp\in A_i\rc \, \right| \, 
\theta\lp X_{\sigma_a}\rp =0 \rp \, d\theta 
\]
using the usual relationship between reflection and rotation. Now $R_{(1,0)} X_{t}$ is also a Brownian motion, which we denote $X'_t$, and $\sigma_a$ for $X_t$ is equal to $\sigma_{-a}$ for $X'_t$. So we have, using the fact that $\sigma_{-a}$ is rotationally invariant and the distribution of $X'_t$ conditioned on the event $\theta\lp X'_{\sigma_{-a}}\rp=\theta$ is the same as the distribution of $X'_t$ conditioned on the event $\theta\lp X_{\sigma_{-a}}\rp=0$ and then rotated by the angle $\theta$ (and similar reasoning as above in the opposite order),
\[\begin{split}
\Prob \lp \cap_i \lc \hat{X}_{t_i}\in A_i\rc \rp &= \frac{1}{2\pi} \int_{\theta=0}^{2\pi} \Prob \lp \left. \cap_i \lc  \Rot_{\theta} X'_{t_i}\in A_i\rc \, \right| \, 
\theta\lp X'_{\sigma_{-a}}\rp =0 \rp \, d\theta \\
&= \frac{1}{2\pi} \int_{\theta=0}^{2\pi} \Prob \lp \left. \cap_i \lc  X'_{t_i}\in A_i\rc \, \right| \, 
\theta\lp X'_{\sigma_{-a}}\rp =\theta \rp \, d\theta .
\end{split}\]
We recognize this last integral on the right as the computation of $\Prob\lp \cap_i \lc  X'_{t_i}\in A_i\rc \rp$ by conditioning on $\theta\lp X'_{\sigma_{-a}}\rp$. Since $X'_t$ is a Brownian motion, it follows that $\hat{X}_t$ has the same finite-dimensional distributions as  a Brownian motion, and thus it is a Brownian motion (on $M$, of course).

Finally, $\tilde{X}_t$ is obtained by running a Brownian motion, namely $\hat{X}_t$ until a stopping time, specifically $\sigma_{-a}$ (meaning the first time the stochastic area associated to $\hat{X}_t$ hits the level $-a$, which is a stopping time for $\hat{X}_t$), and then concatenating with an independent Brownian motion, namely $X_{\sigma_a+s}$. (Since $X_{\sigma_a}$ is a fixed point of the reflection $R_{\lp r\lp X_{\sigma_a} \rp ,\theta\lp X_{\sigma_a}\rp\rp}$, we see that $X_{\sigma_a+s}$ starts from the correct point.) By the strong Markov property, it follows that $\tilde{X}_t$ is also a Brownian motion, on $M$, as desired.
\end{proof}

\begin{remark}\label{Rmk:Sigma}
When we treat $\operatorname{SL}(2)$ and $\operatorname{SU}(2)$, the vertical fiber will be given by $\bS^1$, rather than $\bR$. In particular, the stochastic area process $z_t$ will be understood modulo $4\pi$, giving the natural coordinate $z\in (-2\pi, 2\pi]$, with the endpoints identified. Then for $a\in (0,\pi]$, we will want to let $\sigma_a$ be the first time $z_t$ (starting from 0) hits $a$ or $a-2\pi$, with slight abuse of notation in continuing to use $\sigma_a$ for this hitting time. However, note that the proof of Lemma uses only that $\sigma_a$ is a stopping time, that it's a.s.\ finite, that $r\lp X_{\sigma_a}\rp$ is a.s.\ positive, and that $\theta\lp X_{\sigma_a} \rp$ is uniform. Thus the result and the proof apply, as written, to the case just described, when $\sigma_a$ is the first time $z_t$ hits $a$ or $a-2\pi$ (again, the rotational symmetry is clear).
\end{remark}

With this, we can formalize our construction of the vertical reflection coupling on $\bH$.

\begin{THM}\label{THM:HMax}
Consider two points in $\bH$ on the same vertical fiber with vertical separation $2a>0$; then there exist coordinates such that the points are $(0,0,0)$ and $(0,0,2a)$ (after possibly switching them). Let $B_t=(r_t,\theta_t,z_t)$ be a Heisenberg Brownian motion (expressed in cylindrical coordinates) started from $(0,0,0)$, and let $\sigma_a$ be the first hitting time of the set $\{z=a\}$, which is a.s.\ finite. Then if $\tilde{B}_t=(\tilde{r}_t,\tilde{\theta}_t,\tilde{z}_t)$ is the process defined by
\begin{equation}\label{Eqn:HeisCoupling}
\lp\tilde{r}_t,\tilde{\theta}_t,\tilde{z}_t\rp = \begin{cases}
\lp R_{\lp r_{\sigma_a} ,\theta_{\sigma_a}\rp} (r_t,\theta_t), 2a-z_t\rp  & \text{for $t\leq \sigma_a$} \\
(r_t,\theta_t,z_t)  & \text{for $t> \sigma_a$}
\end{cases}
\end{equation}
(where $R_{(r,\theta)}$ is the reflection of $\bR^2$ as in Lemma \ref{Lem:BM}), $\tilde{B}_t$ is a Heisenberg Brownian motion started from $(0,0,2a)$, coupled with $B_t$ such that $\sigma_a$ is their coupling time. Moreover, this coupling is maximal, and the coupling time $\sigma_a$ satisfies the reflection principle
\[
\Prob\lp \sigma_a> t\rp = 1-2 \Prob\lp z_t\geq a \rp .
\]
\end{THM}

Note that the origin is represented by the triple $(0,0,0)$ in either Cartesian or cylindrical coordinates (though in cylindrical coordinates, any value for $\theta$ would work, and we take 0 for convenience) and similarly for the point $(0,0,2a)$.

\begin{proof}
The proof is more or less collecting the geometric facts already described. Because $r_t$ is positive for all $t>0$ with probability 1, $r_{\sigma_a}>0$ a.s., and thus the reflection in \eqref{Eqn:HeisCoupling}, and also $\tilde{B}_t$, are well defined. It follows from Lemma \ref{Lem:BM} that $(\tilde{x}_t,\tilde{y}_t)$ is an $\bR^2$-Brownian motion. Further, using the fact that reflection the given reflection in $\bR^2$ reverses the sign of the L\'evy area and that $z_{\sigma_a}=2a-z_{\sigma_a}$, we see that $\tilde{z}_t$ as defined in \eqref{Eqn:HeisCoupling} is the L\'evy area associated to $(\tilde{x}_t,\tilde{y}_t)$ starting from $2a$, so that $\tilde{B}_t$ is a Heisenberg Brownian motion started from $(0,0,2a)$. It's clear from the construction that $B_t$ and $\tilde{B}_t$ meet at time $\sigma_a$ and not before (note that $z_t<a<\tilde{z}_t$ for $t<\sigma_a$). So $B_t$ and $\tilde{B}_t$ are coupled Heisenberg Brownian motions with coupling time $\sigma_a$.

Next, we claim that the evolution of $z_t$ is symmetric after $\sigma_a$; that is $z_{\sigma_a+s}-a$ and $-(z_{\sigma_a+s}-a)$ have the same distribution, for all $s>0$. This can be seen from the fact that the L\'evy area is reflection symmetric, but we also provide the following argument. Observe that $(r_t,z_t)$ is a diffusion in its own right, and in particular, satisfies the strong Markov property. Thus, for any Borel $A\subset \bR$, we have
\begin{equation}\label{Eqn:ZSymmetry}\begin{split}
\Prob\lp z_{\sigma_a+s} \in A \rp &= \E\lb\E\lb \Ind_{\{z_{\sigma_a+s} \in A\}} | r_{\sigma_a} \rb\rb \\
&=  \E\lb\Prob\lp  \lp a+ W_{\int_0^s \frac{1}{4} R^2_u\, du}\rp \in A \, \big|\, R_0 \rp\rb
\end{split}\end{equation}
where $W_t$ is a standard one-dimensional Brownian motion started from 0, $R_t$ is a two-dimensional Bessel process, independent of $W_t$, started from $R_0$, and this last expectation is understood with respect to the distribution of $R_0=r_{\sigma_a}$. But then it's clear that the probability on this last line is reflection symmetric, in the sense that
\[
\Prob\lp  \lp a+ W_{\int_0^s \frac{1}{4} R_u\, du}\rp \in A \rp =
\Prob\lp  \lp a+ W_{\int_0^s \frac{1}{4} R_u\, du}\rp \in a-A \rp .
\]
Then this symmetry is preserved after taking the outer expectation, and the desired symmetry of $z_t$ follows by simple algebra.

The reflection principle follows immediately from the symmetry of $z_t$ after $\sigma_a$, just as for one-dimensional Brownian motion. Finally, the argument for the maximality of the coupling is similar. That $\Prob\lp \sigma_a>t\rp$ is an upper bound for the total variation distance is immediate from the Aldous inequality. For the other direction, consider the set $U=\{z<a\}\subset \bH$. Note that $ \Law\lp B_t\rp(U) = \Prob\lp z_t <a \rp$, and similarly $ \Law\lp \tilde{B}_t\rp(U) = \Prob\lp \tilde{z}_t <a \rp$. By the reflection symmetry of $z_t$, after $\sigma_a$ the probability of $z_t>a$ and $z_t<a$ are each $\frac{1}{2}$, and the same is true for $\tilde{z}_t$. On the other hand, before $\sigma_a$, $z_t$ must be less than $a$ and $\tilde{z}_t$ must be greater than $a$. Finally, for each fixed $t$, the probability that either $z_t$ or $\tilde{z}_t$ equals $a$ is 0. This follows from the fact that $z_t$ is given by a time-changed Brownian motion with strictly increasing time-change or from the fact that $B_t$ has a density (by H\"ormander's theorem), and it implies that, for each fixed $t$, the probability that $\sigma_a=t$ is also 0. It follows that
\begin{equation}\label{Eqn:Max}\begin{split}
\Law\lp B_t\rp(U)- \Law\lp \tilde{B}_t\rp(U) &=  \Prob\lp z_t <a \rp- \Prob\lp \tilde{z}_t <a \rp  \\
& =  \Prob\lp z_t <a \text{ and } \sigma_a<t \rp +  \Prob\lp z_t <a \text{ and } \sigma_a>t \rp \\
& \qquad - \Prob\lp \tilde{z}_t <a \text{ and } \sigma_a<t \rp - \Prob\lp \tilde{z}_t <a \text{ and } \sigma_a>t \rp  \\
&=  \frac{1}{2} \Prob\lp \sigma_a<t \rp + \Prob\lp \sigma_a> t \rp
-  \frac{1}{2} \Prob\lp \sigma_a<t \rp  - 0  \\
& = \Prob\lp \sigma_a>t\rp .
\end{split}\end{equation}
This gives the matching lower bound for the total variation distance, completing the proof.
\end{proof}

In the present case, we can easily follow up on this qualitative result by computing $\Prob \lp \sigma_a>t\rp$. The point is that
the density of $z_t$ is explicitly known. For $\bR^2$-Brownian motion, started from the origin, the L\'evy area at time $t$, which is our $z_t$, has density $f_t(z)\, dz$, where 
\begin{equation}\label{Eqn:f_t}
f_t(z) = \frac{1}{t}\sech\lp \pi\frac{z}{t} \rp .
\end{equation}
It also satisfies the scaling relationship $f_t(z)=\frac{1}{t}f_1\lp \frac{z}{t}\rp$ and is symmetric around 0, which is consistent with reflection symmetry of $z_t$ just used in the preceding proof. Using this, we compute
\begin{equation}\begin{split}
\Prob\lp \sigma_a> t\rp &= 1-2 \int_a^{\infty} f_t(z)\, dz \\
&=  2 \int_0^{a/t} \sech\lp \pi u\rp\, du \\
&= \frac{4}{\pi}\arctan\lp\tanh\lp\frac{\pi}{2}\cdot\frac{a}{t}\rp\rp . \label{Eqn:HExact}
\end{split}\end{equation}
To understand the asymptotic rate of decay of the coupling time, we can integrate the bounds
\[
1-\frac{\pi^2 u^2}{2}< \sech\lp \pi u\rp < 1 ,
\]
 for $u>0$ to see that
\[
2\frac{a}{t}-\frac{\pi^2}{3}\lp \frac{a}{t} \rp^3<  \Prob\lp \sigma_a> t\rp < 2\frac{a}{t} .
\]
This establishes the following.
\begin{THM}\label{THM:HLimit}
In the same situation as Theorem \ref{THM:HMax}, the coupling time, and thus the total variation distance between $\Law\lp B_t\rp$ and $\Law( \tilde{B}_t )$ , which we know is equal to $\Prob\lp \sigma_a> t\rp$, is given exactly by \eqref{Eqn:HExact} and satisfies the asymptotics
\begin{equation}
\frac{ \Prob\lp \sigma_a> t\rp}{2\frac{a}{t}} \rightarrow 1 \quad \text{from below, as $t\rightarrow\infty$}.
\end{equation}
\end{THM}
Not only do we have the sharp decay rate of  $\Prob\lp \sigma_a> t\rp$ in time, but also linear dependence on the vertical displacement, which allows us to prove a vertical gradient estimate.

\begin{THM}\label{THM:HVGrad}
Let $P_t=e^{\frac{t}{2}\Lap_{\sH}}$ be the heat semigroup on $\bH$ and consider $f\in L^{\infty}(\bH)$. Then at any point $(x,y,z)\in\bH$ and for ant time $t>0$, we have
\[
\lab Z P_t f (x,y,z) \rab = \lab \nabla_{V} P_t f(x,y,z) \rab \leq \frac{1}{t}\|f\|_{\infty} .
\]
\end{THM}

\begin{proof}
By left translation, it's enough to consider the vertical gradient at the origin. That the vertical gradient exists for any positive time is a consequence of the hypoellipticity of the operator and smoothness of the heat kernel. Then we have
\[
\lab Z P_t f (0,0,0) \rab= \lab \lim_{a\searrow 0} \frac{P_t f (0,0,2a)-P_t f (0,0,0)}{2a} \rab \leq \frac{\|f\|_{\infty} \cdot 2\frac{a}{t}}{2a}
=\frac{1}{t}\|f\|_{\infty} .
\]
\end{proof}

\subsection{An efficient horizontal coupling}

For points not on the same fiber, by left translation and rotational invariance, they can be taken to be $(0,0,0)$ and $(h,0,v)$, for some $h>0$ and $v\in \bR$. Here we work again in Cartesian coordinates. We observe that $h$ is the horizontal displacement between the points while $v$ generalizes the previous notion of vertical displacement for points on the same fiber.

We first couple the $\bR^2$-marginals by reflection, and let $\sigma'_h$ be the coupling time for these marginals. While the distribution of $\sigma'_h$ is just given by the first passage time of a Brownian motion on $\bR$, we need the joint distribution of $\sigma'_h$ and the vertical displacement of the processes at $\sigma'_h$, which we denote by $2a(\sigma'_h)$, since this constitutes the random initial conditions for the vertical coupling. In particular, the second stage of the coupling is conditionally independent of the first given the vertical displacement when the $\bR^2$-marginals couple. Indeed, if we let $\tau$ be the coupling time for the two-stage coupling (which depends on the initial points $q$ and $\tilde{q}$ via the corresponding $h$ and $v$, even though we don't indicate this in the notation), then $\tau$ is equal in distribution to $\sigma'_h +\sigma_{a(\sigma'_h)}$, where $\sigma'_h$ and $\sigma_{a(\sigma'_h)}$ are conditionally independent, given $a(\sigma'_h)$.

Fortunately, \cite{BGM} showed that (recall that the horizontal stage of their coupling is the same as ours)
\[
\E\lb \frac{\lab 2a\lp \sigma'_h\rp \rab}{2t}\wedge 1 \rb \leq
C\lp \frac{h}{\sqrt{t}}+\frac{|v|}{t}\rp ,
\]
which, combined with the first passage distribution of Brownian motion, controls the distribution of $\sigma'_h$ and $a(\sigma'_h)$. From this and Theorem \ref{THM:HLimit}, one deduces that
\[
\Prob\lp \tau>t\rp \leq C\lp \frac{h}{\sqrt{t}} + \frac{|v|}{t} \rp \quad \text{for $t>\max\{h^2,2|v|\}$.}
\]
This gives an upper bound on the total variation distance, and the horizontal gradient bound of Theorem \ref{THM:HorSum} follows after noting that $\sqrt{h^2+|v|}$ is comparable to $\dist_{\sR}\lp (0,0,0), (h,0,v)\rp$. For more details as well as further applications of the coupling, see \cite{BGM}.

\section{Coupling on $\widetilde{\operatorname{SL}(2)}$}

We next treat the related model spaces $\operatorname{SL}(2)$ and its universal cover, $\widetilde{\operatorname{SL}(2)}$. As the nomenclature implies, it is natural to define them at the same time. While one might consider $\operatorname{SL}(2)$ to be more basic, it is actually $\widetilde{\operatorname{SL}(2)}$ which is more easily studied in our context, and so we start there.

\subsection{The geometry of $\widetilde{\operatorname{SL}(2)}$ and $\operatorname{SL}(2)$}\label{Sect:SLGeo}

The Lie group $\operatorname{SL}(2)=\operatorname{SL}(2,\bR)$ is the group of $2\times 2$ real matrices of determinant 1. Its Lie algebra $\mathfrak{sl}(2,\bR)$ consists of all $2\times 2$ matrices of trace $0$. A basis of $\mathfrak{sl}(2,\bR)$ is given by matrices
\begin{align*}
X=\frac{1}{2}\begin{pmatrix}
1 & 0 \\ 0 & -1
\end{pmatrix}, \quad Y=\frac{1}{2}\begin{pmatrix}
0 & 1 \\ 1 & 0
\end{pmatrix}, \text{ and}\quad Z=\frac{1}{2}\begin{pmatrix}
0 & 1 \\ -1 & 0
\end{pmatrix},
\end{align*}
for which the following relationships hold
\begin{align*}
[X,Y]=Z, \quad [Y,Z]=-X, \quad\text{and }[Z,X]=-Y.
\end{align*}
We denote by $\widetilde{X}$, $\widetilde{Y}$, and $ \widetilde{Z}$ the left-invariant vector fields on $\operatorname{SL}(2)$ corresponding to the matrices $X$, $Y$, and $Z$. Then $\operatorname{SL}(2)$ can be equipped with a natural sub-Riemannian structure $\left(\operatorname{SL}(2),\mathcal{H},\langle\cdot,\cdot\rangle_{\mathcal{H}}\right)$ where $\mathcal{H}_g=\operatorname{Span}\{\widetilde{X}(g),\widetilde{Y}(g)\}$ at any $g\in \operatorname{SL}(2)$ and $\{\widetilde{X},\widetilde{Y}\}$ forms an orthonormal frame for $\langle\cdot,\cdot\rangle_{\mathcal{H}}$. The sub-Laplacian on $\operatorname{SL}(2)$ has the form $\Delta_{\mathcal{H}}^{\operatorname{SL}(2)}=(\widetilde{X})^2+(\widetilde{Y})^2$, and the horizontal and vertical gradients are defined analogously to those of $\bH$.

We will work in cylindrical coordinates on both $\operatorname{SL}(2)$ and $\widetilde{\operatorname{SL}(2)}$. With a different normalization, this description of $\operatorname{SL}(2)$ and $\widetilde{\operatorname{SL}(2)}$ and their associated Brownian motions goes back at least to \cite{Bonnefont-Th} (see also \cite{Bonnefont-SL}). Here we follow the exposition in \cite{MagalieAdapted}, Section 2 of which can be consulted for details and proofs. We have
\begin{align*}
& \bR^{+}  \times [0,2\pi] \times (-2\pi,2\pi] \rightarrow SL(2)
\\
&
(r,\theta,z) \mapsto g=\operatorname{exp}\left((r\cos\theta)X+(r\sin\theta)Y\right)\operatorname{exp}(zZ)
\end{align*}
where $\operatorname{exp}:\mathfrak{sl}(2,\bR) \rightarrow \operatorname{SL}(2)$ is the exponential map. In cylindrical coordinates $(r,\theta,z)$, left-invariant vector fields $\widetilde{X},\widetilde{Y},\widetilde{Z}$ have the following expression
\begin{align*}
& \widetilde{X}=\cos(\theta+z)\partial_{r}-\sin(\theta+z)\left(\tanh \left(\frac{r}{2}\right) \partial_{z}+\frac{1}{2}\left(\frac{1}{\tanh \left(\frac{r}{2}\right)}-\tanh \left(\frac{r}{2}\right)\right) \partial_{\theta}\right),
\\
&
\widetilde{Y}=\sin(\theta+z)\partial_{r}+\cos(\theta+z)\left(\tanh \left(\frac{r}{2}\right) \partial_{z}+\frac{1}{2}\left(\frac{1}{\tanh \left(\frac{r}{2}\right)}-\tanh \left(\frac{r}{2}\right)\right) \partial_{\theta}\right),
\\
&
\widetilde{Z}=\partial_{z}
\end{align*}
and thus the sub-Laplacian $\Delta_{\mathcal{H}}^{\operatorname{SL}(2,\bR)}$ can be written as the following second-order differential operator 
\begin{align*}
\Delta_{\mathcal{H}}^{\operatorname{SL}(2)}=\partial^2_r+\coth(r)\partial_r+\frac{1}{\sinh^2 r}\partial_\theta^2
+\frac{1}{\cosh^2(\frac{r}{2})} \partial _{\theta}\partial_z +\tanh^2\lp\frac{r}{2}\rp \partial^2_z.
\end{align*}

The map $(r,\theta,z) \mapsto (r,\theta)$ gives a submersion onto the hyperbolic plane $\bH_+^2$, with $(r,\theta)$ giving (geodesic) polar coordinates on $\bH^2_+$. Any smooth curve in $\bH_+^2$ has a horizontal lift in $\operatorname{SL}(2)$, unique up to choosing the staring value of $z$, given by the signed area swept out by the curve relative to the origin, modulo $4\pi$. The length-minimizing geodesics are the lifts of curves in $\bH_+^2$ with constant geodesic curvature. (This is related to the isoperimetric problem on the hyperbolic plane and is the direct analogue of the fact that lifts of circles give geodesics for the Heisenberg group. See \cite[Section 5.3]{BauerFurutaniIwasaki2011} for more details.) Moreover, $\operatorname{SL}(2)$-Brownian motion, started from $(r_0,\theta_0,z_0)$ is given by $\bH_+^2$-Brownian motion, written in polar coordinates and started from $(r_0,\theta_0)$, with $z_t$ being the (signed) stochastic area process, started from $z_0$ and taken modulo $4\pi$, associated to the $\bH_+^2$-Brownian motion. In particular, $\operatorname{SL}(2)$-Brownian motion is given, in cylindrical coordinates, by the following system of SDEs,
\begin{equation}\label{Eqn:SLSDEs}
\begin{split}
dr_t &= dW_t^{(1)} +\frac{1}{2}\coth(r_t)\, dt \\
d\theta_t &= \frac{1}{\sinh(r_t)}\, dW_t^{(2)} \\
dz_t&=  \tanh\lp\frac{r_t }{2}\rp\, dW_t^{(2)} \, \mod 4\pi
\end{split}
\end{equation}
where $W_t^{(1)}$ and $W_t^{(2)}$ are independent one-dimensional Brownian motions. Further, we see that $z_t$ can be written as
\begin{equation}\label{Eqn:SLTC}
z_t = z_0+W_{\int_0^t \tanh^2\lp\frac{r_s}{2}\rp \, ds}  \mod 4\pi
\end{equation}
where $W_t$ is a one-dimensional Brownian motion independent of $r_t$ (and $r_t$ is the radial process on the hyperbolic plane).

It's clear from the cylindrical coordinates and submersion structure that $\operatorname{SL}(2)$ is diffeomorphic to $\bH_+^2\times\bS^1$ (and thus also to $\bR^2\times\bS^1$) and, in particular, is not simply connected. The universal cover $\widetilde{\operatorname{SL}(2)}$ is obtained by taking the universal cover of each vertical fiber, so that $\widetilde{\operatorname{SL}(2)}$ is diffeomorphic to $\bH_+^2\times \bR$. The cylindrical coordinates $(r,\theta,z)$ extend in the obvious way, now with values $r\in \bR^{+}$, $\theta\in [0,2\pi)$, and  $z\in \bR$. The sub-Riemannian structure is inherited from that of $\operatorname{SL}(2)$, and the expressions for $\widetilde{X}$, $\widetilde{Y}$, and $\widetilde{Z}$ in cylindrical coordinates remain the same, as does $\Delta_{\sH}$. While $\widetilde{\operatorname{SL}(2)}$ is not a matrix group, it is still a Lie group with the same Lie algebra, and the sub-Riemannian structure is left invariant.

The submersion map is essentially the same as for $\operatorname{SL}(2)$, and the horizontal lift of a curve in $\bH^2_+$ is again given by the signed area swept out relative to the origin, only now it is left real-valued, rather than taken modulo $4\pi$. Similarly, $\widetilde{\operatorname{SL}(2)}$-Brownian motion is again given by \eqref{Eqn:SLSDEs} and \eqref{Eqn:SLTC}, except without the ``modulo $4\pi$'' for $z_t$.

\begin{remark}\label{Rmk:Norm}
We choose our basis ${X,Y,Z}$ for the Lie algebra with the ``$\frac{1}{2}$'s'' in front of the matrices in order to make the submersion be onto $\bH^2_+$ with the standard metric of curvature $-1$. Other authors omit the factors of $\frac{1}{2}$, which makes some formulas simpler, but also rescales the metric so that the submersion gives $\bH^2_+$ the metric of curvature $-4$. In a similar vein, we note that one can rescale the vertical axis so that 
\[
dz_t =  x_t \,dy_t - y_t\, dx_t = r_t\,dW_t ,
\]
which is perhaps cleaner, but also means that the $z$-coordinate is twice the classical L\'evy area. Because our approach to constructing the vertical coupling is based on the submersion geometry in which the sub-Riemannian structure is given by a Riemannian space form augmented with the stochastic area functional, we choose our normalizations in order to make the space forms and their area functionals as geometrically natural as possible. 
\end{remark}

\subsection{Vertical coupling on  $\widetilde{\operatorname{SL}(2)}$}

Suppose we have two points in $\widetilde{\operatorname{SL}(2)}$ on the same vertical fiber with respect to the submersion. Without loss of generality (by left invariance and possibly exchanging the two points), we can assume that the points are $(0,0,0)$ and $(0,0,2a)$ for some $a>0$. 

\begin{THM}\label{THM:SLUMax}
Consider two points in $\widetilde{\operatorname{SL}(2)}$ on the same vertical fiber with vertical displacement $2a>0$; then there exist cylindrical coordinates such that the points are $(0,0,0)$ and $(0,0,2a)$ (after possibly switching them). Let $B_t=(r_t,\theta_t,z_t)$ be an $\widetilde{\operatorname{SL}(2)}$-Brownian motion (expressed in cylindrical coordinates) started from $(0,0,0)$, and let $\sigma_a$ be the first hitting time of the set $\{z=a\}$, which is a.s.\ finite. Then if $\tilde{B}_t=(\tilde{r}_t,\tilde{\theta}_t,\tilde{z}_t)$ is the process defined by
\[
\lp\tilde{r}_t,\tilde{\theta}_t,\tilde{z}_t\rp = \begin{cases}
\lp R_{\lp r_{\sigma_a} ,\theta_{\sigma_a}\rp} (r_t,\theta_t), 2a-z_t\rp  & \text{for $t\leq \sigma_a$} \\
(r_t,\theta_t,z_t)  & \text{for $t> \sigma_a$}
\end{cases}
\]
(where $R_{(r,\theta)}$ is the reflection of $\bH_+^2$ as in Lemma \ref{Lem:BM}), $\tilde{B}_t$ is an $\widetilde{\operatorname{SL}(2)}$-Brownian motion started from $(0,0,2a)$, coupled with $B_t$ such that $\sigma_a$ is their coupling time. Moreover, this coupling is maximal, and the coupling time $\sigma_a$ satisfies the reflection principle
\[
\Prob\lp \sigma_a> t\rp = 1-2 \Prob\lp z_t\geq a \rp .
\]
\end{THM}

\begin{proof}
In light of the discussion of cylindrical coordinates on $\widetilde{\operatorname{SL}(2)}$, the proof is almost identical to that of Theorem \ref{THM:HMax}, so we only briefly describe the argument and necessary adjustments. It's well known that Brownian motion on $\bH_+^2$ is transient, so from \eqref{Eqn:SLTC} (which is valid here without the ``modulo $4\pi$''), we see (more concretely than before) that $\sigma_a$ is a.s.\ finite. Almost surely, $r_t$ never hits 0 for positive $t$ here as well, so the reflection, and also $\tilde{B}_t$, are well defined. From Lemma \ref{Lem:BM} and the fact that reflection in $\bH^2_+$ changes the sign of the stochastic area, we see just as before that $\tilde{B}_t$ is an $\widetilde{\operatorname{SL}(2)}$-Brownian motion started from $(0,0,2a)$ and that the processes meet exactly at $\sigma_a$.

That $z_t$ is symmetric after $\sigma_a$ follows from the obvious analogue of \eqref{Eqn:ZSymmetry}, and the reflection principle follows. If we let $U={z<a}\subset \widetilde{\operatorname{SL}(2)}$, we again have that, for any fixed $t$, the probability that either $z_t$ or $\tilde{z}_t$ equals $a$ is 0, and then the notation is such that \eqref{Eqn:Max} holds as written (just understood as applying in the present context). Maximality of the coupling follows, and the theorem is proven.
\end{proof}

Just as for $\bH$, the natural next step is to understand the distribution of $\sigma_a$. Unlike the case of $\bH$, where the density of $z_t$ was known, for $\widetilde{\operatorname{SL}(2)}$, there is no similar result. The stochastic area process on the hyperbolic plane has not been especially well studied. The most relevant results appear to be those of \cite{FabriceWangArea}, where they show that
\[
\frac{z_t}{\sqrt{t}}\rightarrow N(0,1) \quad\text{in distribution as $t\rightarrow \infty$}.
\]
Thus, if $F_t(x)$ is the cdf of $\frac{z_t}{\sqrt{t}}$, we have that $F_t(x)\rightarrow \Phi(x)$ uniformly in $x$, where $\Phi$ is the cdf of the standard normal, as usual. However, in order to use this in the reflection principle, we need more than this. In particular, we need an explicit estimate of the rate of convergence, which we now establish.

\begin{THM}\label{THM:Ft-SL2}
Let $F_t(x)$ be the cdf of $\frac{z_t}{\sqrt{t}}$, with $z_t$ the obvious component of an $\widetilde{\operatorname{SL}(2)}$-Brownian motion started from the origin. Then there exist positive constants $c$ and $T_0$ such that 
\[
\Phi(x)\leq F_t(x)\leq  \Phi(x) + cx \quad \text{for all $x\geq 0$ and $t>T_0$.}
\]
\end{THM}

\begin{proof}
Naturally, we use cylindrical coordinates, and we write $z_t$ as
\[
W_{\int_0^t \tanh^2\lp \frac{r_s}{2}\rp\, ds}
\]
where $W_s$ is a Brownian motion (started at 0) independent of $r_t$. Thus, if we let $S(t)=\int_0^t \tanh^2\lp \frac{r_{\tau}}{2}\rp\, d\tau$, we see that $S(t)$ is a continuous, strictly increasing process that gives the time change that governs $z_t$. Conditioned on the value of $S(t)$, $z_t$ will be normal with mean 0 and variance $s$, so that 
\[
\Prob\lp \frac{z_t}{\sqrt{t}}\leq x \, \middle| \, S(t)=s  \rp =  \Phi\lp \frac{x\sqrt{t}}{\sqrt{s}} \rp .
\]
If, for each $t>0$, we let $G_t(s)$ be the cdf of $S(t)$, we then see that
\begin{equation}\label{Eqn:Ft}
F_t(x) =\int_0^{\infty} \Phi\lp \frac{x\sqrt{t}}{\sqrt{s}} \rp \, dG_t(s) .
\end{equation}

Next, we estimate $G_t(s)$. Recalling that $r_t$ is the distance from (a point designated as) the origin in the hyperbolic plane of a hyperbolic Brownian motion started from the origin, and the basic properties of $\tanh x$, we see that $\tanh r_t$ is a.s.\ strictly between 0 and 1 for all $t>0$. Thus, for any fixed $t>0$, we have $G_t(0)=0$ and $G_t(t)=1$, and moreover, the probability measure $dG_t$ assigns no mass to the interval $s\in[t,\infty)$. To understand the behavior of $G_t$ near $s=0$, we recall that Brownian motion on the hyperbolic plane is transient, so that the associated Green's function (which gives the total occupation density of Brownian motion) is locally integrable. Moreover, Green's function has a logarithmic singularity at the origin of the form $-\frac{1}{2\pi} \log r$. In polar normal coordinates, the hyperbolic area element is asymptotic to the Euclidean one, so, for small $\eps$, we can bound the total time $r_t$ spends below the level $\eps$ as
\[\begin{split}
\E\lb \int_0^{\infty} \mathbf{1}_{r_t<\eps} \,dt \rb &=
\int_0^{2\pi}\int_0^{\eps} \lp -\frac{1}{2\pi} \log r +O(1)\rp\lp r+O\lp r^3\rp  \rp \, dr\, d\theta \\
&= -\frac{1}{2}\eps^2\log \eps + O\lp  \eps^2  \rp .
\end{split}\]
Since $\tanh^2 x$ is comparable to $x^2$ near $x=0$, it follows that, for small enough $\eps'>0$, there is a constant $C$ such that 
\[
\E\lb \int_0^{\infty} \mathbf{1}_{\{\tanh^2 \frac{r_t}{2}<\eps\}} \,dt \rb < C\cdot\eps\lp -\log \eps\rp ,
\]
for all $\eps\in(0,\eps']$ (which certainly means that $\eps'<1$), and then by Markov's inequality,
\[
\Prob \lp \int_0^{\infty} \mathbf{1}_{\{\tanh^2 \frac{r_t}{2}<\eps\}}dt >u \, \rp< C\cdot \frac{ \eps\lp -\log \eps\rp}{u}, 
\]
where again the constant $C$ is independent of $\eps\in(0,\eps']$. We now fix some small, positive $\eps'$ so that the above asymptotic analysis holds, and we denote the event on the left-hand side by $B(u,\eps)$.

On the complement of $B(u,\eps)$, we have that, for any terminal time $T$, $\tanh^2 \frac{r_t}{2}\geq \eps$ for time at least $T-u$ (on the interval $t\in[0,T]$), meaning that $S(T)> \eps(T-u)$ on $B^c(u,\eps)$. It follows that if $\eps(T-u) \geq s$, the only contribution to $G_T(s)$ comes from $B(u,\eps)$. In particular, for a fixed $T>0$, we have that, for any $s>0$,
\[
G_T(s) <C\cdot \frac{ \eps\lp -\log \eps\rp}{u}
\]
for all pairs $\eps$ and $u$ such that $\eps\in (0,\eps']$, $u\in (0,T)$, and $s\leq \eps(T-u)$. So for $s\in (0,\eps']$, as long as $T>2$, we can take $\eps=s$ and $u=\frac{T}{2}$ so that
\begin{equation}\label{Eqn:SmallS}
G_T(s) <C\cdot \frac{s\lp -\log s\rp}{T} \quad\text{for all $s\in (0,\eps']$ and $T>2$,}
\end{equation}
where the constant $C$ depends only on $\eps'$.

Continuing, for $s>\eps'$, we have a similar estimate. The expected amount of time $\tanh^2 \frac{r_t}{2}$ spends below $\frac{1}{2}$ is finite, simply because the Green's function is locally integrable and $\tanh r\nearrow 1$ as $r\rightarrow\infty$ so that the corresponding sub-level set is bounded. Let this expected amount of time be $2c$. So again by Markov's inequality,
\[
\Prob \lp \int_0^{\infty} \mathbf{1}_{\{\tanh^2 \frac{r_t}{2}<\frac{1}{2}\}} \, dt >u\rp< \frac{2c}{u}, 
\]
for some constant $c>0$. If we call the event on the left-hand side of the above $C(u)$, then on the complement of $C(u)$, as before, we have that $S(T)> \frac{1}{2}(T-u)$.  For any $s<\frac{1}{2} T$, this will be the case with $u=2(T-2s)$, so we have
\begin{equation}\label{Eqn:MediumS}
G_T(s) < \frac{c}{T-2s} \quad\text{for all $s<\frac{1}{2}T$} .
\end{equation} Finally, we recall the basic fact that 
\begin{equation}\label{Eqn:LargeS}
G_T(s) \leq 1 \quad\text{for all $s$ and $T$} .
\end{equation}

Returning to \eqref{Eqn:Ft}, we can restrict the integral to $s\in[0,t]$, since $dG_t$ assigns no mass to $[t,\infty)$. Let 
\[
\phi(x) = \Phi'(x) = \frac{1}{\sqrt{2\pi}} e^{-\frac{1}{2}x^2}
\]
be the density of a standard normal. Assuming that $t>2$, we use integration by parts to write
\[\begin{split}
F_t(x) &= \lim_{a\searrow 0} \int_a^t \Phi\lp \frac{x\sqrt{t}}{\sqrt{s}} \rp \, dG_t(s) \\
&= \lim_{a\searrow 0}\lb  \int_a^t \frac{1}{2}\frac{x\sqrt{t}}{s^{3/2}} \phi\lp \frac{x\sqrt{t}}{\sqrt{s}}\rp G_t(s)\, ds
+\Phi(x)G_t(t)-\Phi\lp \frac{x\sqrt{t}}{\sqrt{a}}\rp G_t(a)\rb \\
&= \Phi(x) +\frac{1}{2}x\sqrt{t}\int_0^t \frac{1}{s^{3/2}} \phi\lp \frac{x\sqrt{t}}{\sqrt{s}}\rp G_t(s)\, ds,
\end{split}\]
where we've used that $G_t(t)=1$, and $G_t(a)\searrow 0$ as $a\searrow 0$ by \eqref{Eqn:SmallS} while $\Phi$ is bounded. Further, every factor in the last integral is non-negative, and $\phi\leq \frac{1}{\sqrt{2\pi}}$ everywhere, so, for $x\geq 0$ and $t>2$, we have
\begin{equation}\label{Eqn:FEst}
\Phi(x) \leq F_t(x) \leq \Phi(x)+\frac{x\sqrt{t}}{2\sqrt{2\pi}}\int_0^t \frac{1}{s^{3/2}} G_t(s)\, ds .
\end{equation}

We partition the region of integration $s\in[0,t]$ into 3 intervals. First, on $[0,\eps']$, we use \eqref{Eqn:SmallS} to see that
\[
\int_0^{\eps'} \frac{1}{s^{3/2}} G_t(s)\, ds \leq \frac{C}{t} \int_0^{\eps'} \frac{-\log s}{s^{1/2}} \, ds 
= \frac{C}{t} \sqrt{\eps'}\lp 4-2\log\eps'\rp 
=O\lp \frac{1}{t}\rp ,
\]
where we recall that $\eps'<1$ so that $\log \eps' <0$. Next, on $[\eps',\frac{1}{4}t]$, we use \eqref{Eqn:MediumS} to see that
\[\begin{split}
\int_{\eps'}^{\frac{1}{4}t} \frac{1}{s^{3/2}} G_t(s)\, ds &\leq c \int_{\eps'}^{\frac{1}{4}t} \frac{1}{s^{3/2}\lp t-2s \rp} \, ds \\
&= c\lb \frac{2\sqrt{2}}{t^{3/2}}\lp \Arctanh\lp\frac{\sqrt{2}}{2}\rp - \Arctanh\lp\frac{\sqrt{2}\sqrt{\eps'}}{\sqrt{t}}\rp\rp
+\frac{2}{\sqrt{\eps'}}\frac{1}{t}-\frac{4}{t^{3/2}}\rb  \\
&=O\lp \frac{1}{t}\rp .
\end{split}\]
Finally,  on $[\frac{1}{4}t,t]$, we use \eqref{Eqn:LargeS} to see that
\[
\int_{\frac{1}{4}t}^t \frac{1}{s^{3/2}} G_t(s)\, ds \leq \int_{\frac{1}{4}t}^t \frac{1}{s^{3/2}} \, ds
=\frac{2}{\sqrt{t}}  
=O\lp \frac{1}{\sqrt{t}}\rp .
\]
Putting these three together, we see that
\[
\int_0^t \frac{1}{s^{3/2}} G_t(s)\, ds = O\lp\frac{1}{\sqrt{t}}\rp ,
\]
and comparing with \eqref{Eqn:FEst} gives the theorem.
\end{proof}

We return to determining the asymptotic rate of $\Prob\lp \sigma_a >t \rp$. The reflection principle for $z_t$ gives
\[
\Prob\lp \sigma_a >t \rp =1-2\Prob\lp z_t>a \rp .
\]
Since $P\lp z_t>a \rp = P\lp \frac{z_t}{\sqrt{t}}>\frac{a}{\sqrt{t}} \rp=1-F_t\lp \frac{a}{\sqrt{t}} \rp$, assuming that $t>T_0$, we  apply Theorem \ref{THM:Ft-SL2} to get
\[
2\Phi\lp  \frac{a}{\sqrt{t}} \rp-1 \leq  P\lp \sigma_a >t \rp  \leq 2\Phi\lp\frac{a}{\sqrt{t}}\rp-1+2cx .
\]
Recall that
\[\begin{split}
\frac{1}{2}+\frac{1}{\sqrt{2\pi}}x+O\lp x^2\rp & \leq \Phi(x) \quad\text{for $x\geq 0$ and near 0} , \\
\text{and}\quad \Phi(x)&\leq \frac{1}{2}+\frac{1}{\sqrt{2\pi}}x \quad\text{for all $x\geq 0$.} 
\end{split}\]
Thus, we have that, for $t\geq T_0$,
\[
P\lp \sigma_a >t \rp  \leq \lp\frac{1}{\sqrt{2\pi}} +2c\rp \frac{a}{\sqrt{t}} \quad \text{for all $a>0$.}
\]
Moreover, this bound is sharp, in the sense that, for any $a>0$ and $\eps>0$, there exists $T'>0$ (which may depend on $a$ and $\eps$), such that
\[
P\lp \sigma_a >t \rp  \geq \lp\frac{1}{\sqrt{2\pi}}-\eps\rp \frac{a}{\sqrt{t}}  \quad\text{for $t\geq T'$}.
\]

We summarize these computations as follows.

\begin{THM}\label{THM:SLVert}
Consider two points in $\widetilde{\operatorname{SL}(2)}$ that are in the same vertical fiber with vertical displacement $2a>0$; then there exist cylindrical coordinates such that the points are $(0,0,0)$ and $(0,0,2a)$. If $\Law\lp B_t\rp$ and $\Law\lp \widetilde{B}_t\rp$ are the transition measures for the heat flow from these two points, then there exist constants $c>0$ and $T_0>0$, which do not depend on $a$, such that
\[
\TVd \lp \Law\lp B_t\rp, \Law\lp \tilde{B}_t\rp \rp \leq c\frac{a}{\sqrt{t}} \quad \text{for all $t>T_0$.}
\]
Moreover, this bound is sharp in the sense that there exists a constant $c'>0$ such that, for any $a>0$ there exists $T'>0$ (which may depend on $a$), such that
\[
\TVd \lp \Law\lp B_t\rp, \Law\lp \tilde{B}_t\rp \rp  \geq c' \frac{a}{\sqrt{t}}  \quad\text{for $t\geq T'$}.
\]
\end{THM}

From here, the same proof as for Theorem \ref{THM:HVGrad} yields the following.

\begin{THM}\label{THM:SLUVGrad}
Let $P_t=e^{\frac{t}{2}\Lap_{\sH}}$ be the heat semigroup on $\widetilde{\operatorname{SL}(2)}$ and consider $f\in L^{\infty}(\widetilde{\operatorname{SL}(2)})$. Then there exist constants $c>0$ and $T_0>0$ such that, at any point $(x,y,z)\in\widetilde{\operatorname{SL}(2)}$ and for any time $t>T_0$, we have
\[
\lab Z P_t f (x,y,z) \rab = \lab \nabla_{V} P_t f(x,y,z) \rab \leq \frac{c}{\sqrt{t}}\|f\|_{\infty} .
\]
\end{THM}

\subsection{Horizontal coupling} \label{Sect:SLU-H}

Unlike $\bH$, for points not on the same vertical fiber of $\widetilde{\operatorname{SL}(2)}$, a successful coupling is not possible. Indeed, the coupling of the $(r,\theta)$-marginals (which are Brownian motions on $\bH_+^2$) is already an obstruction to a successful coupling, but it follows from the above that this is the only obstruction. If we consider the two-stage coupling in which we first use the standard (Riemannian, Markov) reflection coupling on $\bH_+^2$ to couple the $(r,\theta)$-marginals and then, once the particles are on the same vertical fiber, use the vertical reflection coupling just developed, it follows from the fact that the vertical coupling is successful that the particles will couple in $\widetilde{\operatorname{SL}(2)}$ if and only if the first stage of the coupling is successful. 

\begin{lemma}\label{Lem:HBM}
Let $p$ and $q$ be two distinct points of $\bH_+^2$, and let the (hyperbolic) distance between them be $2r$. Then the reflection coupling of Brownian motions started from these two points is maximal and it succeeds with probability
\[
1-\frac{4}{\pi}\arctan\lp \tanh \lp\frac{r}{2}\rp \rp ,
\]
that is, the probability that the processes couple in finite time is given by the above.
\end{lemma}

\begin{proof}
Consider the Poincar\'e disk model of hyperbolic space, with Cartesian coordinates $(x,y)$ on the disk. By the symmetries of $\bH_+^2$, without loss of generality, we can assume our two processes start from $(a,0)$ and $(-a,0)$ for $0<a<1$, where $a$ depends on $r$ in a way we discuss below. If $(x_t,y_t)$ is the Brownian motion started from $(a,0)$ and $\sigma'_r$ is the first time $x_t$ hits 0, then the reflection coupling is given by letting $(\tilde{x}_t,\tilde{y}_t)=(-x_t,y_t)$ until $\sigma'_r$, and then letting $(\tilde{x}_t,\tilde{y}_t)=(x_t,y_t)$ for $t\geq \sigma'_r$ (is $\sigma'_r$ is finite, otherwise the previous equation holds for all time). Then the processes successfully couple if and only if $\sigma'_r$ is finite. That this coupling is maximal and Markov follows from and \cite{Kuwada-Suf} and \cite{Kuwada-Nec} as discussed in Section \ref{Sect:Intro1}, but in this simple case, it's also clear from the construction (maximality follows from the analogous computation to \eqref{Eqn:Max}). So it remains to compute the probability that $\tau$ is finite. There are a few ways to do this, such as considering the SDE satisfied by the distance between the processes, but we give a simple argument using the harmonic function theory of the disk.

Let $u(a,b)$ be the probability that a Brownian motion started from $(a,b)$ in the right half-disk (that is, with $a^2+b^2<1$ and $a>0$) hits the circle $x^2+y^2=1$ before it hits the line $\{x=0\}$, which is the diameter of the circle corresponding to the bisector between $(-a,0)$ and $(a,0)$. (Whether the Brownian motion is understood as a hyperbolic or Euclidean Brownian motion doesn't matter; since $\bH^2_+$ and the disk are conformally equivalent and thus their Brownian motions differ only by time change. In the disk model, the transience of hyperbolic Brownian motion corresponds to the fact that the Euclidean Brownian motion goes to the boundary of the disk.) Then $u$ is harmonic and is determined by having boundary values 1 on the semi-circle and 0 on the diameter. By symmetry, $u$ is the restriction to the right half-disk of the harmonic function on the disk which has boundary values $-1$ on the part of the unit circle with with $x<0$ and 1 on the part with $x\geq 0$. Thus we now extend $u$ to the entire disk in this way, and observe that $u$ is given in terms of the integral of the boundary values against the Poisson kernel.

In particular, if we let $(r,\theta)$ be the usual polar coordinates on the disk, we have the Poisson kernel
\[
P_r(\theta)=\frac{1-r^2}{1-2r\cos\theta+r^2}.
\]
Then for $u(a,0)$, the contributions from $y>0$ and $y<0$ are the same by symmetry, so we have
\[
u(a,0)=\frac{1}{\pi}\int_0^{\pi}P_a(-t)f\lp e^{it}\rp\, dt
\]
where
\[
f\lp e^{it}\rp=\begin{cases} 1& \text{ if $t< \pi/2$}  \\ -1& \text{ if $t\geq \pi/2$}  \end{cases} .
\]
After the change of variable $t\mapsto t-\frac{\pi}{2}$ on the interval $\lp\frac{\pi}{2},\pi\rp$, we can rewrite this as
\[
u(a,0)=\frac{1}{\pi}\int_0^{\pi/2} \lp \frac{1-a^2}{1-2a\cos t+a^2}-\frac{1-a^2}{1+2a\sin t+a^2}\rp \, dt .
\]
This can be evaluated explicitly, giving
\[
u(a,0)= \frac{1}{2}-\frac{2}{\pi}\lp \arctan\lp \frac{a+1}{a-1}\rp +\arctan\lp \frac{1}{a}\rp\rp .
\]

Continuing, the point $(a,0)$ has hyperbolic distance from the origin of $2\Arctanh\lp a\rp$, so that $a$ is determined by asking for this distance to be $r$. Solving this and using that $u(a,0)$ is the probability of the particle leaving the half disk at the circle, rather than the diameter, we conclude that the probability that $\sigma'_r$ is finite is given by
\[
\Prob\lp \sigma'_r <\infty \rp =
\frac{1}{2}+\frac{2}{\pi}\lp \arctan\lp \frac{a+1}{a-1}\rp +\arctan\lp \frac{1}{a}\rp\rp
 \text{ where $a=\tanh\lp \frac{r}{2}\rp$} .
\]
Using that $\arctan\lp\frac{x+1}{x-1}\rp = -\frac{\pi}{4}-\arctan x $ and $\arctan\lp \frac{1}{x}\rp=\frac{\pi}{2}-\arctan x$ (at least for $0<x<1$, which is the current situation), we can simplify the above, giving the desired result.
\end{proof}

In light of the above, for any two points in $\widetilde{\operatorname{SL}(2)}$, we refer to the hyperbolic distance between their projections onto $\bH_+^2$, under the submersion, as their horizontal displacement.

We can now prove the following.

\begin{THM}\label{THM:SLU-Hor}
Let $(r_0,\theta_0,z_0)$ and $\lp\tilde{r}_0,\tilde{\theta}_0,\tilde{z}_0\rp$ be two points in $\widetilde{\operatorname{SL}(2)}$, and let $2r$ be their horizontal displacement (that is, the hyperbolic distance between $(r_0,\theta_0)$ and $\lp\tilde{r}_0,\tilde{\theta}_0\rp$ as points in $\bH_+^2$, written in polar coordinates). Then if $B_t$ and $\tilde{B}_t$ are $\widetilde{\operatorname{SL}(2)}$-Brownian motions from $(r_0,\theta_0,z_0)$ and $\lp\tilde{r}_0,\tilde{\theta}_0,\tilde{z}_0\rp$ respectively, we have
\begin{align*}
\lim_{t\to \infty}\TVd \lp \Law\lp B_t\rp, \Law\lp \tilde{B}_t\rp \rp = \frac{4}{\pi}\arctan\lp \tanh \lp\frac{r}{2}\rp \rp.
\end{align*}
If $f\in L^{\infty}\lp \widetilde{\operatorname{SL}(2)} \rp$ and $P_t=e^{\frac{1}{2}\Lap_{\sH}t}$ is the heat semigroup, then 
\[
\limsup_{t\to \infty} \lab P_t f (r_0,\theta_0,z_0) -P_t f\lp\tilde{r}_0,\tilde{\theta}_0,\tilde{z}_0\rp  \rab \leq \frac{4}{\pi}\arctan\lp \tanh \lp\frac{r}{2}\rp \rp \cdot\|f\|_{\infty}.
\]
Finally, if $f$ is also harmonic, then it is constant on vertical fibers and is given by the lift of a bounded harmonic function on $\bH_+^2$, that is, $f(r,\theta,z)=f_0(r,\theta)$ where $f_0$ is a bounded harmonic function on the hyperbolic plane.
\end{THM}

\begin{proof}
We start with the total variation. After change of coordinates by an isometry, we can assume that $(r_0,\theta_0)=(r,0)$ and $\lp\tilde{r}_0,\tilde{\theta}_0\rp=(-r,0)$. Note that the set $\{(r,\theta,z)\in \widetilde{\operatorname{SL}(2)}:\theta\in(-\pi/2,\pi/2)\}$ gives the half-disk $x>0$ in notation of the proof of Lemma \ref{Lem:HBM}. Because the reflection coupling in Lemma \ref{Lem:HBM} is maximal (for $\bH_+^2$) and because the total variation distance of $\Law\lp B_t\rp$ and $\Law\lp \tilde{B}_t\rp$ can't be less than the total variation distance of their $\bH_+^2$-marginals, it follows that
\begin{equation}\label{Eqn:TV-LB}
\TVd \lp \Law\lp B_t\rp, \Law\lp \tilde{B}_t\rp\rp \geq 
\frac{4}{\pi}\arctan\lp \tanh (r/2) \rp \quad\text{for all $t\geq 0$.}
\end{equation}
 
 For the matching upper bound, we use the two-stage coupling introduced above. Choose some small, positive $\eps$. Then by Lemma \ref{Lem:HBM}, there exists a time $T_1>0$ such that the horizontal stage of the coupling has succeeded with probability  at least $ \frac{4}{\pi}\arctan\lp \tanh (r/2) \rp-\eps$ by time $T_1$; that is, in our earlier notation $\sigma'_r\leq T_1$ with probability at least $ \frac{4}{\pi}\arctan\lp \tanh (r/2) \rp-\eps$. Then $z_{\sigma'_r}-\tilde{z}_{\sigma'_r}$ on the set $\{\sigma'_r\leq T_1\}$ gives a sub-probability distribution on $\bR$ (corresponding to the random vertical displacement from which we start the vertical reflection coupling), and there exists some $A>0$ such that this sub-probability distribution gives mass at least $ \frac{4}{\pi}\arctan\lp \tanh (r/2) \rp-2\eps$ to the interval $[-A,A]$. Considering the vertical reflection coupling (the second stage of the two-stage coupling), we see from Theorem \ref{THM:SLVert} that, given the $A$ from above, there is a time $T_2$ such that for any vertical displacement less than or equal to $A$, the vertical reflection coupling succeeds with probability $1-\eps$ in time $T_2$. Putting this together, we conclude that, under the two-stage coupling, $B_t$ and $\tilde{B}_t$ couple by time $T_1+T_2$ with probability at least  $ \frac{4}{\pi}\arctan\lp \tanh (r/2) \rp-3\eps$. Since $\eps>0$ was arbitrary, this shows that
 \[
 \limsup_{t\to \infty}\TVd \lp \Law\lp B_t\rp, \Law\lp \tilde{B}_t\rp \rp \leq  \frac{4}{\pi}\arctan\lp \tanh \lp\frac{r}{2}\rp \rp.
 \]
 Combining this with \eqref{Eqn:TV-LB}, we see that $\lim_{t\to \infty}\TVd \lp \Law\lp B_t\rp, \Law\lp \tilde{B}_t\rp \rp$ is as claimed.

The asymptotic bound on the semigroup follows immediately from the total variation.

Finally, let $f$ be bounded and harmonic (on $\widetilde{\operatorname{SL}(2)}$). Then $P_tf=f$ for all $t$, and thus Theorem \ref{THM:SLVert} implies that $f$ is constant on vertical fibers. So writing $f$ as $f(r,\theta,z)=f_0(r,\theta)$ where $f_0$ is a bounded function on $\bH_+^2$, we see that $f$ being harmonic becomes the condition
\[
\lp \partial^2_r+\coth(r)\partial_r+\frac{1}{\sinh^2 r}\partial_\theta^2\rp f_0 =0
\]
at all points. But this operator is just the Laplacian on $\bH_+^2$ in polar coordinates, finishing the proof.
\end{proof}

\section{Coupling on $\operatorname{SL}(2)$}

We have already described the cylindrical coordinates and submersion geometry of $\operatorname{SL}(2)$ in Section \ref{Sect:SLGeo}. The difference with $\widetilde{\operatorname{SL}(2)}$ is that the $z$-coordinate is now taken modulo $4\pi$, which means that the vertical fiber is a circle, rather than a line.

\subsection{Vertical coupling on $\operatorname{SL}(2)$}

If we consider two points on the same vertical fiber of $\operatorname{SL}(2)$, without loss of generality they can be taken to be $(0,0,0)$ and $(0,0,2a)$ for $a\in (0,\pi]$, where we use cylindrical coordinates and take $\theta=0$ when $r=0$ by convention. If $B_t=(r_t,\theta_t,z_t)$, then we want to couple the processes once $z_t$ gets halfway to $2a$, but because the vertical fiber is a circle, this means that $z_t$ can hit $a$ or $a-2\pi$. Indeed, each fiber can be decomposed into the two open semi-circles $(a-2\pi,a)$ and $(-2\pi,a-2\pi)\cup (a,2\pi]$, plus the two boundary points. Thus, with a combination of abuse of notation and a desire to emphasize the parallels between all the cases considered in this paper, in the present context, we let $\sigma_a$ be the first time $B_t$ hits $\{z=a\}\cup\{z=a-2\pi\}$.

\begin{THM}\label{THM:SLMax}
Consider two points in $\operatorname{SL}(2)$ on the same vertical fiber with vertical displacement $2a\in(0,\pi)$; then there exist cylindrical coordinates such that the points are $(0,0,0)$ and $(0,0,2a)$ (after possibly switching them). Let $B_t=(r_t,\theta_t,z_t)$ be an $\operatorname{SL}(2)$-Brownian motion (expressed in cylindrical coordinates) started from $(0,0,0)$, and let $\sigma_a$ be the first hitting time of the set $\{z=a\}\cup\{z=a-2\pi\}$, which is a.s.\ finite. Then if $\tilde{B}_t=(\tilde{r}_t,\tilde{\theta}_t,\tilde{z}_t)$ is the process defined by
\[
\lp\tilde{r}_t,\tilde{\theta}_t,\tilde{z}_t\rp = \begin{cases}
\lp R_{\lp r_{\sigma_a} ,\theta_{\sigma_a}\rp} (r_t,\theta_t), 2a-z_t\rp  & \text{for $t\leq \sigma_a$} \\
(r_t,\theta_t,z_t)  & \text{for $t> \sigma_a$}
\end{cases}
\quad \text{with $z_t$ and $\tilde{z}_t$ understood modulo $4\pi$}
\]
(where $R_{(r,\theta)}$ is the reflection of $\bH_+^2$ as in Lemma \ref{Lem:BM}), $\tilde{B}_t$ is an $\operatorname{SL}(2)$-Brownian motion started from $(0,0,2a)$, coupled with $B_t$ such that $\sigma_a$ is their coupling time. Moreover, this coupling is maximal, and the coupling time $\sigma_a$ satisfies the reflection principle
\[
\Prob\lp \sigma_a> t\rp = 1-2 \Prob\lp z_t\in  (-2\pi,a-2\pi)\cup (a,2\pi] \rp .
\]
\end{THM}

\begin{proof}
In comparison with the proof of Theorem \ref{THM:SLUMax}, we need to account for taking the quotient of the vertical fiber. So we first observe that
\[
\tilde{z}_{\sigma_a}=2a-z_{\sigma_a} =\begin{cases}
a &  \text{if $z_{\sigma_a}=a$} \\
2a-(a-2\pi) =a+2\pi = a-2\pi \mod 4\pi &\text{if $z_{\sigma_a}=a-2\pi$} .
\end{cases}
\]
This shows that $\tilde{B}_t$ is continuous at $t=\sigma_a$, so then by Lemma \ref{Lem:BM} (plus Remark \ref{Rmk:Sigma}) and the fact that reflection in $\bH_+^2$ reverses the sign of the stochastic area, we see that $\tilde{B}_t$ is an $\operatorname{SL}(2)$-Brownian motion. Moreover, we also see that the processes meet at time $\sigma_a$. Prior to $\sigma_a$, $z_t$ is in the open semi-circle $(a-2\pi,a)$ while $\tilde{z}_t$ is in the open semi-circle $(-2\pi,a-2\pi)\cup (a,2\pi]$, so $\sigma_a$ is the coupling time.

Note that the map $z\mapsto 2a-z$ on the fiber $(-2\pi,2\pi]$ exchanges $(a-2\pi,a)$ and $(-2\pi,a-2\pi)\cup (a,2\pi]$ and the evolution of $z_t$ is symmetric after $\sigma_a$ with respect to this reflection, by \eqref{Eqn:SLTC} and the same argument as in Theorem \ref{THM:HMax} (see \eqref{Eqn:ZSymmetry} in particular), from which the reflection principle follows. Finally, if we let $U=\{(r,\theta,z)\in \operatorname{SL}(2):z\in (a-2\pi,a)  \}\subset \operatorname{SL}(2)$, then we see that Equation \eqref{Eqn:Max} holds in the present situation, which gives the maximality of the coupling.
\end{proof}

\begin{remark}\label{Rmk:Circle}
While we have again reduced the coupling time to the hitting time of a certain set by the $z_t$-process, and we have a representation of $z_t$ as a time-changed Brownian motion, in this case we are not yet able to determine the precise asymptotics of $\sigma_a$. The difficulty comes from the fact that, even though $z_t$ is the same process as above, just taken modulo $4\pi$, for $\widetilde{\operatorname{SL}(2)}$, when the fiber was a line, we were interested in the hitting time of a certain level, while now, when the fiber is a circle, we're interested in the exit time from a bounded interval (that is, the first hitting time of $z=a$ or $z=a-2\pi$ is the same asking for the exit time from $(a-2\pi, a)$). Even for the classical case of a real-valued Brownian motion, computing the first passage time of the level $a$ by the reflection principle is straightforward, while computing the exit time of an interval using ``double reflection'' is much more involved. Indeed, it is easier to compute the lowest eigenvalue of the interval using Fourier series. For $z_t$ on $\operatorname{SL}(2)$, the independent time-change and the fact that $z_t$ is not a Markov process makes this more complicated. Further, for Theorem \ref{THM:SLVert}, it was enough to obtain a relatively course estimate on the time change and therefore on the density of $z_t$, since the sharp rate of decay of $\Prob\lp \sigma_a >t \rp$ was $1/\sqrt{t}$. But for $\operatorname{SL}(2)$, we expect an exponential decay for $\Prob\lp \sigma_a >t \rp$ (which we will see is the case in a moment), and determining the correct exponent requires more precise estimates.
\end{remark}

Instead of a sharp bound on the coupling time, we give a soft argument, using the compactness of the fibers, showing that the coupling time has exponential decay, as one would expect.

\begin{THM}\label{THM:SVert}
Consider two points in $\operatorname{SL}(2)$ that are in the same vertical fiber with vertical displacement $2a>0$; then there exist cylindrical coordinates such that the points are $(0,0,0)$ and $(0,0,2a)$. If $\Law\lp B_t\rp$ and $\Law\lp \tilde{B}_t\rp$ are the transition measures for the heat flow from these two points, then there exist constants $C>0$, $c>0$, and $T_0>0$, which do not depend on $a$, such that
\[
\TVd \lp \Law\lp B_t\rp, \Law\lp \tilde{B}_t\rp \rp \leq Ce^{-ct} \quad \text{for all $t>T_0$.}
\]
\end{THM}

\begin{proof}
Let $p_t (q_0,q_1)$ be the heat kernel on $\operatorname{SL}(2)$ (with respect to the Haar measure $\mu$), for time $t>0$ and points $q_0$ and $q_1$ in $\operatorname{SL}(2)$. Then $p_t$ is smooth and everywhere positive, by H\"ormander's theorem. Assume that $q_0\in \bH_+^2\times (a-2\pi,a)$ and let $V=  \bH_+^2\times \lp (-2\pi,a-2\pi)\cup (a,2\pi] \rp$. Then 
\[
p_1\lp q_0, V\rp = \int_{V}  p_1 (q_0,q)\, d\mu(q)
\]
is the probability that an $\operatorname{SL}(2)$-Brownian motion $B_t$, started from $q_0$ is in $V$ at time 1. By left-invariance, there exists $\alpha\in (-2\pi,2\pi]$ such that
\[
p_1\lp q_0, V\rp =  p_1\lp (0,0,0), \bH_+^2\times (\alpha,\alpha+2\pi) \rp 
\]
where the interval $(\alpha,\alpha+2\pi)$ is understood modulo $4\pi$ to denote a semi-circle in the vertical fiber. By the basic properties of the heat kernel just mentioned and the fact that $\alpha$ ranges over a compact (the circular vertical fiber), we see that there is some $c'>0$ such that the right-hand side is bounded below by $c'$ independent of $\alpha$, and it follows that $p_1\lp q_0, V\rp> c'$ for all $q_0$, independent of $a$.

Thus, if $B_t$ is a Brownian motion from $(0,0,0)$, the probability that it is not in $V$ at time 1 is less than or equal to $1-c'$, and by the Markov property, the probability that $B_i$ is not in $V$ for any $i\in\{1,2,\ldots, n\}$ is less than or equal to $(1-c')^n$. Now $V$ is chosen so that $\sigma_a$ is less or equal to the first hitting time of $V$ (and in fact they will be equal), so
\[
\Prob \lp \sigma_a>n\rp \leq \lp 1-c'\rp^n \quad\text{for all $n=1,2,\ldots$} 
\]
and this holds independently of $a$. Since $\Prob \lp \sigma_a>t\rp$ is non-increasing in $t$ and the coupling is maximal, the theorem follows by simple algebra.
\end{proof}

\subsection{Horizontal coupling on $\operatorname{SL}(2)$}

The horizontal coupling for $\operatorname{SL}(2)$ is the same as for $\widetilde{\operatorname{SL}(2)}$. Of course, the vertical coupling is faster for $\operatorname{SL}(2)$, but the speed of the vertical coupling isn't relevant for the results of Section \ref{Sect:SLU-H}, only the fact that the vertical coupling is successful. Thus  the same results hold, with essentially the same proofs (or by considering the behavior of the heat semigroup under taking a quotient of the vertical fiber), which we simply record here.

\begin{THM}
Theorem \ref{THM:SLU-Hor} holds with $\widetilde{\operatorname{SL}(2)}$ replaced by $\operatorname{SL}(2)$.
\end{THM}

\section{Coupling on $\operatorname{SU}(2)$}

\subsection{The sub-Riemannian geometry of $\operatorname{SU}(2)$}

The Lie group $\operatorname{SU}(2)$ is the group of $2\times 2$ complex unitary
matrices of determinant $1$, i.e.
\begin{align*}
G=\operatorname{SU}(2)=
\left\{\begin{pmatrix}
z_1 & z_2 \\
-\overline{z_2} & \overline{z_1}
\end{pmatrix}:z_1,z_2\in \mathbb{C},\vert z_1\vert^2+\vert z_2\vert^2=1 \right\}.   
\end{align*}
Its Lie algebra $\mathfrak{su}(2)$ consists of $2\times 2$ complex skew-adjoint
matrices with trace $0$. A basis of $\mathfrak{su}(2)$ is formed by the (normalized) Pauli matrices
\begin{align*}
X=\frac{1}{2}\begin{pmatrix}
0 & 1 \\ -1 & 0
\end{pmatrix}, \quad Y=\frac{1}{2}\begin{pmatrix}
0 & i \\ i & 0
\end{pmatrix}, \text{ and}\quad Z=\frac{1}{2}\begin{pmatrix}
i & 0 \\ 0 & -i
\end{pmatrix},
\end{align*}
for which the following relationships hold
\begin{align*}
[X,Y]=Z,\quad [Y,Z]=X, \text{ and}\quad [Z,X]=Y.
\end{align*}
We denote by $\widetilde{X}$, $\widetilde{Y}$, and $\widetilde{Z}$ the left-invariant vector fields on $\operatorname{SU}(2)$ corresponding to $X$, $Y$, and $Z$. Then $\operatorname{SU}(2)$ can be equipped with a natural sub-Riemannian structure $\left(\operatorname{SU}(2),\mathcal{H},\langle\cdot,\cdot\rangle_{\mathcal{H}}\right)$ where $\mathcal{H}_g=\operatorname{Span}\{\widetilde{X}(g),\widetilde{Y}(g)\}$ at any $g\in \operatorname{SU}(2)$ and $\{\widetilde{X},\widetilde{Y}\}$ forms an orthonormal frame for $\langle\cdot,\cdot\rangle_{\mathcal{H}}$. The sub-Laplacian on $\operatorname{SU}(2)$ has the form $\Delta_{\mathcal{H}}^{\operatorname{SU}(2)}=(\widetilde{X})^2+(\widetilde{Y})^2$, and the horizontal and vertical gradients are defined analogously to those of $\bH$.

We note that $\operatorname{SU}(2)$ is canonically identified with the set $\bS^3=\{x_1^2+x_2^2+x_3^2+x_4^2=1\}\subset \bR^4$. Also, $\tilde{Z}$ is tangent to the Hopf fibration of $\bS^3$ (see Example 1.3.5 in \cite{Petersen} and also Remark \ref{Rmk:Hopf} regarding the normalization), so the sub-Riemannian structure is compatible with a submersion.

Following \cite{MagalieAdapted}, we introduce cylindrical coordinates $(r, \theta, z)$ on $\operatorname{SU}(2)$ as
\[
\begin{bmatrix}
\cos \lp\frac{r}{2}\rp e^{i\frac{z}{2}} & \sin \lp\frac{r}{2}\rp e^{i\lp \theta-\frac{z}{2}\rp} \\
-\sin \lp\frac{r}{2}\rp e^{-i\lp \theta-\frac{z}{2}\rp}   &   \cos \lp\frac{r}{2}\rp  e^{-i\frac{z}{2}}   
\end{bmatrix}
\]
for $r\in[0,\pi]$, $\theta\in [0,2\pi]$, and $z\in (-2\pi,2\pi]$. Note that the submersion structure degenerates at $r=\pi$. Indeed, if we restrict $r$ to $ [0,\pi)$, then $(r,\theta)$ gives polar coordinates on $\bS^2$, centered at the north pole and omitting the south pole, which we denote by $S$ and which corresponds to $r=\pi$. Then $z\in (-2\pi,2\pi]$ is understood modulo $4\pi$ and gives a coordinate on the Hopf fibers, which are topological circles. Let $\tilde{S}\subset \operatorname{SU}(2)$ be the fiber over the south pole. The map $(r,\theta,z)\mapsto (r,\theta)$ gives the Hopf submersion, restricted to $\operatorname{SU}(2)\setminus\tilde{S}$. In particular, on $\operatorname{SU}(2)\setminus\tilde{S}$ we obtain a description of the sub-Riemannian structure directly analogous to what we had for $\operatorname{SL}(2)$. Any smooth curve on $\bS^2$ minus the south pole has a horizontal lift, uniquely determined by the lift $z_0\in (-2\pi,2\pi]$ of the initial point, given by the area swept out by the curve, relative to the north pole and modulo $4\pi$. Sub-Riemannian geodesics are given by the lifts of curves in $\bS^2$ with constant geodesic curvature. (For this description of the sub-Riemannian geodesics, we refer also to Chapter 1 of \cite{Montgomery}.)

As before, this description of horizontal curves via lifts of curves in $\bS^2$ extends to $\operatorname{SU}(2)$-Brownian motion. In fact, the situation is somewhat better, because Brownian motion on $\bS^2$ started from a point other than $S$ a.s.\ never hits $S$, and thus the $\operatorname{SU}(2)$ Brownian motion stays within $\operatorname{SU}(2)\setminus\tilde{S}$, which is the domain of the cylindrical coordinates, for all time, a.s. Thus, Brownian motion on $\operatorname{SU}(2)$ started anywhere other than $\tilde{S}$, written as $(r_0,\theta_0,z_0)$, is given by Brownian motion on $\bS^2$ from $(r_0,\theta_0)$, lifted by letting $z_t$ be the stochastic area starting from $z_0$. In particular,
$\operatorname{SU}(2)$-Brownian motion is given, in cylindrical coordinates, by the following system of SDEs,
\[
\begin{split}
dr_t &= dW_t^{(1)} +\frac{1}{2}\cot(r_t)\, dt \\
d\theta_t &= \frac{1}{\sin(r_t)}\, dW_t^{(2)} \\
dz_t&=  \tan\lp\frac{r_t }{2}\rp\, dW_t^{(2)} \, \mod 4\pi
\end{split}
\]
where $W_t^{(1)}$ and $W_t^{(2)}$ are independent one-dimensional Brownian motions. Further, we see that $z_t$ can be written as
\begin{equation}\label{Eqn:SUTime}
z_t = z_0+W_{\int_0^t \tan^2\lp\frac{r_s}{2}\rp \, ds}  \mod 4\pi
\end{equation}
where $W_t$ is a one-dimensional Brownian motion independent of $r_t$ (and $r_t$ is the radial process on the 2-sphere).

As indicated, Brownian motion on $\operatorname{SU}(2)$, and specifically, our construction of a coupling of Brownian motions, could be studied entirely in cylindrical coordinates, ignoring the missing fiber $\tilde{S}$. However, this not completely satisfactory, since it leaves the global geometry unclear; it's really the behavior at $\tilde{S}$ that gives the bundle its non-trivial topology.  Indeed, restricted to $\operatorname{SU}(2)\setminus \tilde{S}$, we can lift $z_t$ to $\bR$, just by passing to the universal cover of each fiber, or more coarsely, dropping the ``mod $4\pi$.'' In fact, this version of the stochastic area process was studied in \cite{FabriceWangArea}, which we discuss briefly below. However, this lift cannot extend to all of $\operatorname{SU}(2)$. This is related to the fact that $\operatorname{SU}(2)$ is already a simply-connected Lie group, unlike $\operatorname{SL}(2)$, so there is no further universal cover globally. More concretely, a curve on $\bS^2$ doesn't have a well-defined (real-valued) notion of enclosed area (say, that behaves well under deformation), which can be seen in part by the fact that there's no way to distinguish the area inside a simple closed curve from the area outside. However, taking the area modulo $4\pi$ resolves this issue. Motivated by these considerations, we give global descriptions of our construction in terms of the identification of $\operatorname{SU}(2)$ with $\{x_1^2+x_2^2+x_3^2+x_4^2=1\}\subset \bR^4$ in addition to the description in cylindrical coordinates on $\operatorname{SU}(2)\setminus\tilde{S}$.

\begin{remark}\label{Rmk:Hopf}
Following up the above and on Remark \ref{Rmk:Norm}, we note that it is common to choose the basis ${X,Y,Z}$ for the Lie algebra without the ``$\frac{1}{2}$'s'' in front of the matrices. This choice corresponds to the standard metric version of the Hopf fibration
\[
\bS^1\lp 1\rp \rightarrow \bS^3\lp 1\rp \rightarrow \bS^2\lp \frac{1}{2}\rp ,
\]
where $\bS^n(r)$ denotes the sphere of radius $r$. However, it means the submersion is onto the 2-sphere with curvature 4. The normalization we chose above rescales everything so that the submersion is onto $\bS^2\lp 1\rp$, the 2-sphere with curvature 1, which is natural for working in cylindrical coordinates. Our normalization does mean that, if we extend the sub-Riemannian metric to a Riemannian metric by declaring $\{X,Y,Z\}$ to be an orthonormal frame, then we get $\bS^3(2)$, that is, the round 3-sphere with constant (sectional) curvature 4.
\end{remark}

\subsection{The vertical coupling}
Just as for $\operatorname{SL}(2)$, we start at a Brownian motion at $(0,0,0)$ and we want to reflect once $z_t$ hits either $a$ or $a-2\pi$. This makes sense in cylindrical coordinates, but we wish to give a global description. To this end, let $ \dist_{\sR}$ denote the sub-Riemannian distance associated to the structure just described, and let $\dist_{\bS^3}$ denote the usual Riemannian distance on $\bS^3$.

The point is that the surface in  $\operatorname{SU}(2)\cong \bS^3\subset \bR^4$ corresponding to (the closure of) the set where $z= a \text{ or } a-2\pi$ in cylindrical coordinates is given by a great sphere in $\bS^3$ (that is, an equator). Moreover, that great sphere is equidistant from $(0,0,0)$ and $(0,0,2a)$, with respect to both the sub-Riemannian distance and the Riemannian distance. We make this precise as follows.

\begin{Lemma}\label{Lem:Hemi}
Consider two (distinct) points of $\operatorname{SU}(2)$ on the same Hopf fiber; without loss of generality, they may be taken as $q=(0,0,0)$ and $\tilde{q}=(0,0,2a)$ in cylindrical coordinates, for some $a\in(0,\pi]$. Then the set of equidistant points $S_a$ between $q$ and $\tilde{q}$ is the same under both the sub-Riemannian metric and the spherical Riemannian metric, that is
\[\begin{split}
S_a &= \lc p\in \operatorname{SU}(2) : \dist_{\sR}\lp (0,0,0), p\rp = \dist_{\sR}\lp (0,0,2a), p\rp \rc \\ 
&= \lc p\in \operatorname{SU(2)} : \dist_{\bS^3}\lp (0,0,0), p\rp = \dist_{\bS^3}\lp (0,0,2a), p\rp \rc ,
\end{split}\]
and moreover, $S_a$ is a great sphere in $\bS^3$, given as the intersection of $x_1^2+x_2^2+x_3^2+x_4^2=1$ with the hyperplane through the origin having normal vector $-\sin\lp \frac{a}{2}\rp \partial_{x_1}  +\cos\lp \frac{a}{2}\rp \partial_{x_2}$. The restriction of $S_a$ to $\operatorname{SU}(2)\setminus\tilde{S}$ is given in cylindrical coordinates by $\{z=a\}\cup\{z=a-2\pi\}$.
\end{Lemma}

\begin{proof}
As we see from the description of cylindrical coordinates above, for any fixed $a$, the set of points in $\operatorname{SU}(2)$ equidistant from $(0,0,0)$ and $(0,0,2a)$ in the sub-Riemannian metric consists of the points of the form $\lp r, \theta, a\rp$ and $\lp r, \theta, a-2\pi\rp$, at least on $\operatorname{SU}(2)\setminus\tilde{S}$. By the continuity of the sub-Riemannian distance, the extension to the fiber over the south pole consists in taking the closure of these two open surfaces. To describe this concretely, we see that the points of the form $\lp r, \theta, a\rp$ in cylindrical coordinates correspond to the points
\begin{equation}\label{Eqn:UpperGreatSphere}
 \lp \cos \lp\frac{r}{2}\rp \cos \lp\frac{a}{2}\rp, \cos \lp\frac{r}{2}\rp \sin \lp\frac{a}{2}\rp,  
 \sin \lp\frac{r}{2}\rp \cos \lp\theta-\frac{a}{2}\rp, \sin \lp\frac{r}{2}\rp \sin \lp\theta-\frac{a}{2}\rp \rp
\end{equation}
in $\bS^3\subset \bR^4$. Note that, for fixed $a$ (and $(r,\theta)\in[0,\pi]\times[0,2\pi]$), all such points lie in the hyperplane through the origin perpendicular to the vector $N_a=-\sin \lp\frac{a}{2}\rp\partial_{x_1} + \cos \lp\frac{a}{2}\rp\partial_{x_2}$, and thus all such points lie in the great sphere described in the lemma. Moreover, for $r\in[0,\pi)$ (and $\theta\in[0,2\pi)$), which is the domain of the cylindrical coordinates, we see that these points all have positive inner product with respect to the vector $\hat{N}_a=\cos\lp\frac{a}{2}\rp\partial_{x_1}+\sin\lp\frac{a}{2}\rp\partial_{x_2}$, and in fact parametrize the corresponding open hemisphere of the great sphere. Since $\sin(x-\pi)=-\sin x$ and $\cos(x-\pi) = -\cos x$, we see that points of the form $\lp r, \theta, a-2\pi \rp$ in cylindrical coordinates give the negatives of the points \eqref{Eqn:UpperGreatSphere}, and thus they parametrize the opposite open hemisphere of the great sphere perpendicular to $N_a$ (that is, the hemisphere of points having negative inner product with respect to $\hat{N}_a$). Letting $r=\pi$ gives the points
\[
\lp 0,0, \cos \lp\theta-\frac{a}{2}\rp, \sin \lp\theta-\frac{a}{2}\rp \rp
\]
which parametrize the equator of the great sphere (given as the points of the great sphere perpendicular to $\hat{N}_a$). These points are the closure of either of the open hemispheres just mentioned (note that replacing $a$ with $a-2\pi$ gives the same set of points as $\theta$ ranges from 0 to $2\pi$), and we also see that these points are exactly the Hopf fiber over the south pole. This last fact agrees with the fact that every point on the Hopf fiber over the north pole is the same distance (namely $\pi$) from every point on the Hopf fiber over the south pole.

Summarizing, we see that $S_a$ is exactly the great sphere perpendicular to $N_a$ as claimed, and that the intersection of $S_a$ with the domain of the cylindrical coordinate system is given by the points $\{z=a\}$ and $\{z=a-2\pi\}$. It remains to see that $S_a$ is also the set of points equidistant to $(0,0,0)$ and $(0,0,2a)$ under the Riemannian metric.

Note that the reflection through the plane normal to $N_a$ (and thus the reflection of $\bS^3$ through $S_a$) is given in Cartesian coordinates by
\[
\begin{bmatrix}
\cos a & \sin a &0&0 \\
\sin a & -\cos a  &0&0 \\
0&0 &1&0 \\
0&0 &0&1 
\end{bmatrix} .
\]
Since the points $(r,\theta,z)=(0,0,0)$ and $(r,\theta,z)=(0,0,2a)$ in cylindrical coordinates correspond to the points $\lp x_1,x_2,x_3,x_4\rp= (1,0,0,0)$ and $\lp x_1,x_2,x_3,x_4\rp= (\cos a,\sin a,0,0)$, respectively, in Cartesian coordinates for $\bR^4$, we see that these points are interchanged under this reflection. (Working entirely in cylindrical coordinates, this reflection takes the point $(r,\theta,z)$ to the point $(r,\theta+a-z,2a-z)$, and a second application takes this point back to $(r,\theta,z)$, of course. In particular, it takes $(0,0,0)$ to $(0,a,2a)$, which is identified with $(0,0,2a)$ in cylindrical coordinates.) It immediately follows that $S_a$ is also the set of points equidistant to $(0,0,0)$ and $(0,0,2a)$ under the Riemannian metric on $\bS^3$, completing the proof. 
\end{proof}

Because $S_a$ is a great sphere in $\operatorname{SU}(2)\cong \bS^3$, it cuts $\operatorname{SU}(2)$ into two open hemispheres. We let $S_a^-$ denote the open hemisphere consisting of points that have negative inner product with respect to $N_a$, which is also the hemisphere containing $(0,0,0)$. Similarly, we let $S_a^+$ denote the open hemisphere consisting of points that have positive inner product with respect to $N_a$, which is also the hemisphere containing $(0,0,2a)$. With this notation in place, we can state our main result for vertical reflection coupling in $\operatorname{SU}(2)$.

\begin{THM}\label{THM:SU2}
Consider two (distinct) points of $\operatorname{SU}(2)$ on the same Hopf fiber; without loss of generality, they may be taken as $(0,0,0)$ and $(0,0,2a)$ in cylindrical coordinates, for some $a\in(0,\pi]$. Let $B_t=\lp r_t,\theta_t,z_t\rp$ be an $\operatorname{SU}(2)$-Brownian motion started from $(0,0,0)$, let $\sigma_a$ be the first hitting time of $S_a$. Let $\tilde{B}_t=\lp \tilde{r}_t,\tilde{\theta}_t,\tilde{z}_t\rp$ be given by
\[
\lp\tilde{r}_t,\tilde{\theta}_t,\tilde{z}_t\rp = \begin{cases}
\lp R_{\lp r_{\sigma_a} ,\theta_{\sigma_a}\rp} (r_t,\theta_t), 2a-z_t\rp  & \text{for $t\leq \sigma_a$} \\
(r_t,\theta_t,z_t)  & \text{for $t> \sigma_a$}
\end{cases}
\quad \text{with $z_t$ and $\tilde{z}_t$ understood modulo $4\pi$}
\]
(where $R_{(r,\theta)}$ is the reflection of $\bS^2$ as in Lemma \ref{Lem:BM}). Then $\tilde{B}_t $ is an $\operatorname{SU}(2)$-Brownian motion started from $(0,0,2a)$ coupled with $B_t$ in such a way that the coupling time is equal to $\sigma_a$. Further, $\sigma_a$ is a.s.\ finite (so that the coupling is successful), this coupling is maximal, and the diffusion satisfies a reflection principle with respect to $S_a$, so that
\[
\Prob\lp \sigma_a>t \rp = 1-2\Prob\lp B_t\in S_A^+\rp =1-2\Prob\lp z\lp B_t\rp \in \lb -2\pi, a-2\pi \rp\cup \lp a, 2\pi\rb\rp.
\]
\end{THM}

\begin{remark}
For any $a\in(0,\pi]$ and $b\in[0,2\pi)$, consider the map given in cylindrical coordinates by $(r,\theta,z)$ to $(r,2b-\theta,2a-z)$ and denote it by $T_{a,b}$. If $b=\theta_{\sigma_a}$, this is the map used to construct $\tilde{B}_t$, so that
\[
\tilde{B}_t=T_{a,\theta_{\sigma_a}} B_t \quad \text{on $\{t\leq \sigma_a\}$.}
\]
Writing $T_{a,b}$ in Cartesian coordinates on $\bR^4$, we see that it is given by the matrix
\[
\begin{bmatrix}
\cos a & \sin a &0&0 \\
\sin a & -\cos a  &0&0 \\
0&0 &\cos (2b-a)&\sin (2b-a) \\
0&0 &\sin (2b-a) &-\cos (2b-a)
\end{bmatrix} .
\]
This is easily seen to be an element of $\operatorname{SO}(4)$, the square of which is the identity. Moreover, a modest computation verifies that it preserves the Hopf fibration. So $T_b$ is a (globally well-defined) sub-Riemannian isometry of $\operatorname{SU}(2)$, consistent with the fact that it maps an $\operatorname{SU}(2)$-Brownian motion to another $\operatorname{SU}(2)$-Brownian motion (that is, potentially started from a different initial point). Of course, $T_{a,\theta_{\sigma_a}}$ corresponds to $I_{B_{\sigma_a}}$ in Theorem \ref{THM:SummaryCoupling}.
\end{remark}

\begin{proof}
The proof is almost identical to that of Theorem \ref{THM:SLMax}. The reflection on fibers works the same way as there, so in light of Lemma \ref{Lem:BM} (plus Remark \ref{Rmk:Sigma}), Lemma \ref{Lem:Hemi}, and the geometry of the cylindrical coordinates above, we see that $\tilde{B}_t$ is an $\operatorname{SU}(2)$-Brownian motion that couples with $B_t$ at $\sigma_a$. Prior to $\sigma_a$, $B_t$ is in $S_a^-$ and $\tilde{B}_t$ is in $S_a^+$, and after $\sigma_a$, $B_t$ is equally likely to be in either, by the strong Markov property and \eqref{Eqn:SUTime} (or Lemma \ref{Lem:Hemi})). This gives the reflection principle, and maximality follows by letting $U=S_a^-$ in \eqref{Eqn:Max}.
\end{proof}

Using the compactness of the fiber, we have the direct analogue of Theorem \ref{THM:SVert}, with essentially the same proof.

\begin{THM}\label{THM:SUVert}
Let everything be as in Theorem \ref{THM:SU2}. Then there exist constants $C>0$, $c>0$, and $T_0>0$ such that, for any such points $q$ and $\tilde{q}$ (equivalently for any $a\in (0,4\pi)$), 
\[
\Prob\lp \sigma_a>t \rp =\TVd \lp \Law\lp B_t\rp, \Law\lp \tilde{B}_t\rp \rp \leq Ce^{-ct} \quad \text{for all $t>T_0$.}
\]
\end{THM}

\subsection{Horizontal coupling on $\operatorname{SU}(2)$}
Again, for points not on the same fiber, we use a two-stage coupling. If $q$ and $\tilde{q}$ of $\operatorname{SU}(2)$ are any two points on different fibers, we first consider their projections onto $\bS^2$ under the submersion, and let $2r$ be the distance between them (as points on $\bS^2$). We couple them via the usual (Markov) reflection coupling on $\bS^2$, with the $z$-processes being whatever is induced by the joint evolution of the $\bS^2$-marginals, until the first time the $\bS^2$-marginals meet. We then run the vertical coupling just developed, starting from the random vertical displacement coming from the first stage.

\begin{THM}
Consider the two-stage coupling just described. There exist constants $K>0$, $k>0$, and $T_0>0$ such that, for any two points $q$ and $\tilde{q}$ of $\operatorname{SU}(2)$, the coupling time $\tau$ of Brownian motions $B_t$ and $\tilde{B}_t$ from $q$ and $\tilde{q}$ respectively, under the two-stage coupling, satisfies
\[
\Prob\lp \tau>t \rp < K e^{-kt} \quad \text{for all $t>T_0$.}
\]
\end{THM}

\begin{proof}
It is standard that, for any two points of $\bS^2$, the coupling time $\sigma'_r$ (under reflection coupling) decays exponentially. Indeed, by smoothness and compactness, the heat kernel at time one, $p^{\bS^2}_1(x,y)$, is bounded from below by a positive constant independent of $x,y\in \bS^2$. Then the exponential decay follows as in the proof of Theorem \ref{THM:SVert}. Alternatively, the reflection coupling on $\bS^2$ is maximal (on $\bS^2$), as discussed in Section \ref{Sect:Intro1}, and considering the eigenfunction expansion of the heat kernel, it follows that the coupling time has to decay exponentially with a rate given by the spectral gap of $\bS^2$.

Further, the second stage of the coupling is conditionally independent of the first given the vertical displacement when the $\bS^2$-marginals couple, which we can write as $2a\lp \sigma'_r \rp$. Thus $\tau$ is equal in distribution to $\sigma'_r +\sigma_{a\lp \sigma'_r \rp}$, and since each of these satisfies a uniform exponential decay bound, so does their sum.
\end{proof}

\begin{remark}
We note that \cite{FabriceMichel} computes the sharp exponential decay rate using spectral methods, in the context of establishing functional inequalities on $\operatorname{SU}(2)$. Adjusting for their normalization, we see that their results give the optimal asymptotic decay rate (up to a constant factor) as $e^{-\frac{1}{4}t}$. Note that the decay rate for the reflection coupling on $\bS^2$ is $e^{-t}$, which is better than the the two-stage coupling on the total space. This means that, in contrast to the Heisenberg groups and $\operatorname{SL}(2)$, the vertical coupling is slower than the horizontal coupling, which is presumably a consequence of the positive curvature of $\bS^2$.

As noted in Section \ref{Sect:ResSum}, it would be desirable to compute the sharp exponential rate in Theorem \ref{THM:SUVert}, and then the joint distribution of $\sigma'_r$ and $a\lp \sigma'_r \rp$, analogously to \cite{BGM}, in order to recover, and possibly refine, the result of \cite{FabriceMichel}.
\end{remark}

\begin{remark}
Looking back, it's clear that the sets $\{z=a\}$ (and $\{z=a\}\cup\{z=a-2\pi\}$ in the case of $\operatorname{SL}(2)$) give the set of equidistant points between $(0,0,0)$ and $(0,0,2a)$ in the all of the cases previously considered, and the vertical coupling is successful when one (hence both) of the particles hits this submanifold of equidistant points between the starting points. In the previous cases, the existence of global coordinates made such an invariant description unnecessary, but for $\operatorname{SU}(2)$, it is the natural way to characterize the vertical reflection coupling. Moreover, describing the vertical reflection coupling as ``coupling upon the first hitting time of the equidistant sub-manifold'' further highlights the analogy with the Markovian maximal reflection couplings in the Riemannian case.
\end{remark}


\section{Non-isotropic Heisenberg groups}

\subsection{Preliminaries}

We consider a family of non-isotropic Heisenberg groups of a symplectic space $\left( \mathbb{R}^{2n}, \omega \right)$ defined as follows.

\begin{definition}\label{defn.NonisotropicHeisenbergGroup}
A \emph{non-isotropic Heisenberg group} $\mathbb{H}^{n}_{\omega}$ is the set $\mathbb{R}^{2n}\times\mathbb{R}$ equipped with the group law given by

\begin{align}\label{GroupLaw}
& \left( \mathbf{v}, z \right)\star \left( \mathbf{v}^{\prime}, z^{\prime} \right)=\left(\mathbf{v}+\mathbf{v}^{\prime},z+z^{\prime}+\frac{1}{2}\omega\left( \mathbf{v}, \mathbf{v}^{\prime} \right)\right),
\\
& \mathbf{v}=\left( x_{1},y_{1},\cdots,x_{n},y_{n}  \right), \mathbf{v}^{\prime}=\left( x_{1}^{\prime},y_{1}^{\prime},\cdots,x_{n}^{\prime}, y_{n}^{\prime} \right) \in \mathbb{R}^{2n},
\notag
\\
& \omega: \mathbb{R}^{2n} \times \mathbb{R}^{2n} \longrightarrow \mathbb{R},
\notag
\end{align}
where $\omega\left( \mathbf{v}, \mathbf{v}^{\prime} \right):=\sum_{i=1}^{n} \alpha_i\left(x_{i}y_{i}^{\prime}-x_{i}^{\prime}y_{i}\right)$ is a  \emph{symplectic form} on $\mathbb{R}^{2n}$ and $\alpha_{1}, \alpha_{2}, \cdots, \alpha_{n}$ are positive constants indexed in such a way that
\begin{align*}
0<\alpha_{1}\leq \alpha_{2}\leq\cdots\leq \alpha_{n}.
\end{align*}
\end{definition}
Note that any non-degenerate symplectic form on $\mathbb{R}^{2n}$, that is, a bilinear anti-symmetric form, can be written as a sum of symplectic forms on $\mathbb{R}^{2}$, as was described in  \cite[Appendix]{GordinaLuo2022}. In particular, this explains why such groups are referred to as non-isotropic.

If $\alpha_{1}=\cdots=\alpha_{n}=1$, we get the standard $2n+1$-dimensional Heisenberg group. Sometimes the parametrization $\alpha_{1}=\cdots=\alpha_{n}=4$ is used for the standard Heisenberg group. These are all isotropic Heisenberg groups referring to the fact that the corresponding symplectic space is isotropic as described in \cite[Appendix]{GordinaLuo2022}.

Consider the following left-invariant vector fields on $\mathbb{H}^{n}_{\omega}$ identified with differential operators on $\mathbb{R}^{2n+1}$ by
\begin{align}\label{e.CanonicalBasis}
& X_i^{\omega}\left( g \right)=\partial _{x_i}-\frac{\alpha_i}{2}y_i\partial_{z}, \, Y_i^{\omega}\left( g \right)=\partial_{y_i}+\frac{\alpha_i}{2}x_i\partial_{z}, \hskip0.1in i=1,\cdots,n, \, Z^{\omega}\left( g \right)=\partial_{z} \notag
\end{align}
for any $g=(x_{1},y_{1},\cdots,x_{n},y_{n},z)\in\mathbb{H}^{n}_{\omega}$. Note that the only non-zero Lie brackets for  left-invariant vector fields $X_i^{\omega}$ and $Y_i^{\omega}$ are
\[
[X_i^{\omega},Y_i^{\omega}]=\alpha_i Z^{\omega},  i=1,...,n,
\]
so the vector fields $\{X_i^{\omega}, Y_i^{\omega}, i=1, \cdots, n \}$ and their Lie brackets span the tangent space at every point, and therefore H\"{o}rmander's condition is satisfied.

This implies  that the group $\mathbb{H}^{n}_{\omega}$ has a natural sub-Riemannian structure $\left(\mathbb{H}^{n}_{\omega}, \mathcal{H}^{\omega}, \langle \cdot, \cdot \rangle^{\omega}_{\mathcal{H}}\right)$, where
\[
\mathcal{H}^{\omega}=\mathcal{H}_{g}^{\omega}= \operatorname{Span}\{X_i^{\omega} \left( g \right), Y_i^{\omega}\left( g \right), i=1,\cdots,n\}
\]
is the \emph{horizontal distribution} and the left-invariant inner product $\langle\cdot,\cdot\rangle_{\mathcal{H}^{\omega}}$ is chosen in such a way that  $\{X_i^{\omega},Y_i^{\omega}:i=1,\cdots,n\}$ is an orthonormal frame for the sub-bundle $\mathcal{H}^{\omega}$. Note that both the vector space $\mathcal{H}^{\omega}_g$ and the left-invariant sub-Riemannian metric $\langle \cdot, \cdot \rangle^{\omega}_{\mathcal{H}}=\langle \cdot, \cdot \rangle^{\omega}_{\mathcal{H}^{\omega}}$ depend on the symplectic form $\omega$.

By the classical result in \cite{Hormander1967a} H\"{o}rmander's condition implies that the \emph{sub-Laplacian}
\[
\Delta_{\mathcal{H}}^{\omega}:=\sum_{i=1}^{n} \left(\left(X_i^{\omega}\right)^2+\left(Y_i^{\omega}\right)^2\right)
\]
is a hypoelliptic operator.

There is a natural sub-Riemannian distance $\dist_{sR}$ on $\mathbb{H}^n_{\omega}$ (see \cite[Section 5.2]{BonfiglioliLanconelliUguzzoniBook} for more details). In this case, $\dist_{sR}$ is a left-invariant metric on $\mathbb{H}^n_{\omega}$ and $\dist_{sR}\lp e,\cdot\rp=\dist_{sR}\lp \cdot,e\rp:\mathbb{H}^n_{\omega} \rightarrow [0,+\infty)$ is a homogeneous norm (see \cite[Definition 5.1.1]{BonfiglioliLanconelliUguzzoniBook} for precise definition). By \cite[Proposition 5.1.4]{BonfiglioliLanconelliUguzzoniBook}, all homogeneous norms on $\mathbb{H}^n_{\omega}$ are equivalent. Here for convenience, we will work with another homogeneous norm instead of $\dist_{sR}\lp e,\cdot\rp=\dist_{sR}\lp \cdot,e\rp$.

Define $\Vert \cdot \Vert_{\omega}:\mathbb{H}^n_{\omega} \rightarrow [0,+\infty)$ by
\begin{align*}
\Vert g \Vert_{\omega} =\Vert \lp x_{1},y_{1},\cdots,x_{n},y_{n},z \rp \Vert_{\omega} =\lp \sum_{i=1}^n \lp \lp x_i\rp^2+ \lp y_i\rp^2\rp + \vert z \vert \rp^{\frac{1}{2}}
\end{align*}
for any $\lp x_{1},y_{1},\cdots,x_{n},y_{n},z \rp\in \mathbb{H}^n_{\omega}$.

\begin{proposition} \label{prop.HomogeneousNorm}
The continuous function $\Vert \cdot \Vert_{\omega}:\mathbb{H}^n_{\omega} \rightarrow [0,+\infty)$ is a homogeneous norm on $\mathbb{H}^n_{\omega}$.
\end{proposition}

\begin{proof}
We check that $\Vert \cdot \Vert_{\omega}$ satisfies the definition of the homogeneous norm. For $g \neq e$, it is clear that $\Vert g \Vert_{\omega}>0$. For any $\lambda>0$ and any $\lp x_1,y_1,\ldots,x_n, y_n,z\rp \in \mathbb{H}^n_{\omega}$, we have $\Vert \lambda g \Vert_{\omega}=\Vert \lp\lambda x_1,\lambda y_1,\ldots,\lambda x_n,\lambda y_n,\lambda^2 z\rp \Vert_{\omega}=\lambda\Vert  g \Vert_{\omega}$. Also, for any $\lp x_1,y_1,\ldots,x_n, y_n,z\rp \in \mathbb{H}^n_{\omega}$, we have $g^{-1}=\lp -x_1,-y_1,\ldots,-x_n, -y_n,-z\rp $ and thus $\Vert  g^{-1} \Vert_{\omega}=\Vert  g \Vert_{\omega}$. Therefore, $\Vert \cdot \Vert_{\omega}$ is a homogeneous norm.
\end{proof}

\subsection{Brownian motion on $\mathbb{H}^{n}_{\omega}$ and its associated vertical process}

We first give the explicit expression of the $\mathbb{H}^n_{\omega}$-Brownian $B_t$ generated by $\frac{1}{2}\Delta_{\mathcal{H}}^{\omega}$.

\begin{lemma}
The $\mathbb{H}^n_{\omega}$-Brownian $B_t$ generated by $\frac{1}{2}\Delta_{\mathcal{H}}^{\omega}$ started from $g_0=\lp x_{1},y_{1},\cdots,x_{n},y_{n},z\rp\in \mathbb{H}^{n}_{\omega}$ has the form $\lp x^1_t,y^1_t,\cdots,x^n_t,y^n_t,z_t\rp$, where $\lp x^1_t,y^1_t,\cdots,x^n_t,y^n_t\rp$ is a Brownian motion on $\bR^{2n}$ started from $\lp x_{1},y_{1},\cdots,x_{n},y_{n}\rp$ and $z_t$ has the form
\[
z_t=z+\frac{1}{2}\sum_{i=1}^n\alpha_i\int_0^{t} \lp x^i_{s}dy^i_{s}-y^i_{s}dx^i_{s}\rp.
\]
\end{lemma}

\begin{proof}
The Brownian $B_t$ generated by $\frac{1}{2}\Delta_{\mathcal{H}}^{\omega}$ satisfies
\[
dL_{g_0}\lp dB_t\rp=\lp dx^1_t,dy^1_t,\cdots,dx^n_t,dy^n_t,0\rp
\]
where $L_g:\mathbb{H}^n_{\omega} \rightarrow \mathbb{H}^n_{\omega}$ is the left translation and $\lp x^1_t,y^1_t,\cdots,x^n_t,y^n_t\rp$ is a Brownian motion started from $\lp x_{1},y_{1},\cdots,x_{n},y_{n}\rp$ on $\bR^{2n}$. Making use of the explicit group law \eqref{GroupLaw} of $\mathbb{H}^n_{\omega}$, we can solve the above system of SDEs, which yields our desired result.
\end{proof}

Throughout this section, we will call the process $z_t$ the \emph{vertical process} for simplicity. Here we list some properties of it.

\begin{lemma}\label{Lem:AreaProcess.TimeChangeOfBM} 
The vertical process $z_t$ has the same distribution as $z+W_{\int_{0}^t \frac{1}{4}\lp r^{\omega}_s\rp^2ds}$, where $\lp W_t\rp_{t\geq 0}$ is a standard Brownian motion and the process $\lp r^{\omega}_t\rp_{t\geq 0}$ has the form 
\begin{align*}
r^{\omega}_t=\sqrt{\sum_{i=1}^n \lp \alpha_i\rp^2\lp \lp x^i_{t}\rp^2+\lp y^i_{t}\rp^2\rp}.
\end{align*}
\end{lemma}

\begin{proof}
By It\^o's formula,
\begin{align*}
dz_t=\frac{1}{2}r^{\omega}_t\cdot \frac{\sum_{i=1}^n\alpha_i \lp x^i_{t}dy^i_{t}-y^i_{t}dx^i_{t}\rp}{r^{\omega}_t}.
\end{align*}
Let $W^{\omega}_t=\int_0^t\frac{\sum_{i=1}^n\alpha_i \lp x^i_{s}dy^i_{s}-y^i_{s}dx^i_{s}\rp}{r^{\omega}_s}$. We can see that $\langle W^{\omega}_t \rangle_t=t$, so $\lp W^{\omega}_t\rp_{t\geq 0}$ is a standard Brownian motion, which is independent of $\omega$, and we will denote $\lp W^{\omega}_t\rp_{t\geq 0}$ by $\lp W_t\rp_{t\geq 0}$ instead. Therefore, the desired result follows.
\end{proof}

We denote the density of the law of the vertical process $z_t$ with respect to the Lebesgue measure on $\bR$ by $f_t^{\omega}$ and give the following estimate.

\begin{lemma} \label{Lem:eqn.DensityUpperBound2}
For any $z\in \bR$ and $n\geq 2$, we have
\begin{align} \label{eqn.DensityUpperBound2}
f_t^{\omega}(z) \leq \frac{1}{\alpha_n t}.
\end{align}
\end{lemma}

\begin{proof}
Note that the $z^i_{t}$ are i.i.d., each with the density \eqref{Eqn:f_t}, and the scaling relationship implies that the density $f_t^{\alpha_i}$ of $\alpha_iz^i_{t}$ for $i=1,\cdots,n$ has the form
\[
f_t^{\alpha_i}(z)=f_{\alpha_i t}(z) = \frac{1}{\alpha_i t} \sech\left(\frac{\pi z}{\alpha_i t}\right).
\]
Then the density $f_t^{\omega}$ of the area process $z_t=\sum_{i=1}^n\alpha_iz^i_{t}$ can be written as a convolution of all $f_t^{\alpha_i}$ for $i=1,\cdots,n$ as below 
\begin{align*}
f_t^{\omega}(z)= \lp f_t^{\alpha_1} * f_t^{\alpha_2} * \cdots * f_t^{\alpha_n} \rp (z).
\end{align*}
Since the convolution operator is commutative, we have
\begin{align*}
f_t^{\omega}(z)=\lp f_t^{\alpha_n} * f_t^{\alpha_{n-1}} * \cdots * f_t^{\alpha_1} \rp (z).
\end{align*}

Now we prove \eqref{eqn.DensityUpperBound2}. We know that $\int_{-\infty}^{\infty} \sech\left(\frac{\pi z}{\alpha_it}\right)dz=\alpha_i t$ and $0< \sech\left(\frac{\pi z}{\alpha_it}\right)\leq 1$ for $i=1,\cdots,n$. When $n=2$, we have
\begin{align*}
f_t^{\omega}(z) &
=\frac{1}{\alpha_1\alpha_2t^2}\int_{-\infty}^{\infty} \sech\left(\frac{\pi \tau}{\alpha_2t}\right) \sech\left(\frac{\pi (z-\tau)}{\alpha_1t}\right)d\tau
\\
&
\leq \frac{1}{\alpha_1\alpha_2t^2} \int_{-\infty}^{\infty} \sech\left(\frac{\pi (z-\tau)}{\alpha_1t}\right)d\tau
\\
&
=\frac{1}{\alpha_1\alpha_2t^2} \cdot \alpha_1 t=\frac{1}{\alpha_2 t}
\end{align*}
where the second equality is by using the change of variable. When $n>2$, the similar idea still applies as below
\begin{align*}
f_t^{\omega}(z) & =\frac{1}{\alpha_1\ldots \alpha_nt^{n}} \int_0^{\infty} \ldots \int_0^{\infty} \sech\left(\frac{\pi \tau_{n-1}}{\alpha_nt}\right) \Pi_{i=2}^{n-1}\sech\left(\frac{\pi (\tau_{i-1}-\tau_i)}{\alpha_{i-1}t}\right) \sech\lp \frac{\pi\lp z-\tau_1\rp}{\alpha_1t}\rp d\tau_1\ldots d\tau_{n-1}
\\
&
\leq \frac{1}{\alpha_1\ldots \alpha_nt^{n}} \int_0^{\infty} \ldots \int_0^{\infty} \Pi_{i=2}^{n-1}\sech\left(\frac{\pi (\tau_{i-1}-\tau_i)}{\alpha_{i-1}t}\right) \sech\lp \frac{\pi\lp z-\tau_1\rp}{\alpha_1t}\rp d\tau_1\ldots d\tau_{n-1}
\\
&
=\frac{1}{\alpha_1\ldots \alpha_nt^{n}} \cdot \lp \alpha_1\ldots\alpha_{n-1}t^{n-1}\rp=\frac{1}{\alpha_n t}.
\end{align*}
Therefore, \eqref{eqn.DensityUpperBound2} holds.
\end{proof}

\subsection{Horizontal coupling}

We take two $\mathbb{H}^n_{\omega}$-Brownian motions $B_t=(x^1_t,y^1_t,\ldots,x^n_t,y^n_t,z_t)$ and $\widetilde{B}_t=(\widetilde{x}^1_t,\widetilde{y}^1_t,\ldots,\widetilde{x}^n_t,\widetilde{y}^n_t,\widetilde{z}_t)$ started from two points $(x_1,y_1,\ldots,x_n,y_n,z)$ and $(\widetilde{x}_1,\widetilde{y}_1,\ldots,\widetilde{x}_n,\widetilde{y}_n,\widetilde{z})$ respectively. Since the L\'evy stochastic area is invariant under rotations of coordinates on each copy of $\bR^2$, it suffices to consider the case when $\widetilde{x}_i=-x_i$ and $\widetilde{y}_i=y_i$ for $i=1,\cdots,n$.

\begin{assumption}
Throughout the rest of the paper, we assume that $\widetilde{x}_i=-x_i$ and $\widetilde{y}_i=y_i$ for $i=1,\cdots,n$.
\end{assumption}

We first couple the ``horizontal'' parts of $B_t$ and $\widetilde{B}_t$, that is, we couple $(x^1_t,y^1_t,\ldots,x^n_t,y^n_t)$ and $(\widetilde{x}^1_t,\widetilde{y}^1_t,\ldots,\widetilde{x}^n_t,\widetilde{y}^n_t)$ on $\bR^{2n}$. To do so, we take a mirror coupling of $(x^1_t,y^1_t,\ldots,x^n_t,y^n_t)$ and $(\widetilde{x}^1_t,\widetilde{y}^1_t,\ldots,\widetilde{x}^n_t,\widetilde{y}^n_t)$ in the sense of Kendall and Cranston. That is, reflecting $(x^1_t,y^1_t,\ldots,x^n_t,y^n_t)$ through a hyperplane through the origin with the unit normal vector given by $\Vec{n}=\frac{1}{L}\left(x_1,0,\ldots,x_n,0\right)$ with $L=\sqrt{\sum_{i=1}^n\left(x_i\right)^2}$, we obtain $(\widetilde{x}^1_t,\widetilde{y}^1_t,\ldots,\widetilde{x}^n_t,\widetilde{y}^n_t)$.

\begin{lemma} \label{lem.HorizontalCoupling}
For any $t>0$, we have
\begin{align}
& \widetilde{x}^i_t=x^i_t-\frac{2x_i}{L^2} \sum_{j=1}^n x_jx_t^j, \notag
\\
&
\widetilde{y}^i_t=y_t^i \notag
\end{align}
for $i=1,\cdots,n$.
\end{lemma}

\begin{proof}
Taking a Householder transformation of $(x^1_t,y^1_t,\ldots,x^n_t,y^n_t)$ with the matrix given by $P=I-2\Vec{n}\Vec{n}^{\intercal}$, we obtain $(\widetilde{x}^1_t,\widetilde{y}^1_t,\ldots,\widetilde{x}^n_t,\widetilde{y}^n_t)$. We can see that
\begin{align*}
P=\begin{bmatrix}
1-2\left(\frac{x_1}{L}\right)^2 & 0 & -\frac{2x_1x_2}{L^2} & 0 & \cdots & -\frac{2x_1x_n}{L^2} & 0 \\
0 & 1 & 0 & 0 & \cdots & 0 & 0 \\
-\frac{2x_2x_1}{L^2} & 0 & 1-2\left(\frac{x_2}{L}\right)^2 & 0 & \cdots & -\frac{2x_2x_n}{L^2} & 0 \\
0 & 0 & 0 & 1 & \cdots & 0 & 0 \\
\vdots & \vdots & \vdots & \vdots & \ddots & \vdots & \vdots \\
-\frac{2x_nx_1}{L^2} & 0 & -\frac{2x_nx_2}{L^2} & 0 & \cdots & 1-2\left(\frac{x_n}{L}\right)^2 & 0 \\
0 & 0 & 0 & 0 & \cdots & 0 & 1
\end{bmatrix}.
\end{align*}
Using $(\widetilde{x}^1_t,\widetilde{y}^1_t,\ldots,\widetilde{x}^n_t,\widetilde{y}^n_t)^{\intercal}=P(x^1_t,y^1_t,\ldots,x^n_t,y^n_t)^{\intercal}$, we obtain the desired result.
\end{proof}

Meanwhile, we track the displacement of such two Brownian motions in the vertical direction, that is, we will track the difference between their vertical processes $z_t$ and  $\widetilde{z}_t$ when $(x^1_t,y^1_t,\ldots,x^n_t,y^n_t)$ and $(\widetilde{x}^1_t,\widetilde{y}^1_t,\ldots,\widetilde{x}^n_t,\widetilde{y}^n_t)$ are coupled.

A key object in our coupling construction  for  $\mathbb{H}^n_{\omega}$-Brownian motions is the \emph{invariant difference of stochastic areas} given by
\begin{align*}
A_t & =z_t-\widetilde{z}_t+\frac{1}{2}\sum_{i=1}^n \alpha_i \lp x^i_t\widetilde{y}^i_t-\widetilde{x}^i_ty^i_t\rp
\\
&
=z-\widetilde{z}+\frac{1}{2}\sum_{i=1}^n \alpha_i\left[ \int_0^t\lp x^i_sdy^i_s-y^i_sdx^i_s\rp- \int_0^t\lp\widetilde{x}^i_sd\widetilde{y}^i_s-\widetilde{y}^i_sd\widetilde{x}^i_s\rp\right]+\frac{1}{2}\sum_{i=1}^n \alpha_i \lp x^i_t\widetilde{y}^i_t-\widetilde{x}^i_ty^i_t\rp.
\end{align*}
Using the above lemma, we have
\begin{align}
A_t-A_0 & =\sum_{i=1}^n \alpha_i \int_0^t \left(x^i_s-\widetilde{x}^i_s\right) dy^i_s \notag
\\
&
= \frac{2\sqrt{\sum_{i=1}^n\left(\alpha_i x_i\right)^2}}{L} \int_0^t \frac{\sum_{j=1}^n x_jx_s^j}{L} d\left(\frac{\sum_{i=1}^n\alpha_i x_iy_s^i}{\sqrt{\sum_{i=1}^n\left(\alpha_i x_i\right)^2}}\right).
\end{align}
for any $t>0$.

\begin{lemma}
The stochastic processes $\frac{\sum_{j=1}^n x_jx_t^j}{L}$ and $\frac{\sum_{i=1}^n\alpha_i x_iy_t^i}{\sqrt{\sum_{i=1}^n\left(\alpha_i x_i\right)^2}}$ are two independent Brownian motions started from $L$ and $\frac{\sum_{i=1}^n \alpha_i x_iy_i}{\sqrt{\sum_{i=1}^n\left(\alpha_i x_i\right)^2}}$ respectively.
\end{lemma}

\begin{proof}
Using the fact that $\{x_t^i,y_t^i:i=1,\in n\}$ are independent Brownian motions, we have
\begin{align*}
\left \langle \frac{\sum_{j=1}^n x_jx_t^j}{L},\frac{\sum_{j=1}^n x_jx_t^j}{L} \right \rangle_t & =\frac{1}{L^2} \sum_{i,j=1}^n x_ix_j \langle x_t^i,x_t^j \rangle_t
\\
&
=\frac{1}{L^2} \sum_{i=1}^n (x_i)^2 \langle x_t^i,x_t^i \rangle_t
\\
&
=\frac{1}{L^2} \sum_{i=1}^n (x_i)^2t=t,
\end{align*}
\begin{align*}
\left\langle \frac{\sum_{i=1}^n\alpha_i x_iy_t^i}{\sqrt{\sum_{i=1}^n\left(\alpha_i x_i\right)^2}}, \frac{\sum_{i=1}^n\alpha_i x_iy_t^i}{\sqrt{\sum_{i=1}^n\left(\alpha_i x_i\right)^2}} \right\rangle_t & = \frac{1}{\sum_{i=1}^n\left(\alpha_i x_i\right)^2} \sum_{i,j=1}^n \alpha_i\alpha_jx_ix_j \langle y_t^i,y_t^j \rangle_t
\\
&
= \frac{1}{\sum_{i=1}^n\left(\alpha_i x_i\right)^2} \sum_{i=1}^n \left(\alpha_i x_i\right)^2 \langle y_t^i,y_t^i \rangle_t
\\
&
= \frac{1}{\sum_{i=1}^n\left(\alpha_i x_i\right)^2} \sum_{i=1}^n \left(\alpha_i x_i\right)^2 t=t
\end{align*}
and
\begin{align*}
\left\langle \frac{\sum_{j=1}^n x_jx_t^j}{L}, \frac{\sum_{i=1}^n\alpha_i x_iy_t^i}{\sqrt{\sum_{i=1}^n\left(\alpha_i x_i\right)^2}} \right\rangle_t & =\frac{1}{L\sqrt{\sum_{i=1}^n\left(\alpha_i x_i\right)^2}} \sum_{i,j=1}^n\alpha_ix_ix_j \langle x_t^j,y_t^i \rangle_t=0.
\end{align*}
By L\'evy's characterization of Brownian motions, we see that $\frac{\sum_{j=1}^n x_jx_t^j}{L}$ and $\frac{\sum_{i=1}^n\alpha_i x_iy_t^i}{\sqrt{\sum_{i=1}^n\left(\alpha_i x_i\right)^2}}$ are two independent Brownian motions started from $L$ and $\frac{\sum_{i=1}^n \alpha_i x_iy_i}{\sqrt{\sum_{i=1}^n\left(\alpha_i x_i\right)^2}}$ respectively.
\end{proof}

Let $T_1=\inf\{t\geq 0: x_t^i=\widetilde{x}^i_t, i=1,\cdots,n\}$ be the horizontal coupling time. By the construction of the mirror coupling (Lemma~\ref{lem.HorizontalCoupling}), we see that $T_1=\inf\{t \geq 0: \sum_{j=1}^n x_jx_s^j=0\}$. This shows that $T_1$ is the first hitting time of the level $0$ by the Brownian motion $\frac{\sum_{j=1}^n x_jx_t^j}{L}$. By a standard estimate for the first hitting time of a Brownian motion, there exists a positive constant $C$ (independent of $\omega$) such that
\begin{align} \label{ineq.HorizontalCouplingTime.TailBound}
\Prob \lp T_1>t\rp \leq \frac{C \sqrt{\sum_{i=1}^n\left(x_i\right)^2}}{\sqrt{t}}
\end{align}
for $t \geq \sum_{i=1}^n\left(x_i\right)^2$.

In this way, we obtain the following theorem.

\begin{THM} \label{THM:TailExpectationBound}
There exists a constant $C>0$ (independent of $\omega$) such that
\begin{align} \label{ineq.TailExpectationBound}
\E\left(\frac{\vert A_{T_1}\vert}{t}\wedge 1\right) & \leq C \lp \frac{\lp \sum_{i=1}^n \lp \alpha_i\rp^2\rp^{\frac{1}{4}} \cdot \lp \sum_{i=1}^n \lp x_i\rp^2\rp^{\frac{1}{2}}}{\sqrt{t}} + \frac{\vert z-\widetilde{z} +\sum_{i=1}^n \alpha_i x_iy_i\vert}{t}\rp 
\end{align}
for any $t \geq 2 \max \left\{ 2\lp \sum_{i=1}^n \lp \alpha_i\rp^2\rp^{\frac{1}{2}} \lp\sum_{i=1}^n  \lp x_i\rp^2\rp, \vert z-\widetilde{z} +\sum_{i=1}^n \alpha_i x_iy_i\vert \right\}$.
\end{THM}

\begin{proof}

In this case, we have $A_t-A_0=a \int_0^t W_s^2dW_s^1$ with  $a=\frac{2\sqrt{\sum_{i=1}^n\left(\alpha_i x_i\right)^2}}{L}$ where $W_t^1=\frac{\sum_{i=1}^n\alpha_i x_iy_t^i}{\sqrt{\sum_{i=1}^n\left(\alpha_i x_i\right)^2}}$ and $W_t^2=\frac{\sum_{j=1}^n x_jx_t^j}{L}$ are two independent Brownian motions started from $\frac{\sum_{i=1}^n \alpha_i x_iy_i}{\sqrt{\sum_{i=1}^n\left(\alpha_i x_i\right)^2}}$ and $L$ respectively. The idea of the proof of \cite[Theorem 3.5]{BGM} still applies here and we can use it to obtain the desired result. For completeness, we still include the proof here. 

For any $t>0$, we have
\begin{align}
\E\left(\frac{\vert A_{T_1}\vert}{t}\wedge 1\right)=\Prob\left(\vert A_{T_1}\vert>t\right)+\E\left(\frac{\vert A_{T_1}\vert}{t}\chi_{\{\vert A_{T_1}\vert \leq t\}}\right).
\end{align}
We estimate two terms in the right side of the above equality separately.

For $t \geq \max\{2\vert A_0\vert,2a\lp W_0^2\rp^2\}$, applying \cite[Lemma 3.3]{BGM}, we have
\begin{align*}
\Prob\left(\vert A_{T_1}\vert>t\right) & \leq \Prob \lp \vert A_{T_1}-A_0\vert >\frac{t}{2}\rp
\\
&
=\Prob \lp \left\vert a\int_0^{T_1} W_s^2dW_s^1 \right\vert >\frac{t}{2} \rp
\\
&
\leq 2W_0^2 \sqrt{\frac{2a}{t}}.
\end{align*}

Now we estimate the second term $\E\left(\frac{\vert A_{T_1}\vert}{t}\chi_{\{\vert A_{T_1}\vert \leq t\}}\right)$. Let $k_0$ be the smallest $k$ such that $2^{-k-1}t \leq \max\{2\vert A_0\vert,2a\lp W_0^2\rp^2\}$. Then
\begin{align*}
\E\left(\frac{\vert A_{T_1}\vert}{t}\chi_{\{\vert A_{T_1}\vert \leq t\}}\right) & \leq \sum_{k=0}^{\infty} 2^{-k} \Prob\left(2^{-k-1}t < \vert A_{T_1}\vert \leq 2^{-k}t\right)
\\
&
\leq \sum_{k=0}^{\infty} 2^{-k} \Prob\left( \vert A_{T_1}\vert \geq 2^{-k-1}t\right) 
\\
&
=\sum_{k=0}^{k_0} 2^{-k} \Prob\left( \vert A_{T_1}\vert \geq 2^{-k-1}t\right) + \sum_{k=k_0+1}^{\infty} 2^{-k} \Prob\left( \vert A_{T_1}\vert \geq 2^{-k-1}t\right). 
\end{align*}
Next we estimate the above two terms separately.

For the first term $\sum_{k=0}^{k_0} 2^{-k} \Prob\left( \vert A_{T_1}\vert \geq 2^{-k-1}t\right)$, we can apply \cite[Lemma 3.3]{BGM} again and obtain
\begin{align*}
\sum_{k=0}^{k_0} 2^{-k} \Prob\left( \vert A_{T_1}\vert \geq 2^{-k-1}t\right) & \leq \sum_{k=0}^{k_0} 2\cdot 2 ^{-\frac{k}{2}} W_0^2 \sqrt{\frac{a}{t}}
\\
&
\leq \frac{CW_0^2 \sqrt{a}}{\sqrt{t}}
\end{align*}
where $C$ is a constant independent of $A_0$, $a$ and $\omega$.

For the second term $\sum_{k=k_0+1}^{\infty} 2^{-k} \Prob\left( \vert A_{T_1}\vert \geq 2^{-k-1}t\right)$, we have
\begin{align*}
\sum_{k=k_0+1}^{\infty} 2^{-k} \Prob\left( \vert A_{T_1}\vert \geq 2^{-k-1}t\right) & \leq 2^{-k_0}
\\
&
\leq \frac{2\max\{2\vert A_0\vert,2a\lp W_0^2\rp^2\}}{t}
\\
&
\leq \frac{4}{t} \lp \vert A_0\vert+a\lp W_0^2\rp^2\rp.
\end{align*}
For $t \geq \max\{2\vert A_0\vert,2a\lp W_0^2\rp^2\}$, we have $\sqrt{t} \geq \sqrt{2a} W_0^2$, so $\frac{W_0^2 \sqrt{a}}{\sqrt{t}} \geq \sqrt{2} \frac{\lp W_0^2\rp^2 a}{t}$. Then
\begin{align*}
\sum_{k=k_0+1}^{\infty} 2^{-k} \Prob\left( \vert A_{T_1}\vert \geq 2^{-k-1}t\right) \leq C\lp \frac{\vert A_0\vert}{t}+\frac{W_0^2 \sqrt{a}}{\sqrt{t}}\rp 
\end{align*}
where $C$ is a constant independent of $A_0$, $a$ and $\omega$.

Altogether, using $a=\frac{2\sqrt{\sum_{i=1}^n\left(\alpha_i x_i\right)^2}}{L} \leq 2 \sqrt{\sum_{i=1}^n\left(\alpha_i \right)^2}$, we can obtain 
\begin{align*}
\E\left(\frac{\vert A_{T_1}\vert}{t}\wedge 1\right) & \leq C \lp \frac{\vert A_0\vert}{t}+\frac{W_0^2 \sqrt{a}}{\sqrt{t}}\rp
\\
&
\leq C \lp \frac{\lp \sum_{i=1}^n \lp \alpha_i\rp^2\rp^{\frac{1}{4}} \cdot \lp \sum_{i=1}^n \lp x_i\rp^2\rp^{\frac{1}{2}}}{\sqrt{t}} + \frac{\vert z-\widetilde{z} +\sum_{i=1}^n \alpha_i x_iy_i\vert}{t}\rp
\end{align*}
for any $t \geq 2 \max \left\{ 2\lp \sum_{i=1}^n \lp \alpha_i\rp^2\rp^{\frac{1}{2}} \lp\sum_{i=1}^n  \lp x_i\rp^2\rp, \vert z-\widetilde{z} +\sum_{i=1}^n \alpha_i x_iy_i\vert \right\} \geq \max\{2\vert A_0\vert,2a\lp W_0^2\rp^2\}$, which is the desired result.
\end{proof}

\subsection{Vertical coupling}

We couple the vertical processes $z_t$ and $\widetilde{z}_t$.   

\begin{notation}
On each copy of $\bR^2$ for all $i=1,\ldots,n$, we can take the cylindrical coordinates $\lp r_i,\theta_i\rp$. For any $g=\lp x_1,y_1,\ldots,x_n,y_n,z\rp \in \mathbb{H}^n_{\omega}$, we can take the cylindrical coordinates instead of the Cartesian coordinates on each copy of $\bR^2$ for all $i=1,\ldots,n$, and thus we can denote the element $g\in \mathbb{H}^n_{\omega}$ by $\lp r_1,\theta_1,\ldots,r_n,r_n,z\rp$.

Let $B_t=(x^1_t,y^1_t,\ldots,x^n_t,y^n_t,z_t)$ be an $\mathbb{H}^n_{\omega}$-Brownian motion. In this section, we take the cylindrical coordinates instead of the Cartesian coordinates on each copy of $\bR^2$ for all $i=1,\ldots,n$, and use $\lp r^1_t,\theta^1_t,\ldots,r^n_t,\theta^n_t,z_t\rp$ to denote the Brownian motion $B_t$.
\end{notation}

Let $B_t=(x^1_t,y^1_t,\ldots,x^n_t,y^n_t,z_t)$ be an $\mathbb{H}^n_{\omega}$-Brownian motion started at $(0,\ldots,0,0)$ and let $\sigma_a$ be the first time $z_t$ hits $a$ for each fixed $a\in \bR$. We first list some properties of $\sigma_a$.

\begin{Lemma}\label{Lem.HittingTime.FiniteAS}
The hitting time $\sigma_a$ is finite a.s.
\end{Lemma}

\begin{proof}
We have $\int_{0}^t \sum_{i=1}^n\lp \alpha_i r^i_s\rp^2 ds=\sum_{i=1}^n \lp\alpha_i\rp^2 \lp\int_{0}^t \lp r^i_s\rp^2 ds\rp \to \infty$ a.s. as $t\to \infty$, since $\int_0^t r^2_s\, ds\rightarrow \infty$ a.s. as $t\to \infty$ for each $i=1,\cdots,n$. Using the fact that $z_t$ can be written as a time changed Brownian motion as in Lemma~\ref{Lem:AreaProcess.TimeChangeOfBM}, we can see that $z_t$ hits every real value in finite time, and thus $\sigma_a$ is finite a.s.
\end{proof}

\begin{lemma} \label{Lem:HittingTimeRotationInvariance.Nonisotropic}
For any $a\in \bR$ with $a\neq 0$, the hitting time $\sigma_a$ is invariant under the rotation on each copy of $\bR^2$ from the base space $\bR^{2n}\cong \bR^2 \times \ldots \bR^2$. Also, $\sigma_a$ is independent of $\theta^i_{\sigma_a}$.
\end{lemma}

\begin{proof}
Note that the L\'evy stochastic area process $z_t$ is invariant under the rotation on each copy of $\bR^2$ in the base space $\bR^{2n}\cong \bR^2 \times \ldots \bR^2$. For any $\alpha_i,\beta
_i\in [0,2\pi]$ and any $\gamma,\delta\in [0,2\pi]$, we have
\begin{align*}
\Prob \lp \gamma \leq \sigma_a \leq \delta, \alpha_i\leq \theta^i_{\sigma_a} \leq \beta_i, i=1,\cdots,n\rp= \int_{\gamma}^{\delta} \Prob \lp\left. \alpha_i\leq \theta^i_{\sigma_a} \leq \beta_i, i=1,\cdots,n \right\vert \sigma_a=c\rp dF_{\sigma_a}(c)
\end{align*}
where $F_{\sigma_a}$ is the probability density function of $\sigma_a$. Since $\lp \theta^1_{\sigma_a},\ldots,\theta^n_{\sigma_a}\rp$ is uniformly distributed on $[0,2\pi]\times \ldots \times [0,2\pi]$, we have
\begin{align*}
\Prob \lp \gamma \leq \sigma_a \leq \delta, \alpha_i\leq \theta^i_{\sigma_a} \leq \beta_i, i=1,\cdots,n\rp & =\int_{\gamma}^{\delta} \Pi_{i=1}^n \frac{\beta_i-\alpha_i}{2\pi}dF_{\sigma_a}(c)
\\
&
=\Prob \lp \gamma \leq \sigma_a \leq \delta\rp \cdot \Prob \lp \alpha_i\leq \theta^i_{\sigma_a} \leq \beta_i, i=1,\cdots,n\rp.
\end{align*}
Therefore, the desired result follows.
\end{proof}

Taking advantage of left-invariance, we suppose that the vertical processes $z_t$ and $\widetilde{z}_t$ of two Brownian motions start from $(0,\ldots,0,-a)$ and $(0,\ldots,0,a)$ for $a\neq 0$. To construct a coupling of two such horizontal Brownian motions, we start with the following lemma.  

\begin{Lemma}\label{Lem:BM.Nonisotropic}
Let $B_t=(r^1_t,\theta^1_t,\ldots,r^n_t,\theta^n_t,z_t)$ be an $\mathbb{H}^n_{\omega}$-Brownian motion started from the identity where $\lp r^1_t,\theta^1_t,\ldots,r^n_t,\theta^n_t \rp$ is a Brownian motion on $\bR^{2n}$ started from the origin and $\lp r^i_{t},\theta^i_{t}\rp$ for $i=1,\cdots,n$ are $n$ independent Brownian motions (expressed in the cylindrical coordinates) on $\bR^2$. For any point $(r_i,\theta_i)\in \bR^{2}$ and $a\neq 0$, let $\sigma_a$ be the first time that the vertical process $z_t$ hits $a$ and let $R_{(r_i,\theta_i)}:\bR^2 \rightarrow \bR^2$ be the reflection map as in Lemma~\ref{Lem:BM}. Then the process 
\[
\lp \widetilde{r}^1_{t},\widetilde{\theta}^1_{t},\ldots, \widetilde{r}^n_{t},\widetilde{\theta}^n_{t} \rp = \begin{cases}
\lp R_{\lp r^1_{\sigma_a},\theta^1_{\sigma_a}\rp}\lp r^1_{t},\theta^1_{t}\rp, \ldots, R_{\lp r^n_{\sigma_a},\theta^n_{\sigma_a}\rp}\lp r^n_{t},\theta^n_{t}\rp \rp
 & \text{for $t\leq \sigma_a$} \\
\lp r^1_{t},\theta^1_{t},\ldots, r^n_{t},\theta^n_{t} \rp  & \text{for $t> \sigma_a$}
\end{cases}
\]
is a Brownian motion on $\bR^{2n}$ started from the origin.
\end{Lemma}

\begin{proof}
The well-definedness of $\lp \widetilde{r}^1_{t},\widetilde{\theta}^1_{t},\ldots, \widetilde{r}^n_{t},\widetilde{\theta}^n_{t}\rp$ is guaranteed by the a.s. finiteness of $\sigma_a$ from Lemma~\ref{Lem.HittingTime.FiniteAS}. Now we show that $\lp\widetilde{r}^1_{t},\widetilde{\theta}^1_{t},\ldots, \widetilde{r}^n_{t},\widetilde{\theta}^n_{t}\rp$ is a Brownian motion on $\bR^{2n}$. It suffices to show that such a process has the same finite-dimensional distribution as a Brownian motion on $\bR^{2n}$. 

Note that each random reflection $R_{\lp r^i_{\sigma_a},\theta^i_{\sigma_a}\rp}:\bR^2 \rightarrow \bR^2$ only depends on $\theta^i_{\sigma_a}$ by basic planar geometry, even though it is first constructed using the point $\lp r^i_{\sigma_a},\theta^i_{\sigma_a}\rp$. In this way, we have $R_{\lp r^i_{\sigma_a},\theta^i_{\sigma_a}\rp}=R_{\lp 1,\theta^i_{\sigma_a}\rp}$ for each $i=1,\ldots,n$. We denote $R_{\lp r^i_{\sigma_a},\theta^i_{\sigma_a}\rp}\lp r^i_t,\theta^i_t\rp$ by $\hat{X}^i_t$.

Consider any finite collection of times $0< t_1<t_2<\cdots <t_m<\infty$ and any sequence of Borel subsets $A_{ij}$ of $\bR^{2}$ for $i=1,\ldots,n$ and $j=1,\ldots,m$. Since the product $\sigma-$algebra is generated by the collection of measurable rectangles, we only need to look at the event $\lp\lp \hat{X}^1_{t_1},\ldots,\hat{X}^1_{t_m}\rp, \ldots, \lp \hat{X}^n_{t_1},\ldots,\hat{X}^n_{t_m}\rp\rp \in \lp A_{11}\times\cdots\times A_{1m}\rp \times \ldots \times \lp A_{n1}\times\cdots\times A_{nm}\rp$ here. Furthermore, from the a.s. finiteness of $\sigma_a$, a.e.\ path of $\hat{B}^i_t$ is given by a reflected path of $B^i_t$ for $i=1,\ldots,n$. We know that $\sigma_a$  and $\theta_{\sigma}^i$ for all $i=1,\ldots,n$ are independent (Lemma~\ref{Lem:HittingTimeRotationInvariance.Nonisotropic}), and $\lp \theta^1_{\sigma_a},\ldots,\theta^n_{\sigma_a}\rp$ is uniformly distributed on $[0,2\pi]\times \ldots \times [0,2\pi]$. Thus, we can decompose the event that $\lp\lp \hat{X}^1_{t_1},\ldots,\hat{X}^1_{t_m}\rp, \ldots, \lp \hat{X}^n_{t_1},\ldots,\hat{X}^n_{t_m}\rp\rp \in \lp A_{11}\times\cdots\times A_{1m}\rp \times \ldots \times \lp A_{n1}\times\cdots\times A_{nm}\rp$ by conditioning on $\theta^i_{\sigma_a}$ as
\begin{align*}
& \Prob \lp \hat{X}^i_{t_j}\in A_{ij},i=1,\ldots,n,j=1,\ldots,m \rp 
\\
&
=\frac{1}{(2\pi)^n} \int_{0}^{2\pi} \ldots \int_{0}^{2\pi} \Prob \lp \left. R_{(1,\tau_i)} \hat{X}^i_{t_j}\in A_{ij}, i=1,\ldots,n,j=1,\ldots,m \, \right| \, 
\theta^i_{\sigma_a}=\tau_i , i=1,\ldots,n\rp \, d\tau_1\ldots d\tau_n.
\end{align*}

Now we study the above integral. Let $\Rot_{\tau}$ denote rotation in $\bR^2$ around the origin by the angle $\tau$. Here we can proceed as in the proof of Lemma \ref{Lem:BM}. Using basic planar geometry and taking conditional probability, we have
\begin{align*}
& \Prob \lp \left. R_{(1,\tau_i)} \hat{X}^i_{t_j}\in A_{ij}, i=1,\ldots,n,j=1,\ldots,m \, \right| \, 
\theta^i_{\sigma_a}=\tau_i , i=1,\ldots,n\rp 
\\
&
=\Prob \lp \left. \Rot_{\tau_i} \lp R_{(1,0)} \lp r^i_{t_j},\theta^i_{t_j}\rp\rp\in A_{ij},i=1,\ldots,n,j=1,\ldots,m \right| \, 
\theta^i_{\sigma_a}=0,i=1,\ldots,n \rp.
\end{align*}
Using the change of variable, we have
\begin{align*}
& \Prob \lp \hat{X}^i_{t_j}\in A_{ij},i=1,\ldots,n,j=1,\ldots,m \rp
\\
&
=\frac{1}{(2\pi)^n} \int_{0}^{2\pi} \ldots \int_{0}^{2\pi} \Prob \lp \left. \Rot_{\tau_i} \lp R_{(1,0)} \lp r^i_{t_j},\theta^i_{t_j}\rp\rp\in A_{ij},i=1,\ldots,n,j=1,\ldots,m \right| \, 
\theta^i_{\sigma_a}=0,i=1,\ldots,n \rp d\tau_1\ldots d\tau_n.
\end{align*}
We see that $\{R_{(1,0)} \lp r^i_{t},\theta^i_{t}\rp\}_{i=1}^n$ is a family of i.i.d.\ Brownian motions on $\bR^{2}$. For each $i=1,\ldots,n$, we denote $R_{(1,0)} \lp r^i_{t},\theta^i_{t}\rp$ by $\widetilde{X}^i_{t}$ and the hitting time of the corresponding L\'evy area process associated with $\lp \widetilde{X}^1_{t},\cdots,\widetilde{X}^n_{t}\rp$ to the point $a\in \bR$ by $\widetilde{\sigma}_a$. For each $a\in \bR$, we have $\sigma_a=\widetilde{\sigma}_{-a}$ using the fact that reflection on each copy of $\bR^2$ reverses the sign of the L\'evy area process. This gives
\begin{align*}
& \Prob \lp \left. \Rot_{\tau_i} \lp R_{(1,0)} \lp r^i_{t_j},\theta^i_{t_j}\rp\rp\in A_{ij},i=1,\ldots,n,j=1,\ldots,m \right| \, 
\theta^i_{\sigma_a}=0,i=1,\ldots,n \rp 
\\
&
=\Prob \lp \left. \Rot_{\tau_i} \widetilde{X}^i_{t_j} \in A_{ij},i=1,\ldots,n,j=1,\ldots,m \right| \, 
\theta^i_{\widetilde{\sigma}_{-a}}=0,i=1,\ldots,n \rp.
\end{align*}
Note that $\widetilde{\sigma}_{-a}$ is also invariant under rotation on each copy of $\bR^2$, which can be implied by the fact that $\sigma_a=\widetilde{\sigma}_{-a}$ and Lemma \ref{Lem:HittingTimeRotationInvariance.Nonisotropic}. This gives
\begin{align*}
&\Prob \lp \left. \Rot_{\tau_i} \widetilde{X}^i_{t_j} \in A_{ij},i=1,\ldots,n,j=1,\ldots,m \right| \, 
\theta^i_{\widetilde{\sigma}_{-a}}=0,i=1,\ldots,n \rp
\\
&
=\Prob \lp \left.  \widetilde{X}^i_{t_j} \in A_{ij},i=1,\ldots,n,j=1,\ldots,m \right| \, 
\theta^i_{\widetilde{\sigma}_{-a}}=\tau_i,i=1,\ldots,n \rp. 
\end{align*}
Altogether, we obtain
\begin{align*}
& \Prob \lp \hat{X}^i_{t_j}\in A_{ij},i=1,\ldots,n,j=1,\ldots,m \rp 
\\
&
=\frac{1}{(2\pi)^n} \int_{0}^{2\pi} \ldots \int_{0}^{2\pi} \Prob \lp \left.  \widetilde{X}^i_{t_j} \in A_{ij},i=1,\ldots,n,j=1,\ldots,m \right| \, 
\theta^i_{\widetilde{\sigma}_{-a}}=\tau_i,i=1,\ldots,n \rp d\tau_1\ldots d\tau_n.
\end{align*}
This means that $\lp \hat{X}^1_{t},\ldots,\hat{X}^n_{t}\rp$ has the same finite-dimensional distribution as a Brownian motion $\lp \widetilde{X}^1_{t},\ldots,\widetilde{X}^n_{t}\rp$ on $\bR^{2n}$. Therefore, $\lp \hat{X}^1_{t},\ldots,\hat{X}^n_{t}\rp$ is a Brownian motion on $\bR^{2n}$.
\end{proof}

Now we construct a coupling of two Brownian motions $B(t)=(r^1_t,\theta^1_t,\ldots,r^n_t,\theta^n_t,z_t)$ and $\widetilde{B}(t)=(\widetilde{r}^1_{t},\widetilde{\theta}^1_{t},\ldots, \widetilde{r}^n_{t},\widetilde{\theta}^n_{t},\widetilde{z}_t)$ started from two points $(r^1,\theta^1,\ldots,r^n,\theta^n,z)$ and $(\widetilde{r}_1,\widetilde{\theta}_1,\ldots,\widetilde{r}_n,\widetilde{\theta}_n,\widetilde{z})$ in cylindrical coordinates respectively. Moreover, we connect the coupling time with the hitting time $\sigma_a$. This gives a way to estimate the tail bound for the coupling time by studying the hitting time $\sigma_a$.

\begin{THM}\label{THM:HMax.Nonisotropic}
Consider two points in $\bH^n_{\omega}$ on the same vertical fiber with vertical separation $2a>0$; then there exist coordinates such that the points are $(0,\ldots, 0,0)$ and $(0,\ldots, 0,2a)$ (after possibly switching them). Let $B_t=(r^1_t,\theta^1_t,\ldots,r^n_t,\theta^n_t,z_t)$ be an $\bH^n_{\omega}$-Brownian motion started from the origin, and let $\sigma_a$ be the first hitting time of the set $\{z=a\}$ by the process $z_t$, which is a.s.\ finite. Define a stochastic process $\widetilde{B}_t=(\widetilde{r}^1_{t},\widetilde{\theta}^1_{t},\ldots, \widetilde{r}^n_{t},\widetilde{\theta}^n_{t},\widetilde{z}_t)$ on $\bH^n_{\omega}$ by
\begin{equation}\label{Eqn:VerticalCoupling.Nonisotropic}
\lp\widetilde{r}^1_{t},\widetilde{\theta}^1_{t},\ldots, \widetilde{r}^n_{t},\widetilde{\theta}^n_{t},\widetilde{z}_t\rp = \begin{cases}
\lp R_{\lp r^1_{\sigma_a},\theta^1_{\sigma_a}\rp}\lp r^1_{t},\theta^1_{t}\rp, \ldots, R_{\lp r^n_{\sigma_a},\theta^n_{\sigma_a}\rp}\lp r^n_{t},\theta^n_{t}\rp, 2a-z_t\rp  & \text{for $t\leq \sigma_a$} \\
(r^1_t,\theta^1_t,\ldots,r^n_t,\theta^n_t,z_t)  & \text{for $t> \sigma_a$}
\end{cases}
\end{equation}
(where $R_{(r^i_{\sigma_a},\theta^i_{\sigma_a})}:\bR^2 \rightarrow \bR^2$ is the reflection map in $\bR^2$ as in Lemma \ref{Lem:BM}). Then $\tilde{B}_t$ is a $\bH^n_{\omega}$-Brownian motion started from $(0,\ldots, 0,2a)$, coupled with $B_t$ such that their coupling time $T_2$ is equal to $\sigma_a$. Moreover, this coupling is maximal, and we have the following reflection principle
\begin{align} \label{eqn.CouplingTimeTail.Nonisotropic}
\Prob\lp T_2> t\rp=\Prob\lp \sigma_a> t\rp = 1-2 \Prob\lp z_t\geq a \rp .
\end{align}
\end{THM}

\begin{proof}
We can proceed in a similar as in the proof of Theorem~\ref{THM:HMax}.

First, we show that $\widetilde{B}_t$ is a well-defined Brownian motion on $\bH^n_{\omega}$. We know that $r^i_t$ is positive for all $t>0$ with probability 1, $r^i_{\sigma_a}>0$ a.s. for all $i=1,\ldots,n$, and thus the reflection in \eqref{Eqn:VerticalCoupling.Nonisotropic} and $\widetilde{B}_t$ are well-defined. It follows from Lemma~\ref{Lem:BM.Nonisotropic} that $\lp R_{\lp r^1_{\sigma_a},\theta^1_{\sigma_a}\rp}\lp r^1_{t},\theta^1_{t}\rp, \ldots, R_{\lp r^n_{\sigma_a},\theta^n_{\sigma_a}\rp}\lp r^n_{t},\theta^n_{t}\rp \rp$ is an $\bR^{2n}$-Brownian motion. Further, using the fact that reflection in $\bR$ reverses the sign of the L\'evy area , we see that each $\tilde{z}^i_t=-z^i_t$ is the L\'evy area process associated with $(\tilde{r}^i_t,\tilde{\theta}^i_t)$ for each $i=1,\ldots,n$. Since $\tilde{z}_t=2a-z_t=2a-\sum_{i=1}^n\alpha_iz^i_t=2a-\sum_{i=1}^n\alpha_i\tilde{z}^i_t$, so that $\widetilde{B}_t$ is a horizontal Brownian motion started from $(0,\ldots,0,2a)$ on $\bH^{n}_{\omega}$. 

Now, from the construction $\widetilde{B}_t$, we see that $B_t$ and $\widetilde{B}_t$ meet at time $\sigma_a$ and not before (note that $z_t<a<\tilde{z}_t$ for $t<\sigma_a$). So $B_t$ and $\widetilde{B}_t$ are coupled horizontal Brownian motions on $\bH^n_{\omega}$ with coupling time $\sigma_a$.

Next, we show that $z_t$ satisfies a reflection principle. Note that if the symmetry of $z_t$ after $\sigma_a$ is given, then the reflection principle follows by the same argument as in the case of one-dimensional Brownian motion on $\bR$. Thus, it suffices to show that the evolution of $z_t$ is symmetric after $\sigma_a$; that is, $z_{\sigma_a+s}-a$ and $-(z_{\sigma_a+s}-a)$ have the same distribution for all $s>0$. Observe that $(r^1_t,\ldots,r^n_t,z_t)$ is diffusion in its own right and satisfies the strong Markov property. Thus, for any $A\subset \mathcal{B}\lp\bR\rp$, we have
\begin{equation}\begin{split}
\Prob\lp z_{\sigma_a+s} \in A \rp &= \E\lb\E\lb \Ind_{\{z_{\sigma_a+s}\in A\}}  | r^i_{\sigma_a},i=1,\ldots,n \rb\rb \\
&=  \E\lb\Prob\lp  \lp a+ \frac{1}{2} W_{\int_0^s \frac{1}{4}\sum_{i=1}^n \lp\alpha_iR^i_u\rp^2\, du}\rp \in A \, \big|\, R^i_0, i=1,\ldots,n \rp\rb
\end{split}\end{equation}
where $W_t$ is a standard one-dimensional Brownian motion started from 0, $R^i_t$ is a two-dimensional Bessel process on the $i$th copy of $\bR^2$ started from $R^i_0$, and this last expectation is understood with respect to the distribution of $R^i_0=r^i_{\sigma_a}$ for $i=1,\ldots,n$. But then it's clear that the probability on this last line is reflection symmetric, in the sense that
\[
\Prob\lp  \lp a+ \frac{1}{2} W_{\int_0^s  \frac{1}{4}\sum_{i=1}^n \lp\alpha_iR^i_u\rp^2\, du}\rp \in A \rp =
\Prob\lp  \lp a+ \frac{1}{2} W_{\int_0^s  \frac{1}{4}\sum_{i=1}^n \lp\alpha_iR^i_u\rp^2\, du}\rp \in a-A \rp .
\]
Then this symmetry is preserved after taking the outer expectation, and therefore the desired symmetry of $z_t$ follows.

Finally, the argument for the maximality of the coupling and \eqref{eqn.CouplingTimeTail.Nonisotropic} is similar as in the proof of Theorem~\ref{THM:HMax}.
\end{proof}

\begin{proposition} \label{Prop.VerticalCouplingTime.TailBound}
The coupling time $T_2$ satisfies
\begin{align*}
\Prob\lp T_2>t\rp=\Prob\lp \sigma_a>t\rp \leq \frac{2a}{\alpha_n t}.
\end{align*}
\end{proposition}

\begin{proof}
Theorem~\ref{THM:HMax.Nonisotropic} gives that $\Prob\lp T_2>t\rp=\Prob\lp \sigma_a>t\rp= 1-2 \Prob\lp z_t\geq a \rp$. By Lemma~\ref{Lem:eqn.DensityUpperBound2} and the symmetry of $f_t^{\omega}$ around $0$, we have
\[\begin{split}
\Prob\lp \sigma_a>t\rp &= 1- 2\int_a^{\infty} f_t^{\omega}(z) \, dz \\
& = 2\int_0^a f_t^{\omega}(z) \, dz \leq \frac{2a}{\alpha_n t}.
\end{split}\]
\end{proof}

In this way, we have the following vertical gradient estimate in a similar way as Theorem~\ref{THM:HVGrad}.

\begin{THM}
Let $P_t^{\omega}=e^{\frac{t}{2}\Lap_{\sH}^{\omega}}$ be the heat semigroup on $\bH^n_{\omega}$ and consider $f\in L^{\infty}(\bH^n_{\omega})$. Then at any point $(x,y,z)\in\bH^n_{\omega}$ and for ant time $t>0$, we have
\[
\lab Z P_t f  \rab = \lab \nabla_{V} P_t f \rab \leq \frac{1}{\alpha_nt}\|f\|_{\infty} .
\]
\end{THM}

\subsection{Coupling time estimate and gradient estimate}

\begin{THM} \label{THM:Coupling.Nonisotropic}
There exists a non-Markovian coupling $\lp B_t,\widetilde{B}_t\rp$ of two $\mathbb{H}^n_{\omega}$-Brownian motions started from $\lp x_1,y_1,\ldots,x_n,y_n,z\rp$ and $\lp \widetilde{x}_1,\widetilde{y}_1,\ldots,\widetilde{x}_n,\widetilde{y}_n,\widetilde{z}\rp$ respectively with $\widetilde{x}_i=-x_i$ and $\widetilde{y}_i=y_i$ for $i=1,\cdots,n$, and a positive constant $C$ (independent of $\omega$) such that the coupling time $\tau$ satisfies
\begin{align} \label{ineq.CouplingTime.TailBound}
\Prob \lp \tau>t\rp & \leq C \left[\lp \frac{\lp \sum_{i=1}^n \lp \alpha_i\rp^2\rp^{\frac{1}{4}}}{\alpha_n}+1\rp \cdot \frac{\lp \sum_{i=1}^n \lp x_i\rp^2\rp^{\frac{1}{2}}}{\sqrt{t}} + \frac{\vert z-\widetilde{z} +\sum_{i=1}^n \alpha_i x_iy_i\vert}{\alpha_n t}\right] 
\end{align}
for $t\geq  \max \left\{ 4\lp \sum_{i=1}^n \lp \alpha_i\rp^2\rp^{\frac{1}{2}} \lp\sum_{i=1}^n  \lp x_i\rp^2\rp, \sum_{i=1}^n \lp x_i\rp^2, 2\vert z-\widetilde{z} +\sum_{i=1}^n \alpha_i x_iy_i\vert \right\}$. 
\end{THM}

\begin{proof}
Since the L\'evy stochastic area is invariant under rotation of coordinates on each copy of $\bR^2$, it suffices to consider the case when $x_i=\widetilde{x}_i$ for all $i=1,\ldots,n$.

First, we couple their ``horizontal'' part, that is, we couple $(x^1_t,y^1_t,\ldots,x^n_t,y^n_t)$ and $(\widetilde{x}^1_t,\widetilde{y}^1_t,\ldots,\widetilde{x}^n_t,\widetilde{y}^n_t)$ on $\bR^{2n}\cong \bR^2 \times \ldots \times \bR^2$. Here, we take a mirror coupling of $(x^1_t,y^1_t,\ldots,x^n_t,y^n_t)$ and $(\widetilde{x}^1_t,\widetilde{y}^1_t,\ldots,\widetilde{x}^n_t,\widetilde{y}^n_t)$ on $\bR^{2n}$ in the sense of Kendall and Cranston. We also have the tail estimate for the horizontal coupling time $T_1$ given by \eqref{ineq.HorizontalCouplingTime.TailBound}. 

Next, after time $T_1$, we apply the coupling strategy described in Theorem~\ref{THM:HMax.Nonisotropic} to these two diffusions 
\[\lp x^1_t,y^1_t,\ldots,x^n_t,y^n_t,A_{T_1}+\frac{1}{2}\sum_{i=1}^n \alpha_i \int_0^t\lp x^i_sdy^i_s-y^i_sdx^i_s\rp \rp\] 
and 
\[\lp \widetilde{x}^1_t,\widetilde{y}^1_t,\ldots,\widetilde{x}^n_t,\widetilde{y}^n_t, \frac{1}{2}\sum_{i=1}^n \alpha_i \int_0^t\lp\widetilde{x}^i_sd\widetilde{y}^i_s-\widetilde{y}^i_sd\widetilde{x}^i_s\rp \rp.\]
By Lemma~\ref{THM:TailExpectationBound} and Proposition~\ref{Prop.VerticalCouplingTime.TailBound}, we have
\begin{align*}
\Prob\lp \tau-T_1>t\rp & \leq \frac{2}{\alpha_n} \E\left(\frac{\vert A_{T_1}\vert}{t}\wedge 1\right)
\\
&
\leq \frac{C}{\alpha_n} \lp \frac{\lp \sum_{i=1}^n \lp \alpha_i\rp^2\rp^{\frac{1}{4}} \cdot \lp \sum_{i=1}^n \lp x_i\rp^2\rp^{\frac{1}{2}}}{\sqrt{t}} + \frac{\vert z-\widetilde{z} +\sum_{i=1}^n \alpha_i x_iy_i\vert}{t}\rp
\end{align*}
for $t \geq \max \left\{ 4\lp \sum_{i=1}^n \lp \alpha_i\rp^2\rp^{\frac{1}{2}} \lp\sum_{i=1}^n  \lp x_i\rp^2\rp, \sum_{i=1}^n \lp x_i\rp^2, 2\vert z-\widetilde{z} +\sum_{i=1}^n \alpha_i x_iy_i\vert \right\}$. Altogether, we obtain  the desired tail bound on the coupling
time probability stated in the theorem.
\end{proof}

\begin{proposition} \label{prop.CouplingTimeTailBound.Nonisotropic}
Let $\lp x_1,y_1,\ldots,x_n,y_n,z\rp$ and $(\widetilde{x}_1,\widetilde{y}_1,\ldots,\widetilde{x}_n,\widetilde{y}_n,\widetilde{z})$ be two points in $\mathbb{H}^n_{\omega}$ with $\widetilde{x}_i=-x_i$ and $\widetilde{y}_i=y_i$ for $i=1,\cdots,n$ such that $\vert z-\widetilde{z} +\sum_{i=1}^n \alpha_i x_iy_i \vert<1$. Then there exists a positive constant $C$ such that 
\begin{align*}
\Prob \lp \tau>t\rp \leq C \dist_{sR} 
\lp \lp x_1,y_1,\ldots,x_n,y_n,z\rp,\lp \widetilde{x}_1,\widetilde{y}_1,\ldots,\widetilde{x}_n,\widetilde{y}_n,\widetilde{z}\rp \rp
\end{align*}
for $t \geq \max \left\{ 4\lp \sum_{i=1}^n \lp \alpha_i\rp^2\rp^{\frac{1}{2}} \lp\sum_{i=1}^n  \lp x_i\rp^2\rp, \sum_{i=1}^n \lp x_i\rp^2, 2\vert z-\widetilde{z} +\sum_{i=1}^n \alpha_i x_iy_i\vert \right\}$.
\end{proposition}

\begin{proof}
Since $t\geq 1$, we have $\sqrt{t} \leq t$, so Theorem~\ref{ineq.CouplingTime.TailBound} implies
\begin{align*}
\Prob \lp \tau>t\rp & \leq C \left[\lp \frac{\lp \sum_{i=1}^n \lp \alpha_i\rp^2\rp^{\frac{1}{4}}}{\alpha_n}+1\rp \cdot \frac{\lp \sum_{i=1}^n \lp x_i\rp^2\rp^{\frac{1}{2}}}{\sqrt{t}} + \frac{\vert z-\widetilde{z} +\sum_{i=1}^n \alpha_i x_iy_i\vert}{\alpha_n t}\right] 
\\
&
\leq \frac{C}{\sqrt{t}} \lp  \lp \sum_{i=1}^n \lp x_i\rp^2\rp^{\frac{1}{2}}+\left \vert z-\widetilde{z} +\sum_{i=1}^n \alpha_i x_iy_i \right \vert\rp
\\
&
\leq \frac{C}{\sqrt{t}} \lp \sum_{i=1}^n \lp x_i\rp^2+\left \vert z-\widetilde{z} +\sum_{i=1}^n \alpha_i x_iy_i \right \vert\rp^{\frac{1}{2}}
\\
&
\leq \frac{C}{\sqrt{t}} \lp 4\sum_{i=1}^n \lp x_i\rp^2+\left \vert z-\widetilde{z} +\sum_{i=1}^n \alpha_i x_iy_i \right \vert\rp^{\frac{1}{2}}
\\
&
=\frac{C}{\sqrt{t}} \Vert (\widetilde{x}_1,\widetilde{y}_1,\ldots,\widetilde{x}_n,\widetilde{y}_n,\widetilde{z})^{-1} \circ \lp x_1,y_1,\ldots,x_n,y_n,z\rp \Vert_{\omega}.
\end{align*}
The third inequality is obtained by using $x+y \leq \sqrt{2}\lp x^2+y\rp^{\frac{1}{2}}$ for any $x\geq 0$ and $0 \leq y \leq 1$. By Proposition~\ref{prop.HomogeneousNorm} and the equivalence of the homogeneous norms (see \cite[Proposition 5.1.4]{BonfiglioliLanconelliUguzzoniBook}), we have
\begin{align*}
& \Vert (\widetilde{x}_1,\widetilde{y}_1,\ldots,\widetilde{x}_n,\widetilde{y}_n,\widetilde{z})^{-1} \circ \lp x_1,y_1,\ldots,x_n,y_n,z\rp \Vert_{\omega} 
\\
& 
\leq C \dist_{\sR}
\lp (\widetilde{x}_1,\widetilde{y}_1,\ldots,\widetilde{x}_n,\widetilde{y}_n,\widetilde{z})^{-1} \circ \lp x_1,y_1,\ldots,x_n,y_n,z\rp, e\rp
\\
&
=C\dist_{\sR}
\lp \lp x_1,y_1,\ldots,x_n,y_n,z\rp,(\widetilde{x}_1,\widetilde{y}_1,\ldots,\widetilde{x}_n,\widetilde{y}_n,\widetilde{z})\rp,
\end{align*}
so the desired result follows.
\end{proof}

Now we have the following gradient estimate.

\begin{THM}
Let $P_t^{\omega}$ be the heat semigroup generated by $\frac{1}{2}\Delta_{\mathcal{H}}^{\omega}$. For any bounded $f\in C^{\infty}(\mathbb{H}^n_{\omega})$, there exists a positive constant $C$ such that for any $t \geq 1$ we have
\begin{align} \label{ineq.GradientEstimate.Nonisotropic}
\Vert \nabla_{\mathcal{H}} P_t^{\omega}f \Vert_{\infty} \leq \frac{C}{\sqrt{t}}\Vert f\Vert_{\infty}.
\end{align}
Consequently, if $\Delta_{\mathcal{H}}^{\omega}f=0$, then $f$ is a constant.
\end{THM}

\begin{proof}
The idea of the proof of \cite[Corollary 3.10]{BGM} is applicable here. Fix $t \geq 1$. Without loss of generality, we take two points  $\lp x_1,y_1,\ldots,x_n,y_n,z\rp$ and $(\widetilde{x}_1,\widetilde{y}_1,\ldots,\widetilde{x}_n,\widetilde{y}_n,\widetilde{z})$ in $\mathbb{H}^n_{\omega}$ with $\widetilde{x}_i=-x_i$ and $\widetilde{y}_i=y_i$ for $i=1,\cdots,n$ such that they are close enough with respect to the sub-Riemannian distance $\dist_{sR} $ in the following way
\begin{align*}
\left \vert z_0-\widetilde{z}_0 +\sum_{i=1}^n \alpha_i x_iy_i \right \vert<1.
\end{align*}
Take the coupling $\lp B_t,\widetilde{B}_t\rp$ of two $\mathbb{H}^n_{\omega}$-Brownian motions constructed in Theorem~\ref{THM:Coupling.Nonisotropic}. By Proposition~\ref{prop.CouplingTimeTailBound.Nonisotropic}, we have
\begin{align*}
& \vert P_t^{\omega}f\lp x_1,y_1,\ldots,x_n,y_n,z\rp-P_t^{\omega}f(\widetilde{x}^1_t,\widetilde{y}^1_t,\ldots,\widetilde{x}^n_t,\widetilde{y}^n_t) \vert 
\\
&
=\vert \E(f(B_t))-\E(f(\widetilde{B}_t))\vert 
\\
&
=2\Vert f\Vert_{\infty} \Prob (\tau >t)
\\
&
\leq \frac{C}{\sqrt{t}} \Vert f\Vert_{\infty}
\dist_{sR} 
\lp \lp x_1,y_1,\ldots,x_n,y_n,z\rp,(\widetilde{x}_1,\widetilde{y}_1,\ldots,\widetilde{x}_n,\widetilde{y}_n,\widetilde{z})\rp.
\end{align*}
Dividing both sides by $\dist_{sR} 
\lp \lp x_1,y_1,\ldots,x_n,y_n,z\rp,(\widetilde{x}_1,\widetilde{y}_1,\ldots,\widetilde{x}_n,\widetilde{y}_n,\widetilde{z})\rp$ and taking a supremum over all points $(\widetilde{x}^1_t,\widetilde{y}^1_t,\ldots,\widetilde{x}^n_t,\widetilde{y}^n_t)$ which are different from $\lp x_1,y_1,\ldots,x_n,y_n,z\rp$,  we obtain \eqref{ineq.GradientEstimate.Nonisotropic}.

Moreover, if $\Delta_{\mathcal{H}}^{\omega}f=0$, then $P_t^{\omega}f=f$ for all $t\geq 0$. Taking $t \to \infty$ in \eqref{ineq.GradientEstimate.Nonisotropic}, we get $\nabla_{\mathcal{H}}f=0$, and hence $f\in C^{\infty}\lp \mathbb{H}^n_{\omega}\rp$ is a constant by \cite[Proposition 1.5.6]{BonfiglioliLanconelliUguzzoniBook}.
\end{proof}


\providecommand{\bysame}{\leavevmode\hbox to3em{\hrulefill}\thinspace}
\providecommand{\MR}{\relax\ifhmode\unskip\space\fi MR }
\providecommand{\MRhref}[2]{%
  \href{http://www.ams.org/mathscinet-getitem?mr=#1}{#2}
}
\providecommand{\href}[2]{#2}

\end{document}